\documentclass[10pt]{amsart}

\usepackage{color,graphicx,amssymb}
\usepackage{amsthm}
\usepackage{dsfont}
\usepackage{stmaryrd}

\theoremstyle{definition}

\newtheorem{theorem}{Theorem}[section]
\newtheorem{definition}[theorem]{Definition}
\newtheorem{lemma}[theorem]{Lemma}
\newtheorem{corollary}[theorem]{Corollary}
\newtheorem{proposition}[theorem]{Proposition}

\newtheorem{example}[theorem]{Example}

\newtheorem{remark}[theorem]{Remark}
\newtheorem{problem}[theorem]{Problem}
\newtheorem{acknowledgments}[theorem]{Acknowledgments}

\DeclareMathOperator{\Con}{Con}

\DeclareMathOperator{\Aut}{Aut}

\DeclareMathOperator{\Mod}{Mod}
\DeclareMathOperator{\Cg}{Cg}

\DeclareMathOperator{\Sg}{Sg}
\DeclareMathOperator{\Eqv}{Eqv}

\DeclareMathOperator{\id}{id}
\DeclareMathOperator{\Sub}{Sub}

\DeclareMathOperator{\ar}{ar}

\begin{document}

\title{Extensions realizing affine datum : low-dimensional cohomology}
\author{Alexander Wires}
\address{School of Mathematical Sciences, University of Electronic Science and Technology of China, Chengdu, 611731, Sichuan, PRC }
\email{awires81@uestc.edu.cn}
\date{October 20, 2020}

\begin{abstract}
For arbitrary varieties of universal algebras, we develop the theory around the first and second-cohomology groups characterizing extensions realizing affine datum. Restricted to varieties with a weak-difference term, extensions realizing affine datum are exactly extensions with abelian kernels. This recovers many classic examples of extensions with abelian coefficients since varieties with a weak-difference term give a far-reaching generalization of algebras like groups with multiple operators; indeed, any variety of algebras whose congruences form modular lattices. We introduce a notion of action and its model relation with a set of equations. In varieties with a difference term, central extensions are characterized by a property of their actions. Restricting further to a subclass of varieties with a difference term which still includes groups with multiple operators, we recover a special case of the representation of extensions with abelian kernels.
\end{abstract}


\maketitle


\section{Introduction}\label{sec:1}

Let $\Eqv A$ denote the set of equivalence relations on the set $A$.

\begin{definition}
The algebra $A$ is an \emph{extension} of the equivalence relation $\alpha \in \Eqv A$ by the algebra $Q$ if there is a surjective homomorphism $\pi: A \rightarrow Q$ such that $\ker \pi = \alpha$.
\end{definition}

Here we will be interested in the following general problem.

\begin{problem} (The Extension Problem)
Given the algebra $Q$ in the signature $\tau$ and an equivalence relation $\alpha \in \Eqv A$, classify the extensions of $\alpha$ by $Q$.
\end{problem}

As stated in the problem, the task is to understand the different interpretations of $\tau$ on the set $A$ such that $\alpha$ becomes a congruence $\alpha \in \Con A$ on the algebra $A$ and $A/\alpha \approx Q$. Note that we did not assume any additional structure on the equivalence relation $\alpha$ - this is almost certainly too general to begin with and we will require that $\alpha$ be presented with additional information of some sort; in particular, a partial structure related to abelianess in a given commutator theory. For universal algebras, a commutator can be a complicated object and there are several available. The term-condition commutator (or TC-commutator) \cite[Ch 2.5]{shape} and rectangular commutator \cite[Ch 5]{shape} would be natural choices to develop a cohomology theory since for large classes of varieties their abelian algebras are related to modules and semimodules, respectively. It would be interesting for future study to consider central extensions or extensions with abelian kernels in more general commutators, since it is possible to axiomatize commutator theories even in non-regular categories such as topological spaces and ordered sets \cite{wiresIV}.

In the present manuscript, we consider the deconstruction/reconstruction of extensions of universal algebras realizing affine datum. Instead of fixing a class of algebras with a particular equational theory, we elect to follow the approach of the Riemann integrable: that is, with abelian congruences in varieties with a weak-difference term as a motivating example, we isolate properties which will serve as a definition for a particular type of affine datum and then develop the standard $1^{\mathrm{st}}$ and $2^{\mathrm{nd}}$-cohomology constructions for this abstract definition of datum. Then one would prove that the extensions with abelian kernels in a particular class of algebras under consideration are captured by this notion; analogously, functions with countable jump discontinuities are Riemann integrable, etc. One benefit of this approach is that the equational theory (or varietal membership) of extensions is included in the $2^{\mathrm{nd}}$-cohomology group as a parameter which affords a Galois correspondence between the lattice of subvarieties and the subgroup lattice in cohomology. Formalizing this approach leads us to consider the 2-cocycles parametrizing extensions as a particular type of multisorted structure together with their equations. Future work might explore the general model-theoretic constructions in this situation, but at the moment it suggests to us in the same way that third cohomology for groups describes the realization of outer actions by possible extensions, higher cohomologies of universal algebras should be describing structures interpreting special multisorted expansions of the language of 2-cocycles together with the equations they satisfy.

The genesis for the present work and its sequels Wires \cite{wiresV,wiresVII} is the curious result \cite[Prop 7.1]{commod} and the following corollary \cite[Cor 7.2]{commod} which provides a characterization of 2-step nilpotent algebras in congruence modular varieties. This is extended in Theorem~\ref{thm:abelian1} and Proposition~\ref{prop:5} below to varieties with a difference term. While not stated explicitly, \cite[Prop 7.1]{commod} is essentially the decomposition of a central extension in a manner similar to that found in group theory where a factor set is added to the operations of the direct product of the datum groups. The analogy can be developed further where the unital ring, which comes from the derived module structure of abelian algebras in varieties with a difference term, encodes the action terms which correspond in the case of groups to the trivial action, or $\mathds{Z}$-module structure, on abelian kernels in central extensions. Some classical results concerning central extensions of groups are generalized in \cite{wiresV} to the broader class of varieties with a difference term. The current manuscript is a generalization to arbitrary varieties of the original motivating observations.

The approach follows the constructions and concrete manipulations of functions in what is sometimes referred to as the Schreier cohomology of group extensions found in Schreier \cite{schreier} and Dedecker\cite{dedecker}. We define the notion of affine datum and establish the machinery around the $1^{\mathrm{st}}$ and $2^{\mathrm{nd}}$-cohomology groups characterizing extensions in a variety $\mathcal V$ which realize the datum. The development is satisfyingly replete with the standard toolkit of 2-cocycles, 2-coboundaries, actions, derivations and stabilizing automorphisms which reduce to previous definitions upon specialization for many familiar classes of algebras. We provide a summary of the key points of the development: 
\begin{itemize}

	\item The notion of an action compatible with a set of equations (Definition~\ref{def:5}).

	\item The notion of 2-cocycles as interpretations of a multisorted signature compatible with a set of equations (Definition~\ref{def:2cocyle}).

	\item Definition of affine datum (Definition~\ref{def:datum}). 

	\item Reconstruction of an algebra from affine datum and 2-cocycle compatible with a variety $\mathcal U$ (Theorem~\ref{thm:12}).

	\item For an abelian congruence $\alpha \in \Con A$ in which $\mathcal V(A)$ has a weak-difference term, decomposition of the algebra into the quotient $A/\alpha$, a partial structure $A^{\alpha,\tau}$, and a 2-cocycle $T$ and homomorphic action $A/\alpha \ast A^{\alpha,\tau}$ derived from $A^{\alpha,\tau}$ both compatible with any parameter variety $\mathcal U \geq \mathcal V(A)$ (Theorem~\ref{thm:10}).

	\item The characterization of semidirect products realizing affine datum (Proposition~\ref{prop:6}).
	
	\item Definition of 2-coboundaries associated to affine datum (Definition~\ref{def:2cobound}). We shall see 2-coboundaries are compatible with any variety containing the datum (Lemma~\ref{lem:5}).
	
	\item An equivalence on 2-cocycles (and so extensions) derived from 2-coboundaries which is finer than isomorphism (Theorem~\ref{thm:13}).
	
	\item The abelian $2^{\mathrm{nd}}$-cohomology groups $H^{2}_{\mathcal U}(Q,A^{\alpha,\tau},\ast)$ for affine datum defined from 2-cocycles modulo the 2-coboundaries. The cohomology group accepts as an additional parameter a variety $\mathcal U$ which restricts the class of extensions to membership in $\mathcal U$ (Theorem~\ref{thm:cohomology}).
	
	\item The	$2^{\mathrm{nd}}$-cohomology groups define a Galois connection between the lattice of varieties containing fixed datum and the subgroup lattice of the ableian group generated by 2-cocycles of the datum (Proposition~\ref{prop:14}).
	
	\item The automorphisms stabilizing an extension form an abelian group isomorphic to the 1-cocycles (derivations) of the datum (Theorem~\ref{thm:stabilize}).
	
	\item In varieties with a difference-term, a characterization of central extensions by affine datum with trivial actions (Proposition~\ref{thm:extension} and Proposition~\ref{prop:centralaction}). In such varieties, central extensions of the datum are characterized by the $2^{\mathrm{nd}}$-cohomology group $H^{2}_{\mathcal U}(Q,A^{\alpha,\tau},\ast)$ restricted to trivial actions (Theorem~\ref{thm:centralcohom}).
	
	\item In varieties with a weak-difference term, abelian extension are classified by a subgroup of $2^{\mathrm{nd}}$-cohomology which generalizes the classic group-theoretic case which characterizes abelian extensions by symmetric 2-cocycles (Corollary~\ref{cor:21}).
	
	\item In Section~\ref{sec:4}, the cohomology machinery is specialized in the presence of an idempotent element to a subclass of varieties with a difference term. The representation of extensions with abelian kernels (Theorem~\ref{thm:idemabelian} and Corollary~\ref{cor:AAA}) more closely aligns with the familiar example of groups and can be applied to any variety of groups with multiple operators (see Higgins \cite{higgins}).

\end{itemize}
In the sequel manuscript \cite{wiresV}, we consider central extensions in varieties with a difference term. Such varieties include groups, quasigroups, relative Rota-Baxter groups and braces but also many familiar examples of abelian groups expanded by multilinear operations such as rings, Liebniz algebras, diassociative algebras, Rota-Baxter algebras and conformal algebras. In this very general setting, we establish a low-dimensional Hochschild-Serre exact sequence for central extensions and prove theorems of Schur-type relating relative commutators of free presentations and $2^{\mathrm{nd}}$-cohomology of regular datum, covers and perfect algebras. This generalizes analogous results for many previously established cases, including groups, which can then be derived by specialization to those particular varieties. Highlights include:
\begin{itemize}

	\item A low-dimensional Hochschild-Serre exact sequence (or inflation/restriction sequence) for a central extension with an additional affine datum in varieties with a difference term.
	
	\item Characterizing injectivity and surjectivity of the transgression map and its relation to lifting homomorphisms through central extensions of regular datum.
	
	\item Generalization of Schur's formula relating commutators of presentations with second-cohomology to any variety with a difference term which has an idempotent element (see Schur \cite{schur1} and Karpilovski \cite{karpil}). By specialization this recovers the classical result for groups and more recent work on algebras of Loday-type. 

	\item Discussion of covers and the relationship between cohomology and perfect algebras (see Milnor \cite{milnor}).

\end{itemize}

Consult Bergman \cite{bergman} for fundamentals of universal algebras. For the basic definition and general properties of the term-condition commutator, Kearnes and Kiss \cite[Cha 2.5]{shape} is useful and for the theory of the commutator in congruence modular varieties consult Freese and McKenzie \cite{commod} and McKenzie and Snow \cite{mckenziesnow}. The properties of the term-condition commutator in varieties with a difference term is developed in Kearnes \cite{vardiff}. Any special definitions are introduced in the text as needed.

\vspace{0.3cm}

\section{Varieties with a difference term}\label{sec:2}

Fix a variety $\mathcal V$ in the signature $\tau$, an algebra $A \in \mathcal V$ and $\alpha,\beta \in \Con A$. A \emph{section} of a surjective map $\pi: A \rightarrow Q$ is a map $l: Q \rightarrow A$ such that $\pi \circ l =\id_{Q}$. An $\alpha$-\emph{trace} is a map $r: A \rightarrow A$ such that $r = l \circ \pi$ for a section $l$ of the canonical map $\pi: A \rightarrow A/\alpha$; equivalently, $(r(x),x) \in \alpha$ and $|r(x/\alpha)|=1$ for all $x \in A$. Let us recall the definition and some properties of the congruence $\Delta_{\alpha\beta}$. 
\begin{itemize}

    \item $M(\alpha,\beta)$ is the subalgebra of $A^{4}$ generated by $\left\{ \begin{bmatrix} x & x \\ y & y \end{bmatrix}, \begin{bmatrix} u & v \\ u & v \end{bmatrix}: (x,y) \in \alpha, (u,v) \in \beta \right\}$.

	\item $A(\alpha) = \left\{ (x,y) \in A \times A : (x,y) \in \alpha \right\}$ is the congruence $\alpha$ as a subalgebra of $A \times A$.

	\item $\Delta_{\alpha \beta} = \Cg^{A(\alpha)}\left( \left\{ \left\langle \begin{bmatrix} u \\ u \end{bmatrix} , \begin{bmatrix} v \\ v \end{bmatrix} \right\rangle : (u,v) \in \beta \right\} \right) = \mathrm{Tr} \, M(\alpha, \beta) \in \Con A(\alpha)$ where $\mathrm{Tr}$ denotes the transitive closure of a binary relation.
	
	\item The diagonal homomorphism is the map $\delta: A \rightarrow A(\alpha)/\Delta_{\alpha \alpha}$ given by $\delta(u) = \begin{bmatrix} u \\ u \end{bmatrix}/\Delta_{\alpha \alpha}$.
	
  \item We have the following inclusions
\[
\Delta_{\alpha \beta} \vee \eta_{0} = p_{0}^{-1}(\beta) \quad \quad \text{and} \quad \quad \Delta_{\alpha \beta} \vee \eta_{1} = p_{1}^{-1}(\beta)
\]    
for the projections $p_i: A(\alpha) \rightarrow A$ with $\eta_{i} = \ker p_{i}$.

\end{itemize}

In the case that $\mathcal V$ is a congruence modular variety, $\Delta_{\alpha\beta}$ has stronger properties.

\begin{itemize}

	\item We have the following description of the commutator (\cite[Thm 4.9]{commod}):
	\[
	(x,y) \in [\alpha,\beta] \quad\quad \text{iff} \quad\quad (\exists u \in A) \ \ \begin{bmatrix} u \\ u \end{bmatrix} \Delta_{\alpha \beta} \begin{bmatrix} x \\ y \end{bmatrix} \quad\quad \text{iff} \quad\quad (\exists v \in A) \ \ \begin{bmatrix} v \\ x \end{bmatrix} \Delta_{\alpha \beta} \begin{bmatrix} v \\ y \end{bmatrix}.
	\]

	\item If $\alpha \leq \beta$ and $[\alpha,\beta]=0$, then we have the following description of the congruence (\cite[Thm 5.5, Prop 5.7]{commod}):
	\[
	\begin{bmatrix} x \\ y \end{bmatrix} \Delta_{\alpha \beta} \begin{bmatrix} u \\ v \end{bmatrix} \quad\quad \text{iff} \quad\quad  x \, \alpha \, y \, \beta \, u \ \text{and} \ v=m(y,x,u).
	\]

	\item We have the following inclusions (\cite[Lem 8.6]{mckenziesnow}):
\[
 \Delta_{\alpha \beta} \wedge \eta_{0} \leq p^{-1}_{1}([\alpha,\beta]) \quad \quad \quad \Delta_{\alpha \beta} \wedge \eta_{1} \leq p^{-1}_{0}([\alpha,\beta]) \quad \quad \quad [p_{0}^{-1}(\beta),p_{1}^{-1}(\beta)] \leq \Delta_{\alpha \beta}.
\]

\end{itemize}

If $\alpha$ is an abelian congruence, then we can easily observe some analogues in varieties with a difference term or weak-difference term. Given an equivalence relation $\alpha$ on a set $A$, there is an equivalence relation $\hat{\alpha}$ on the set $A(\alpha) \subseteq A \times A$ defined by $\begin{bmatrix} x \\ y \end{bmatrix} \mathrel{\hat{\alpha}} \begin{bmatrix} u \\ v \end{bmatrix} \Leftrightarrow x \mathrel{\alpha} y \mathrel{\alpha} u \mathrel{\alpha} v$.

\begin{lemma}\label{lem:20}
Let $\mathcal V$ be a variety, $A \in \mathcal V$ and $\alpha, \beta \in \Con A$. 
\end{lemma}
\begin{enumerate}

	\item If $\mathcal V$ has a difference term, then
	\[
	\begin{bmatrix} a \\ b \end{bmatrix} \Delta_{\alpha \beta} \begin{bmatrix} a \\ d \end{bmatrix} \quad \Longrightarrow \quad b \mathrel{[\alpha,\beta]} d .
	\]
	
	\item If $\mathcal V$ has a weak-difference term and $\alpha$ is abelian, then
	\[
	(\exists a \in A) \ \begin{bmatrix} a \\ b \end{bmatrix} \Delta_{\alpha \beta} \begin{bmatrix} a \\ d \end{bmatrix} \quad \Longleftrightarrow \quad b \mathrel{[\beta,\alpha]} d .
	\]

	\item If $\mathcal V$ has a weak-difference term and $\alpha$ is abelian, then $\Delta_{\alpha \alpha} = \Delta_{\alpha \gamma} \wedge \hat{\alpha}$ for any $\alpha \leq \gamma \leq (0:\alpha)$.
	
	\item If $\mathcal V$ has a difference term and $\alpha$ is abelian, the following are equivalent:

		\begin{enumerate}

			\item $\begin{bmatrix} a \\ b \end{bmatrix} \Delta_{\alpha \beta} \begin{bmatrix} c \\ d \end{bmatrix}$;
	
			\item $\begin{bmatrix} a \\ b \end{bmatrix} \Delta_{\alpha \beta} \begin{bmatrix} c \\ m(b,a,c) \end{bmatrix}$ \quad and \quad $d \mathrel{[\alpha, \beta]} m(b,a,c)$;
	
			\item $c \mathrel{\beta} a \mathrel{\alpha} b$ \quad and \quad $d \mathrel{[\alpha, \beta]} m(b,a,c)$.

		\end{enumerate}

	\item Let $m$ be a difference term for $\mathcal V$ and $f$ a function symbol. If $\alpha$ is abelian and $\vec{a},\vec{b},\vec{c} \in A^{\ar f}$ are pair-wise $\alpha$-related $\vec{a} \mathrel{\alpha} \vec{b} \mathrel{\alpha} \vec{c}$, then $m(f(\vec{a}),f(\vec{a}),f(\vec{a})) = f(m(a_{1},b_{1},c_{1}),\ldots,m(a_{1},b_{1},c_{1}))$.

\end{enumerate}
\begin{proof}
(1) Note $\begin{bmatrix} a \\ b \end{bmatrix} \Delta_{\alpha \beta} \begin{bmatrix} a \\ d \end{bmatrix}$ implies $(b,d) \in \alpha \wedge \beta$, and so $d \mathrel{[\alpha,\beta]} m(d,b,b)$ since $m$ is a difference term. Applying the difference term to the sequence 
\[
\begin{bmatrix} a  \\ b \end{bmatrix} \Delta_{\alpha \beta} \begin{bmatrix} a  \\ d \end{bmatrix} \ , \ \begin{bmatrix} a  \\ b \end{bmatrix} \Delta_{\alpha \beta} \begin{bmatrix} a  \\ b \end{bmatrix} \ , \ \begin{bmatrix} b  \\ b \end{bmatrix} \Delta_{\alpha \beta} \begin{bmatrix} b  \\ b \end{bmatrix} \quad \Longrightarrow \quad \begin{bmatrix} b  \\ b \end{bmatrix} = \begin{bmatrix} m(a,a,b)  \\ m(b,b,b) \end{bmatrix} \Delta_{\alpha \beta} \begin{bmatrix} m(a,a,b)  \\ m(d,b,b) \end{bmatrix} = \begin{bmatrix} b  \\ m(d,b,b) \end{bmatrix}.
\]
Since $\Delta_{\alpha \beta} = \mathrm{Tr} \, M(\alpha, \beta)$, we see that $(b,m(d,b,b)) \in [\beta,\alpha] = [\alpha,\beta]$. Then $b \mathrel{[\alpha,\beta]} m(d,b,b) \mathrel{[\alpha,\beta]} d$.

(2) For necessity, this is the same calculation as part (1) except $m$ is the weak-difference term. We use $m(a,a,b)=b$ because $(a,b) \in \alpha$ is abelian, and $(b,d) \in \alpha \wedge \beta \leq \alpha$ implies $d=m(d,b,b)$, as well. For sufficiency, use the recursive generation of the TC-commutator by $M(\alpha,\beta)$ matrices starting from the equality relation. The thing to note is that the elements in the matrices in $M(\alpha,\beta)$ which witness any inclusion $(a,b) \in [\alpha,\beta]$ will all be contained in a single $\alpha$-class on which the weak-difference term behaves as a Mal'cev term.

(3) Note $\Delta_{\alpha \alpha} \leq \Delta_{\alpha \gamma} \wedge \hat{\alpha}$. Conversely, suppose $\begin{bmatrix} a \\ b \end{bmatrix} \Delta_{\alpha \gamma} \wedge \hat{\alpha} \begin{bmatrix} c \\ d\end{bmatrix}$. Then the coordinates are all contained in a single $\alpha$-class. Apply the weak-difference term to the generators
\[
\begin{bmatrix} a  \\ b \end{bmatrix} \Delta_{\alpha \alpha} \begin{bmatrix} a  \\ b \end{bmatrix} \ , \ \begin{bmatrix} a  \\ a \end{bmatrix} \Delta_{\alpha \alpha} \begin{bmatrix} a  \\ a \end{bmatrix} \ , \ \begin{bmatrix} a  \\ a \end{bmatrix} \Delta_{\alpha \alpha} \begin{bmatrix} c  \\ c \end{bmatrix} \quad \Rightarrow \quad \begin{bmatrix} a  \\ b \end{bmatrix} = \begin{bmatrix} m(a,a,a)  \\ m(b,a,a) \end{bmatrix} \Delta_{\alpha \alpha} \begin{bmatrix} m(a,a,c)  \\ m(b,a,c) \end{bmatrix} = \begin{bmatrix} c  \\ m(b,a,c) \end{bmatrix}
\]
using that $\alpha$ is abelian. Then $\Delta_{\alpha \alpha}\leq \Delta_{\alpha \gamma}$ implies $\begin{bmatrix} c \\ d \end{bmatrix} \Delta_{\alpha \gamma} \begin{bmatrix} a \\ b \end{bmatrix} \Delta_{\alpha \gamma}  \begin{bmatrix} c  \\ m(b,a,c) \end{bmatrix}$, and part (2) yields $(d,m(b,a,c)) \in [\gamma,\alpha]=0$.

(4) Assuming (a), we have $(a,b) \in \alpha$ and $(a,c) \in \beta$. Apply the difference term to the sequence
\[
\begin{bmatrix} a  \\ b \end{bmatrix} \Delta_{\alpha \beta} \begin{bmatrix} a  \\ b \end{bmatrix} \ , \ \begin{bmatrix} a  \\ a \end{bmatrix} \Delta_{\alpha \beta} \begin{bmatrix} a  \\ a \end{bmatrix} \ , \ \begin{bmatrix} a  \\ a \end{bmatrix} \Delta_{\alpha \beta} \begin{bmatrix} c  \\ c \end{bmatrix} \quad \Rightarrow \quad \begin{bmatrix} a  \\ b \end{bmatrix} = \begin{bmatrix} m(a,a,a)  \\ m(b,a,a) \end{bmatrix} \Delta_{\alpha \beta} \begin{bmatrix} m(a,a,c)  \\ m(b,a,c) \end{bmatrix} = \begin{bmatrix} c  \\ m(b,a,c) \end{bmatrix}.
\]
Then applying part (1) to $\begin{bmatrix} c \\ d \end{bmatrix} \Delta_{\alpha \beta} \begin{bmatrix} a \\ b \end{bmatrix} \Delta_{\alpha \beta}  \begin{bmatrix} c  \\ m(b,a,c) \end{bmatrix}$ produces $(d,m(b,a,c)) \in [\alpha,\beta]$. From (b), (c) follows immediately. 

Now assume (c). By (2) above, there is $x \in A$ such that $\begin{bmatrix} x \\ d \end{bmatrix} \Delta_{\alpha \beta} \begin{bmatrix} x \\ m(b,a,c) \end{bmatrix}$. We also have $m(b,a,c) \mathrel{\alpha} m(a,a,c)=c$ and $d \mathrel{\alpha \wedge \beta} m(b,a,c) \mathrel{\beta} p(b,a,a) = b$. The condition $c \mathrel{\beta} a \mathrel{\alpha} b$ produces from the generators
\[
\begin{bmatrix} a  \\ b \end{bmatrix} \Delta_{\alpha \beta} \begin{bmatrix} a  \\ b \end{bmatrix} \ , \ \begin{bmatrix} a  \\ a \end{bmatrix} \Delta_{\alpha \beta} \begin{bmatrix} a  \\ a \end{bmatrix} \ , \ \begin{bmatrix} a  \\ a \end{bmatrix} \Delta_{\alpha \beta} \begin{bmatrix} c  \\ c \end{bmatrix} \quad \Rightarrow \quad \begin{bmatrix} a  \\ b \end{bmatrix} \Delta_{\alpha \beta}  \begin{bmatrix} c  \\ m(b,a,c) \end{bmatrix}.
\]
We then apply the difference term to the sequence
\[
\begin{bmatrix} x  \\ d \end{bmatrix} \Delta_{\alpha \beta} \begin{bmatrix} x  \\ d \end{bmatrix} \ , \ \begin{bmatrix} x  \\ d \end{bmatrix} \Delta_{\alpha \beta} \begin{bmatrix} x  \\ m(b,a,c) \end{bmatrix} \ , \ \begin{bmatrix} b  \\ b \end{bmatrix} \Delta_{\alpha \beta} \begin{bmatrix} m(b,a,c)  \\ m(b,a,c) \end{bmatrix} \quad \Rightarrow \quad \begin{bmatrix} b  \\ b \end{bmatrix} \Delta_{\alpha \beta}  \begin{bmatrix} m(b,a,c)  \\ d \end{bmatrix}.
\]
We then use these last two relations and apply the difference term to derive 
\[
\begin{bmatrix} a  \\ b \end{bmatrix} \Delta_{\alpha \beta} \begin{bmatrix} c  \\ m(b,a,c) \end{bmatrix} \ , \ \begin{bmatrix} b  \\ b \end{bmatrix} \Delta_{\alpha \beta} \begin{bmatrix} m(b,a,c)  \\ m(b,a,c) \end{bmatrix} \ , \ \begin{bmatrix} b  \\ b \end{bmatrix} \Delta_{\alpha \beta} \begin{bmatrix} m(b,a,c)  \\ d \end{bmatrix} \quad \Rightarrow \quad \begin{bmatrix} a  \\ b \end{bmatrix} \Delta_{\alpha \beta}  \begin{bmatrix} c  \\ d \end{bmatrix}.
\]

(5) Applying the difference term to the generators of $\Delta_{\alpha \alpha}$ yields 
\begin{align*}
\begin{bmatrix} b_{i} \\ a_{i} \end{bmatrix} = \begin{bmatrix} m(a_{i},a_{i},b_{i}) \\ m(a_{i},b_{i},b_{i}) \end{bmatrix} \Delta_{\alpha \alpha} \begin{bmatrix} m(a_{i},a_{i},c_{i}) \\ m(a_{i},b_{i},c_{i}) \end{bmatrix} = \begin{bmatrix} c_{i} \\ m(a_{i},b_{i},c_{i}) \end{bmatrix}
\end{align*}
for each $i$. Then applying $f$ to the above pairs we have 
\begin{align*}
\begin{bmatrix} f(\vec{b}) \\ f(\vec{a}) \end{bmatrix} \Delta_{\alpha \alpha} \begin{bmatrix} f(\vec{c}) \\ f(m(a_{i},b_{i},c_{i}),\ldots,m(a_{i},b_{i},c_{i})) \end{bmatrix}.
\end{align*}
The result now follows from (4b) above since $\alpha$ is abelian. 
\end{proof}

\begin{lemma}\label{lem:usefulbit1}
Let $\mathcal V$ be a variety with a difference term $m$ and $A \in \mathcal V$ with abelian congruence $\alpha \in \Con A$. If $a \mathrel{\alpha} b \mathrel{\alpha} c$, then for all $x \in A$ we have
\begin{itemize}

	\item $m(m(a,b,c),c,x) = m(a,b,x) = m(m(a,c,x),x,m(c,b,x))$, and 
	
	\item $m(m(a,c,b),c,x) = m(m(a,c,x),x,m(b,c,x))$.

\end{itemize}
\end{lemma}
\begin{proof}
For the first set of equations, observe that
\begin{align*}
\begin{bmatrix} x & x \\ m(m(a,b,c),c,x) &  m(a,b,x) \end{bmatrix} = \begin{bmatrix} m(m(a,a,c),c,x) & m(m(a,a,b),b,x) \\ m(m(a,b,c),c,x) &  m(m(a,b,b),b,x) \end{bmatrix} \in M(\alpha,\alpha)
\end{align*}
which implies $m(m(a,b,c),c,x) \mathrel{[\alpha,\alpha]} m(a,b,x)$. Equality follows since $\alpha$ is abelian. Using Lemma~\ref{lem:20}(5) we see that
\begin{align*}
m(a,b,x) = m(m(a,c,c),m(c,c,b),m(x,x,x)) &= m(m(a,c,x),m(c,c,x),m(c,b,x)) \\
&= m(m(a,c,x),x,m(c,b,x)) .
\end{align*}
The second equation also follows by Lemma~\ref{lem:20}(5) since
\begin{align*}
m(m(a,c,b),c,x) = m(m(a,c,b),m(c,c,c),m(x,x,x)) &= m(m(a,c,x),m(c,c,x),m(b,c,x)) \\
&= m(m(a,c,x),x,m(b,c,x)) .
\end{align*} 
\end{proof}

\begin{lemma}\label{lem:1}
Let $\mathcal V$ be a variety, $A \in \mathcal V$ and $\alpha,\beta,\sigma \in \Con A$. Let $r: A \rightarrow A$ be a $\sigma$-trace. 
\begin{enumerate}

	\item If $\mathcal V$ has a difference term, then
	
		\begin{enumerate}
		
			\item $A(\alpha)/\Delta_{\alpha 1}$ is abelian;
			
			\item the set map \ $\psi: A/[\alpha,\beta] \longrightarrow A(\alpha)/\Delta_{\alpha \beta} \, \times \, A/\sigma$ \ which is defined by
	\[
	\psi(x/[\alpha,\beta]) = \left( \left\langle r(x),x\right\rangle/\Delta_{\alpha \beta}, x/\sigma \right)
	\]
	is injective.
		
		\end{enumerate}
		
	\item If $\mathcal V$ is congruence modular, then
	
		\begin{enumerate}
		
			\item for all $(a,b) \in \alpha$ and $u \in A$, $\begin{bmatrix} u \\ u \end{bmatrix}/\Delta_{\alpha 1} = \begin{bmatrix} a \\ d(a,b,b) \end{bmatrix}/\Delta_{\alpha 1}$ where $d$ is the difference term for $\mathcal V$;
			
			\item $A/[\alpha,1](\alpha/[\alpha,1])/\Delta_{\alpha/[\alpha,1] 1} \approx A(\alpha)/\Delta_{\alpha 1}$;
		
		\end{enumerate}

\end{enumerate}
\end{lemma}
\begin{proof}

(1a) Since $\mathcal V$ has a difference term, the TC-commutator is join-additive in both coordinates \cite{vardiff}; therefore, 
\begin{align*}
[1_{A(\alpha)},1_{A(\alpha)}] = [p_{0}^{-1}(1_{A}),p_{1}^{-1}(1_{A})] &= [\eta_{1} \vee \Delta_{\alpha 1}, \eta_{0} \vee \Delta_{\alpha 1} ] \\
&= [\eta_{1},\eta_{0}] \vee [\eta_{1},\Delta_{\alpha 1}] \vee [\Delta_{\alpha 1},\eta_{0}] \vee [\Delta_{\alpha 1},\Delta_{\alpha 1}]   \\
&= [\eta_{1},\Delta_{\alpha 1}] \vee [\Delta_{\alpha 1},\eta_{0}] \vee [\Delta_{\alpha 1},\Delta_{\alpha 1}]   \\
&\leq \Delta_{\alpha 1}.
\end{align*}

(1b) Since $r: A \rightarrow A$ is a $\sigma$-trace, we have $r(x) \mathrel{\sigma} x$ and $r(x) = r(y)$ iff $(x,y) \in \sigma$. Suppose $\psi(x/[\alpha,\beta]) = \psi(y/[\alpha,\beta])$. Then $\begin{bmatrix} r(x) \\ x \end{bmatrix}/\Delta_{\alpha \beta}  = \begin{bmatrix} r(y) \\ y \end{bmatrix}/\Delta_{\alpha \beta}$ and $x/\sigma = y/\sigma$. So we have $(x,y) \in \sigma$ and $\begin{bmatrix} r(x) \\ x \end{bmatrix} \Delta_{\alpha \beta} \begin{bmatrix} r(y) \\ y \end{bmatrix}$. Then $r(x)=r(y)$ implies $(x,y) \in [\alpha,\beta]$ by Lemma~\ref{lem:20}(1), and so $x/[\alpha,\beta] = y/[\alpha,\beta]$.

(2a) For $(a,b) \in \alpha$, $a \mathrel{[\alpha,\alpha]} d(a,b,b)$ since $d$ is a difference term. Since $[\alpha,\alpha] \leq [\alpha,1]$, by the remarks before Lemma~\ref{lem:20} there is $v \in A$ such that $\begin{bmatrix} a \\ d(a,b,b) \end{bmatrix} \Delta_{\alpha 1} \begin{bmatrix} v \\ v \end{bmatrix} \Delta_{\alpha \beta} \begin{bmatrix} u \\ u \end{bmatrix}$ using the generators of $\Delta_{\alpha 1}$ for the second step.

(2b) Define $\phi : A(\alpha) \rightarrow A/[\alpha,1](\alpha/[\alpha,1])/\Delta_{\alpha/[\alpha,1] 1}$ by $\phi \left( \begin{bmatrix} a \\ b \end{bmatrix} \right) := \begin{bmatrix} a/[\alpha,1] \\ b/[\alpha,1] \end{bmatrix}/\Delta_{\alpha/[\alpha,1] 1}$. It is easy to see that $\phi$ is a surjective homomorphism. We now calculate the kernel. We have $\Delta_{\alpha 1} \subseteq \ker \phi$ since $\phi$ identifies the generators of $\Delta_{\alpha 1}$. Now assume $\left( \begin{bmatrix} a \\ b \end{bmatrix} , \begin{bmatrix} c \\ e \end{bmatrix} \right) \in \ker \phi$. Then $(a,b), (c,e) \in \alpha$ and
\[
\begin{bmatrix} a/[\alpha,1] \\ b/[\alpha,1] \end{bmatrix} \Delta_{\alpha/[\alpha,1] 1} \begin{bmatrix} c/[\alpha,1] \\ e/[\alpha,1] \end{bmatrix}.
\]
Since $\alpha/[\alpha,1]$ is central in $A/[\alpha,1]$, we have $d(b,a,c)/[\alpha,1] = e/[\alpha,1]$ and so $(e,d(b,a,c)) \in [\alpha,1]$. Then by congruence modularity, $\begin{bmatrix} u \\ u \end{bmatrix} \Delta_{\alpha \beta} \begin{bmatrix} e \\ d(b,a,c) \end{bmatrix}$ for some $u \in A$. Define $x +_{0} y : = d(x,0,y)$ where $0:= \begin{bmatrix} u \\ u \end{bmatrix}/\Delta_{\alpha 1}$ is the $\Delta_{\alpha 1}$-class containing the diagonal. Since $A(\alpha)/\Delta_{\alpha 1}$ is abelian by (1a), $x +_{0} y$ is the operation of an abelian group in which $0=\begin{bmatrix} e \\ d(b,a,c) \end{bmatrix}/\Delta_{\alpha 1}$ is the identity.

Now observe
\begin{align*}
\begin{bmatrix} a \\ b \end{bmatrix}/\Delta_{\alpha 1} + \begin{bmatrix} e \\ c \end{bmatrix}/\Delta_{\alpha 1} = \begin{bmatrix} a \\ b \end{bmatrix}/\Delta_{\alpha 1} + \begin{bmatrix} a \\ a \end{bmatrix}/\Delta_{\alpha 1}  + \begin{bmatrix} e \\ c \end{bmatrix}/\Delta_{\alpha 1}   &= \begin{bmatrix} d(a,a,e) \\ d(b,a,c) \end{bmatrix}/\Delta_{\alpha 1} \\
&= \begin{bmatrix} e \\ d(b,a,c) \end{bmatrix}/\Delta_{\alpha 1} = 0
\end{align*}
implies $\begin{bmatrix} e \\ c \end{bmatrix}/\Delta_{\alpha 1} = - \begin{bmatrix} a \\ b \end{bmatrix}/\Delta_{\alpha 1}$. Then
\begin{align*}
\begin{bmatrix} c \\ e \end{bmatrix}/\Delta_{\alpha 1} - \begin{bmatrix} a \\ b \end{bmatrix}/\Delta_{\alpha 1} = \begin{bmatrix} c \\ e \end{bmatrix}/\Delta_{\alpha 1} + \begin{bmatrix} e \\ c \end{bmatrix}/\Delta_{\alpha 1} &= \begin{bmatrix} d(c,c,e) \\ d(e,c,c) \end{bmatrix}/\Delta_{\alpha 1} \\
&= \begin{bmatrix} e \\ d(e,c,c) \end{bmatrix}/\Delta_{\alpha 1} = 0. \\
\end{align*}
This yields $\ker \phi \subseteq \Delta_{\alpha 1}$.
\end{proof}

\begin{remark}
There is freedom in the choice of $\sigma$ in Lemma~\ref{lem:1}(4). Injectivity is immediate if we choose $\sigma=0_{A}$ since the second coordinates will always be distinct. If we take $\sigma = 1_{A}$, then $A/[\alpha,\beta]$ is injective with $A(\alpha)/\Delta_{\alpha \beta}$. The question becomes for which choices of $\sigma$ do we have that $\pi$ is surjective and a homomorphism ?
\end{remark}

Suppose we have two algebras $B$ and $Q$ in the same signature $\tau$ and a binary operation on $B$ denoted by $x + y$. Suppose further that for every operation symbol $f \in \tau$ we have an operation $T_{f} : Q^{\ar f} \rightarrow B$ we shall call the transfer of $f$. We define a new algebra $B \otimes^{T} Q$ over the universe of the direct product $B \times Q$ where each operation symbol $f \in \tau$ is interpreted by the rule
\[
F_{f} \left( (b_1,q_1),\ldots, (b_n,q_n) \right) := \left(  f(b_1,\ldots,b_n) + T_{f}(q_1,\ldots,q_n) ,f(q_1,\ldots,q_n) \right).
\]
In order to prove Theorem~\ref{thm:abelian1}, we will first establish a special case in the next proposition which extends \cite[Prop 7.1]{commod} by considering an abelian congruence instead of the center. We observe the proof is almost the same.

\begin{proposition}\label{thm:extension}
Let $\mathcal V$ be a variety with a difference term and $A \in \mathcal V$. If $\alpha \in \Con A$ is abelian, then
\[
A/ [\alpha,1_{A}] \approx A(\alpha)/\Delta_{\alpha 1} \otimes^{T} A/\alpha.
\]
\end{proposition}
\begin{proof}
We argue by two claims. That $A(\alpha)/\Delta_{\alpha 1}$ is abelian is Lemma~\ref{lem:1}(1). Fix an $\alpha$-trace $r: A \rightarrow A$. Define the map $\psi: A/[\alpha,1] \rightarrow A(\alpha)/\Delta_{\alpha 1} \otimes^{T} A/\alpha$ by
\[
\psi(x/[\alpha,1]):= \left( \begin{bmatrix} r(x) \\ x \end{bmatrix} /\Delta_{\alpha 1} \ , \ x/\alpha \right)
\]
\vspace{0.1cm}

\textbf{Claim}: $\psi$ is bijective.
\begin{proof}
Injectivity of $\psi$ is Lemma~\ref{lem:1}(1b) where we take $\beta=1_{A}$ and $\sigma=\alpha$. To show surjectivity, take $\left( \begin{bmatrix} a \\ b \end{bmatrix}/\Delta_{\alpha 1} \ , \ x/\alpha \right) \in A(\alpha)/\Delta_{\alpha 1} \times A/\alpha$. Let $d$ be a difference term for $\mathcal V$. Then applying the difference term to the sequence
\[
\begin{bmatrix} a  \\ b \end{bmatrix} \Delta_{\alpha 1} \begin{bmatrix} a  \\ b \end{bmatrix} \ , \ \begin{bmatrix} a  \\ a \end{bmatrix} \Delta_{\alpha 1} \begin{bmatrix} a  \\ a \end{bmatrix} \ , \ \begin{bmatrix} a  \\ a \end{bmatrix} \Delta_{\alpha 1} \begin{bmatrix} r(x)  \\ r(x) \end{bmatrix}
\]
produces
\begin{align}\label{Eq:10}
\begin{bmatrix} a  \\ b \end{bmatrix} = \begin{bmatrix} d(a,a,a)  \\ d(b,a,a) \end{bmatrix} \Delta_{\alpha 1} \begin{bmatrix} d(a,a,r(x))  \\ d(b,a,r(x)) \end{bmatrix} &=  \begin{bmatrix} r(x)  \\ d(b,a,r(x)) \end{bmatrix}.
\end{align}
Then $(a,b) \in \alpha$ implies $r(a)=r(b)$ and so $d(b,a,r(x)) \mathrel{\alpha} d(r(b),r(a),r(x))=r(x)$; thus, $r(d(b,a,r(x)))= r(r(x))=r(x)$ and so $d(b,a,r(x))/\alpha = x/\alpha$. Altogether we have
\[
\psi \left(d(b,a,r(x))/[\alpha,1] \right) = \left( \begin{bmatrix} r(x)  \\ d(b,a,r(x)) \end{bmatrix}/\Delta_{\alpha 1}  \ , \ d(b,a,r(x))/\alpha \right) = \left( \begin{bmatrix} a \\ b \end{bmatrix}/\Delta_{\alpha 1} \ , \ x/\alpha \right).
\]
\end{proof}
\vspace{0.1cm}

We now define the transfer functions. For any basic operation $f$ with $\ar f = n$, define $T_{f}: A/\alpha \rightarrow A(\alpha)/\Delta_{\alpha 1}$ by
\[
T(x_{1}/\alpha,\ldots,x_{n}/\alpha) := \begin{bmatrix} r(f(x_1,\ldots,x_n) \\ f(r(x_1),\ldots,r(x_n)) \end{bmatrix} /\Delta_{\alpha 1}.
\]
For the binary operation on $A(\alpha)/\Delta_{\alpha 1}$, we take $x +_{0} y : = d(x,0,y)$ where $0:= \begin{bmatrix} u \\ u \end{bmatrix}/\Delta_{\alpha 1}$ is the $\Delta_{\alpha 1}$-class containing the diagonal. Since $A(\alpha)/\Delta_{\alpha 1}$ is abelian by Lemma~\ref{lem:1}(1a), $x +_{0} y$ is the operation of an abelian group in which $0$ is the identity. Since $\alpha$ is an abelian congruence, the difference term evaluates as a Mal'cev operation on $\alpha$-classes.
\vspace{0.1cm}

\textbf{Claim}: $\psi$ is a homomorphism.
\begin{proof}
Take $\bar{x}=(x_1,\ldots,x_n) \in A$ and write $r(\bar{x})= (r(x_1),\ldots,r(x_{n}))$. We calculate
\begin{align*}
F_{f}(\psi(\bar{x})) &= \left( \begin{bmatrix} f(r(\bar{x})) \\ f(\bar{x}) \end{bmatrix}/\Delta_{\alpha 1}  + \begin{bmatrix} r(f(\bar{x})) \\ f(r(\bar{x})) \end{bmatrix} /\Delta_{\alpha 1} \ , \ f(\bar{x})/\alpha \right) \\
&= \left( \begin{bmatrix} f(r(\bar{x})) \\ f(\bar{x}) \end{bmatrix}/\Delta_{\alpha 1}  + \begin{bmatrix} f(r(\bar{x})) \\ f(r(\bar{x})) \end{bmatrix}/\Delta_{\alpha 1} + \begin{bmatrix} r(f(\bar{x})) \\ f(r(\bar{x})) \end{bmatrix} /\Delta_{\alpha 1} \ , \ f(\bar{x})/\alpha \right) \\
&= \left( \begin{bmatrix} d \left( f(r(\bar{x})), f(r(\bar{x})), r(f(\bar{x})) \right) \\ d \left( f(\bar{x}), f(r(\bar{x})), f(r(\bar{x})) \right) \end{bmatrix} \ , \ f(\bar{x})/\alpha \right) \\
&= \left( \begin{bmatrix} r(f(\bar{x})) \\ f(\bar{x}) \end{bmatrix}  \ , \ f(\bar{x})/\alpha \right) \\
&= \psi (f(\bar{x}))
\end{align*}
because $f(r(\bar{x})) \mathrel{\alpha} f(\bar{x}) \ \Rightarrow \ d \left( f(\bar{x}), f(r(\bar{x})), f(r(\bar{x})) \right) = f(\bar{x})$.
\end{proof}
\vspace{0.1cm}
The theorem is established.
\end{proof}

\begin{remark}
The proof of Proposition~\ref{thm:extension} can show the following: Let $\mathcal V$ be a variety with a difference term and $A \in \mathcal V$. If $\alpha, \beta \in \Con A$ such that $\alpha \leq \beta$ with $\alpha$ abelian, then
\[
A/ [\alpha,\beta] \approx \Sg \left( \left\{ \left( \begin{bmatrix} a \\ b  \end{bmatrix}/\Delta_{\alpha \beta}, c/\alpha \right) : (a,c) \in \beta  \right\} \right) \leq A(\alpha)/\Delta_{\alpha 1} \otimes^{T} A/\alpha.
\]
\end{remark}

We can recover \cite[Prop 7.1]{commod} from the next corollary by taking $\alpha$ to be the center and the variety to be congruence modular.

\begin{corollary}\label{cor:1}(see \cite[Prop 7.1]{commod})
Let $\mathcal V$ be a variety with a difference term, $A \in \mathcal V$ and $\alpha \in \Con A$ central. Then
\[
A \approx A(\alpha)/\Delta_{\alpha 1} \otimes^{T} A/\alpha.
\]
\end{corollary}

Projection onto the second factor $p_{2}: A(\alpha)/\Delta_{\alpha 1} \otimes^{T} \rightarrow A/\alpha$ is a surjective homomorphism. If $\psi$ effects the isomorphism in Corollary~\ref{cor:1}, then $\ker p_{2} = \psi(\alpha)$.

\begin{theorem}\label{thm:abelian1}
Let $\mathcal V$ be a congruence modular variety, $A \in \mathcal V$ and $\alpha \in \Con A$. Then
\[
A/ [\alpha,1_{A}] \approx A(\alpha)/\Delta_{\alpha 1} \otimes^{T} A/\alpha.
\]
\end{theorem}
\begin{proof}
Apply Corollary~\ref{cor:1} to the algebra $A/[\alpha,1]$ and the central congruence $\alpha/[\alpha,1]$ to conclude
\[
A/[\alpha,1] \approx A/[\alpha,1](\alpha/[\alpha,1])/\Delta_{\alpha/[\alpha,1] 1} \otimes^{T} (A/[\alpha,1])/(\alpha/[\alpha,1]).
\]
For the second factor, the $2^{\mathrm{nd}}$-isomorphism theorem yields $(A/[\alpha,1])/(\alpha/[\alpha,1]) \approx A/\alpha$. For the first factor, Lemma~\ref{lem:1}(2b) gives the isomorphism $A/[\alpha,1](\alpha/[\alpha,1])/\Delta_{\alpha/[\alpha,1] 1} \approx A(\alpha)/\Delta_{\alpha 1}$.

To finish, we make the following observation. Suppose $\psi : B \rightarrow C$ is an isomorphism and $B \otimes^{T} Q$ is defined using the binary polynomial $x + y:= t(x,y,b_1,\ldots,b_k)$ for some term $t(x_1,\ldots,x_{k+2})$. Then $B \otimes^{T} Q \approx C \otimes^{T'} Q$ where $C \otimes^{T'} Q$ is defined using the binary polynomial $x \oplus y := t(x,y,\psi(b_{1}),\ldots,\psi(b_{k}))$ and transfer $T_{f}':=\psi(T_{f}): Q \rightarrow C$ for each fundamental operation $f$.
\end{proof}

The following proposition is an extension of \cite[Cor 7.2]{commod} to nilpotent algebras. We need to understand the evaluation of a term $t$ in the algebra $B \otimes^{T} Q$. Since the Mal'cev term is compatible with the operations of the abelian algebra $B$, by recursive evaluation along the composition tree of the term $t$, we see that the interpretation of $t$ in $B \otimes^{T} Q$ is given by 
\begin{align*}
F_{t} &\left( \left\langle \vec{a}, \vec{x} \right\rangle \right) \\
&= \left\langle t^{B}(\vec{a}) + \sum f^{B} \left( T_{g_{1}} \big( h^{Q}_{11}(\vec{y}_{11}),\ldots,h^{Q}_{1n_{1}}(\vec{y}_{1n_1}) \big) , \ldots,T_{g_{m}} \big( h^{Q}_{m1}(\vec{y}_{m1}),\ldots,h^{Q}_{mn_{m}}(\vec{y}_{mn_{m}}) \big) \right), t^{Q}(\vec{x}) \right\rangle \\
&= \left\langle t^{B}(\vec{a}) + s(\vec{x}), t^{Q}(\vec{x}) \right\rangle
\end{align*}
where the sum is taken over all $f,g_i,h_{ij}$ such that $f$ and $h_{ij}$ are subterms of $t$, $g_i \in \tau$ are operation symbols or variables and from the composition tree of $t$ we have $t = f \left( g_{1}(h_{11},\ldots,h_{1n_1}),\ldots,g_{m}(h_{m1},\ldots,h_{mn_{m}}) \right)$. The coordinates of the tuples $\vec{y}_{ij}$ all belong to $\vec{x}$. In the last line, we have written $s(\vec{x})$ for the above sum and we note that it is an operation that depends only on the tuples $\vec{x} \in Q^{\ar t}$. This suffices for the calculation in Lemma~\ref{lem:3}.

Define $[\alpha]_{1}:=[\alpha,\alpha]$ and recursively, $[\alpha]_{n+1}:=[\alpha,[\alpha]_{n}]$ for any $\alpha \in \Con A$, $A \in \mathcal V$. A congruence $\alpha$ is $n$-step nilpotent if $[\alpha]_{n} = 0$.

\begin{lemma}\label{lem:3}
Let $B$ and $Q$ be algebras in the same signature $\tau$ and suppose $B$ abelian with a Mal'cev term m(x,y,z). If the binary term $m(x,0,y):= x+y$ for some $0 \in B$ is used to define $\otimes^{T}$, then $[1,\ker q]=0$ where $q: B \otimes^{T} Q \rightarrow Q$ is projection.
\end{lemma}
\begin{proof}
Since $B$ is abelian, the $m(x,0,y)$ defines an abelian group operation. We verify the centrality condition. Fix a term $F_{t}(\bar{x},\bar{y})$ and tuples $\bar{a}=(\bar{a}^{1},\bar{a}^{2}),\bar{b}=(\bar{b}^{1},\bar{b}^{2}), (\bar{c}^{1},\bar{c}^{2})=\bar{c} \mathrel{\ker q} \bar{d}=(\bar{d}^{1},\bar{d}^{2})$ such that $F_{t}(\bar{a},\bar{c}) = F_{t}(\bar{a},\bar{d})$. This means
\[
\left\langle t^{B}(\bar{a}^{1},\bar{c}^{1}) + s(\bar{a}^{2},\bar{c}^{2}), t^{Q}(\bar{a}^{2},\bar{c}^{2}) \right\rangle = \left\langle t^{B}(\bar{a}^{1},\bar{d}^{1}) + s(\bar{a}^{2},\bar{d}^{2}), t^{Q}(\bar{a}^{2},\bar{d}^{2}) \right\rangle.
\]
Since $\bar{c} \mathrel{\ker q} \bar{d}$, we have $\bar{c}^{2}=\bar{d}^{2}$. This yields the equalities
\[
s(\bar{a}^{2},\bar{c}^{2})= s(\bar{a}^{2},\bar{d}^{2}) \quad \quad s(\bar{b}^{2},\bar{c}^{2}) = s(\bar{b}^{2},\bar{d}^{2}) \quad \quad t^{B}(\bar{b}^{2},\bar{c}^{2})=t^{B}(\bar{b}^{2},\bar{d}^{2}).
\]
Then from $t^{B}(\bar{a}^{1},\bar{c}^{1}) + s(\bar{a}^{2},\bar{c}^{2}) = t^{B}(\bar{a}^{1},\bar{d}^{1}) + s(\bar{a}^{2},\bar{d}^{2})$ we conclude $t^{B}(\bar{a}^{1},\bar{c}^{1}) = t^{B}(\bar{a}^{1},\bar{d}^{1})$. Since $B$ is abelian, we then have $t^{B}(\bar{b}^{1},\bar{c}^{1}) = t^{B}(\bar{b}^{1},\bar{d}^{1})$. Together with the above equalities we conclude $F_{t}(\bar{b},\bar{c}) = F_{t}(\bar{b},\bar{d})$.
\end{proof}

\begin{proposition}\label{prop:5}( see \cite[Cor 7.2]{commod})
Let $\mathcal V$ be a variety with a difference term. An algebra $A \in \mathcal V$ is n-step nilpotent if and only if it can be represented as a right-associated product
\[
A \approx Q_n \otimes^{T_{n-1}} Q_{n-1} \otimes^{T_{n-2}} \cdots \otimes^{T_{1}} Q_1
\]
for abelian algebras $Q_1,\ldots,Q_n \in \mathcal V$.
\end{proposition}
\begin{proof}
Assume $A$ is n-step nilpotent. Then $0=[1]_{n}=[1,[1]_{n-1}] < [1,[1]_{n-2}]$. Set $Q_{1} = A/[1,1]$ and for $2 \leq k \leq n$, define $Q_{k} = A/\Delta_{[1]_{k-1} 1}$ which are abelian by Lemma~\ref{lem:1}(1a). By Theorem~\ref{thm:abelian1}, we have
\[
A/[1]_{k} \approx Q_{k} \otimes^{T} A/[1]_{k-1}
\]
for $2 \leq k \leq n$. The associated product now follows since $A/[1,1]$ is abelian.

Now, suppose we have
\[
A \approx Q_n \otimes^{T_{n-1}} Q_{n-1} \otimes^{T_{n-2}} \cdots \otimes^{T_{1}} Q_1
\]
a right-associated product for abelian algebras $Q_1,\ldots,Q_n \in \mathcal V$. Set $B_{k} = Q_{k} \otimes^{T_{k-1}} \cdots \otimes^{T_{1}} Q_{1}$ for the right-associated product. Note we have surjective homomorphisms $q_{k+1}: B_{k+1} \rightarrow B_{k}$ given by right-projections This implies each $B_{k} \in \mathcal V$ for $1 \leq k \leq n$. The argument is by induction on $k$. Assume $B_{k}$ is $k$-step nilpotent and let $\alpha=\ker q_{k+1}$. By recursive application of the homomorphism property in varieties with a difference term,
\begin{align*}
\left(\left[1_{B_{k+1}}\right]_{k} \vee \alpha \right)/\alpha &= \left[\left(1_{B_{k+1}} \vee \alpha \right)/\alpha, \left(\left[1_{B_{k+1}}\right]_{k-1} \vee \alpha\right)/\alpha \right] \\
&= \left[ 1_{B_{k}}, \left[\left( 1_{B_{k+1}} \vee \alpha \right)/\alpha, \left(\left[1_{B_{k+1}}\right]_{k-2} \vee \alpha\right)/\alpha \right] \right] \\
&= \left[ 1_{B_{k}}, \left[ 1_{B_{k}}, \left[ \left( 1_{B_{k+1}} \vee \alpha \right)/\alpha , \left(\left[1_{B_{k+1}}\right]_{k-3} \vee \alpha\right)/\alpha \right] \right] \right] \\
&\vdots \\
&= [1_{B_{k}}]_{k} = 0;
\end{align*}
thus, $\alpha \geq [1]_{k}$. The hypothesis of Lemma~\ref{lem:3} is satisfied by Lemma~\ref{lem:1}(1a) and so $0=[1,\alpha] \leq [1,[1]_{k}]$; therefore, $B_{k+1}$ is k+1-step nilpotent. The argument is concluded since $A \approx B_{n}$.
\end{proof}

It is useful to consider the previous development in the case of groups; in particular, the $\otimes^{T}$ construction recovers the reconstruction of central extensions where the transfer corresponds to the addition of a 2-cocycle. For a normal subgroup $K \triangleleft G$, let $\alpha_{K} = \{(x,y) \in G^{2} : xy^{-1} \in K \}$ denote the corresponding congruence. Recall, given a homomorphism $\phi: Q \rightarrow \Aut K$ and group 2-cocycle $f: Q \times Q \rightarrow K$ the group $K \rtimes_{\phi,f} Q$ is defined over the set $K \times Q$ with operation
\[
(a,x) \cdot (b,y) = (a \cdot \phi(x)(b) \cdot f(x,y), x \cdot y ). 
\]

\begin{lemma}\label{lem:grp}
Let $G$ be a group and $K, H \triangleleft G$ with $K \leq H$.
\begin{enumerate}

	\item $G(\alpha_{K})/\Delta_{\alpha_{K} \alpha_{H}} \approx K/[K,H] \rtimes_{\phi} G/H$ for a homomorphism $\phi: G/H \rightarrow \Aut K/[K,H]$.
		
  \item $G(\alpha_{K})/\Delta_{\alpha_{K} 1} \approx K/[K,G]$.
	
	\item $G/[K,G] \approx K/[K,G] \ \otimes^{T} G/K$ for some transfer $T$.
	
	\item For a central extensions $\pi \colon G \rightarrow Q$ with $K=\ker \pi$, the transfers $T_{\sigma}: Q^{\ar(\sigma)} \rightarrow G(\alpha_{K})/\Delta_{\alpha_{K}1}$ and their image under the isomorphism from (2) are given by
	
		\begin{itemize}
			
			\item $T_{\times}(x,y)= \begin{bmatrix} l(xy) \\ l(x)l(x) \end{bmatrix}/\Delta_{\alpha_{K}1} \longmapsto l(x)l(y)l(xy)^{-1}$

			\item $T_{-1}(x) = \begin{bmatrix} l(x^{-1}) \\ l(x)^{-1} \end{bmatrix}/\Delta_{\alpha_{K}1} \longmapsto l(x)^{-1}l(x^{-1})^{-1}$
			
			\item $T_{1}(x) = \begin{bmatrix} 0 \\ 0 \end{bmatrix}/\Delta_{\alpha_{K}1} \longmapsto 0$
			
		\end{itemize}
; therefore, $G \approx K \otimes^{T} Q \approx K \rtimes_{0,f} Q$ for the 2-cocycle $f(x,y) =  l(x)l(y)l(xy)^{-1}$.
\end{enumerate}
\end{lemma}
\begin{proof}

(1) Let $l:G/H \rightarrow G$ be a lifting associated to $\alpha_{H}$. Define $\phi: G/H \rightarrow \Aut K/[K,H]$ by $\phi(xH)(y[K,H]) := l(xH)yl(xH)^{-1}[K,H] = y^{l(xH)}[K,H]$. Define $\psi: K \rtimes_{\phi} G/H \rightarrow G(\alpha_{K})/\Delta_{\alpha_{K} \alpha_{H}}$ by $\psi(k,q) := \begin{bmatrix} l(q) \\ kl(q) \end{bmatrix}/\Delta_{\alpha_{K} \alpha_{H}}$. From the generators of $\Delta_{\alpha_{K} \alpha_{H}}$, we see that $\begin{bmatrix} a \\ b \end{bmatrix} \Delta_{\alpha_{K} \alpha_{H}} \begin{bmatrix} c \\ d \end{bmatrix}$ implies $(a,b),(c,d) \in \alpha_{K}$ and $(a,c),(b,d) \in \alpha_{H}$. Note $\psi(k_1,q_1)=\psi(k_2,q_2) \Rightarrow \begin{bmatrix} l(q_{1}) \\ k_{1}l(q_{1}) \end{bmatrix}  \Delta_{\alpha_{K} \alpha_{K}} \begin{bmatrix} l(q_{2}) \\ k_{2}l(q_{2}) \end{bmatrix} \Rightarrow (l(q_1),l(q_2)) \in \alpha_{H} \Rightarrow q_1=q_2$ since $l$ is a lifting for $\pi: G \rightarrow G/H$. Then we must have $(k_1,k_{2}) \in [K,H]$; thus, $\ker \psi \leq \alpha_{[K,H]} \times 0_{G/H}$. Conversely, for $(k_1,k_2) \in \alpha_{[K,G]}$ and $q \in G/H$, we have $\begin{bmatrix} q \\ k_1q \end{bmatrix} \Delta_{\alpha_{K} \alpha_{H}} \begin{bmatrix} q \\ k_2q \end{bmatrix}$ which shows $\alpha_{[K,H]} \times 0_{G/H} \leq \ker \psi$.

For surjectivity, take $\begin{bmatrix} a \\ b \end{bmatrix}/\Delta_{\alpha_{K} \alpha_{H}} \in G(\alpha_{K})/\Delta_{\alpha_{K} \alpha_{H}}$. Then $a$ and $b$ are in the same coset of $K \leq H$ and so there is $q \in G/H$, $h_1,h_2 \in H$ such that $\begin{bmatrix} a \\ b \end{bmatrix}/\Delta_{\alpha_{K} \alpha_{K}} = \begin{bmatrix} h_{1}l(q) \\ h_{2}l(q) \end{bmatrix}/\Delta_{\alpha_{K} \alpha_{H}}$. Since $(a,b) \in \alpha_{K}$, we have $h_{2}h^{-1}_{1} \in K$ and so can write $h_2 = kh_1$ for some $k \in K$. Note $(l(q),h_{1}^{-1}l(q)) \in \alpha_{H}$ which implies $\begin{bmatrix} l(q) \\ l(q) \end{bmatrix} \Delta_{\alpha_{K} \alpha_{H}} \begin{bmatrix} h_{1}^{-1}l(q) \\ h_{1}^{-1}l(q) \end{bmatrix}$ as a generator of the congruence. Then
\[
 \begin{bmatrix} a \\ b \end{bmatrix} = \begin{bmatrix} h_{1}l(q) \\ h_{2}l(q) \end{bmatrix} = \begin{bmatrix} h_{1} \\ h_{2} \end{bmatrix} \cdot \begin{bmatrix} l(q) \\ l(q) \end{bmatrix} \Delta_{\alpha_{K} \alpha_{H}} \begin{bmatrix} h_{1} \\ h_{2} \end{bmatrix} \cdot \begin{bmatrix} h_{1}^{-1}l(q) \\ h_{1}^{-1}l(q) \end{bmatrix} = \begin{bmatrix} l(q) \\ h_{2}h_{1}^{-1}l(q) \end{bmatrix} = \begin{bmatrix} l(q) \\ kl(q) \end{bmatrix}.
\]
Then $\psi((k,q)) = \begin{bmatrix} a \\ b \end{bmatrix}/\Delta_{\alpha_{K} \alpha_{H}}$.

We check $\psi$ is a homomorphism:
\begin{align*}
\psi\left( k_{1},q_{1} \right) \cdot \psi( k_{2},q_{2} ) &= \begin{bmatrix} l(q_{1}) \\ k_{1}l(q_{1}) \end{bmatrix}/\Delta_{\alpha_{K} \alpha_{H}} \cdot \begin{bmatrix} l(q_{2}) \\ k_{2}l(q_{2}) \end{bmatrix}/\Delta_{\alpha_{K} \alpha_{H}} \\
&= \begin{bmatrix} l(q_{1})l(q_{2}) \\ k_{1}l(q_{1})k_{2}l(q_2) \end{bmatrix}/\Delta_{\alpha_{K} \alpha_{H}} \\
&= \begin{bmatrix} f(q_{1},q_{2})l(q_{1}q_{2}) \\ k_{1}k_{2}^{l(q_{1})} f(q_{1},q_{2}) l(q_{1}q_2)) \end{bmatrix}/\Delta_{\alpha_{K} \alpha_{H}} \\
&= \begin{bmatrix} l(q_{1}q_{2}) \\ k_{1}k_{2}^{l(q_{1})} l(q_{1}q_2)) \end{bmatrix}/\Delta_{\alpha_{K} \alpha_{H}} \\
&= \psi\left(  k_{1}k_{2}^{l(q_{1})} ,q_1 q_2 \right) = \psi \left( ( k_1, q_1 )  \cdot (k_2,q_2) \right).
\end{align*}

(2) This follows from (1) since $G(\alpha_{K})/\Delta_{\alpha_{K} 1} \approx K/[K,G] \rtimes_{\phi} G/G \approx K/[K,G]$.

(3) This is just Theorem~\ref{thm:extension}.

(4) Use (1) and (2) and the fact that $K$ is central.
\end{proof}

\begin{remark}
From Lemma~\ref{lem:grp}(1), for $K$ abelian we see that $G(\alpha_{K})/\Delta_{\alpha_{K} \alpha_{K}}$ is isomorphic to the semidirect product induced by the extension $\pi: G \rightarrow G/K$. The inverse to the map $\psi$ is given by $\sigma: \begin{bmatrix} x \\ y \end{bmatrix}/\Delta_{\alpha_{K} \alpha_{K}} \longmapsto (yx^{-1} , \pi(x))$.
\end{remark}

At this point, it is possible to use Corollary~\ref{cor:1} and the group example in Lemma~\ref{lem:grp} to define analogues of the machinery for abelian $2^{\mathrm{nd}}$-cohomology groups in varieties with a difference term. In \cite{wiresV}, we will observe how to derive the central extension theory as a specialization to those varieties from the more general extension theory of affine datum developed in the next section.

\vspace{0.3cm}

\section{Deconstruction/reconstruction with affine datum}\label{sec:3}

We start be defining the notion of action. For $0 < n \in \mathds{N}$, write $[n] = \{1,\ldots,n\}$ for the initial segment of positive integers and $[n]^{\ast} = \{ s \subseteq \{ 1,\ldots,n \}: 0 < |s| < n \}$ for the non-empty proper subsets of $[n]$. Given sets $Q$ and $A$ and $s \in [n]^{\ast}$, we define the set of coordinates $[Q,A]^{s} = \{ \vec{x} \in Q \cup A : \vec{x}(i) \in A \text{ for } i \in s, \vec{x}(i) \in Q \text{ for } i \not\in s \}$.

\begin{definition}\label{def:pairing}
Let $Q,A$ be sets and $f$ an n-ary operation on $Q$ with $n \geq 2$. A \emph{pairing} of $Q$ on $A$ with respect to $f$ is a choice of subsets $\sigma(f) \subseteq [n]^{\ast}$ and a sequence of functions $a(f,s) : [Q,A]^{s} \rightarrow A$ with $s \in \sigma(f)$.

Let $Q$ be an algebra in the signature $\tau$. For a set $A$, an \emph{action} $Q \ast A$ is a sequence of pairings $\{a(f,s) : f \in \tau, \ar f \geq 2, s \in \sigma(f) \subseteq [\ar f]^{\ast} \}$.
\end{definition}

In the following definition, we demonstrate how an action can be associated to a variety (or equational theory). Given a fixed action $Q \ast A$, for any term $t$ in the signature $\tau$, there are no operations on $A$ available in which to define an interpretation of $t$ as an operation $t^{A}: A^{\ar t} \rightarrow A$; however, via the pairings $a(f,s)$ associated to each $f \in \tau$ with $\ar f \geq 2$, it is sometimes possible to define tuples $\bar{c} \in (Q \cup A)^{\ar t}$ on which to define a value $t(\bar{c}) \in (Q \cup A)$ using compositions of the available operations on $Q$ and pairings. If we consider the composition tree of the term with variables on the leaves, the idea is to choose values in $Q \cup A$ for the variables so that when propagated through the nodes of the tree yields a meaningful evaluation using either the function symbols $f \in \tau$ for the algebra $Q$ or the pairing functions $a(f,s)$.

\begin{definition}\label{def:5}
Let $Q$ be an algebra in the signature $\tau$ and $A$ a set. Let $Q \ast A$ be an action. For each term $f$, we will define a subset of sequences $\bar{a} \in C_{f} \subseteq (Q \cup A)^{\ar f}$ \emph{compatible} with $f$ and an operation $f^{\ast} : C_{f} \rightarrow Q \cup A$. If $f=p_{i}$ is the n-ary i-th projection, then $C_{f} := (Q \cup A)^{n}$ and $f^{\ast}:=p_i$. If $f \in \tau$ with $\ar f < 1$, then define $C_{f} := Q$ and $f^{\ast}(a) := f^{Q}(a)$. If $f \in \tau$ with $\ar f \geq 2$, then define $C_{f} :=  Q^{\ar f} \cup \bigcup_{s \in \sigma(f)} [Q,A]^{s}$ and
\[
f^{\ast}(\bar{a}) := \begin{cases}
   f^{Q}(\bar{a}) & , \bar{a} \in Q^{\ar f} \\
   a(f,s)(\bar{a}) & , \bar{a} \in [Q,A]^{s}
  \end{cases}
\]
for $\bar{a} \in C_{f}$. Let $f(x_1,\ldots,x_m)=h(g_1(x_{11},\ldots,x_{1n_1}),\ldots,g_{m}(x_{m1},\ldots,x_{mn_{m}}))$ where $h \in \tau$, $g_i$ are terms or projections and $C_{g_i}$ are the compatible sequences for $g_{i}$. Then take $\bar{a} \in C_{f}$ provided $\bar{a}_{i} = (a_{i1},\ldots,a_{in_i}) \in C_{g_i}$ for $i=1,\ldots,m$ and $(g^{\ast}_1(\bar{a}_1),\ldots,g^{\ast}_{m}(\bar{a}_m)) \in C_{h}$. Then for $\bar{a} \in C_{f}$, define $f^{\ast}(\bar{a}):= h^{Q}(g^{\ast}_1(\bar{a}_1),\ldots,g^{\ast}_{m}(\bar{a}_m))$ in the case $(g^{\ast}_1(\bar{a}_1),\ldots,g^{\ast}_{m}(\bar{a}_m)) \in Q^{m}$ and $f^{\ast}(\bar{a}):= a(h,s)(g^{\ast}_1(\bar{a}_1),\ldots,g^{\ast}_{m}(\bar{a}_m))$ in the case $(g^{\ast}_1(\bar{a}_1),\ldots,g^{\ast}_{m}(\bar{a}_m)) \in [Q,A]^{s}$ for some $s \in [\ar h]^{\ast}$.

Let $\mathrm{var} \, t$ be the set of variables of the term $t$. For an equation $f(\vec{x})=g(\vec{y})$, a pair $(\vec{a},\vec{b}) \in C_{f} \times C_{g}$ is \emph{appropriate} if there is an assignment $\epsilon : \mathrm{var} \, f \cup \mathrm{var} \, g \rightarrow Q \cup A$ such that $(\epsilon(\vec{x}), \epsilon(\vec{y}) )=(\vec{a},\vec{b})$ and $f^{\ast}(\vec{a}) \in A \Leftrightarrow g^{\ast}(\vec{b}) \in A$.

Let $\Sigma$ be a set of equation in the signature $\tau$. An action $Q \ast A$ is \emph{compatible} with $\Sigma$ if for any equation $f(\bar{x}) \approx g(\bar{y}) \in \Sigma$ and appropriate pair $(\vec{a},\vec{b}) \in C_{f} \times C_{g}$ we have $f^{\ast}(\bar{a}) = g^{\ast}(\bar{b})$. We write $Q \ast A \vDash_{\ast} \Sigma$ if the action $Q \ast A$ is compatible with $\Sigma$.

If $\mathcal V$ is a variety in the same signature $\tau$, then an action $Q \ast A$ is \emph{compatible} with $\mathcal V$ if $Q \ast A \vDash_{\ast} \mathrm{Id} \, \mathcal V$.
\end{definition}

We emphasize that for an action $Q \ast A$ with an algebra $Q$ in the signature $\tau$, the compatible sequences for a term are determined according to the composition tree in the fundamental operations given by $\tau$. Note from Definition~\ref{def:5}, $Q^{\ar f} \subseteq C_{f}$ for a term $f$ and so an action $Q \ast A$ compatible with $\mathcal V$ implies $Q \in \mathcal V$. An action $Q \ast A$ is a \emph{full action} if $\sigma(f) = [\ar f]^{\ast}$ for each $f \in \tau$ with $\ar f \geq 2$ and a \emph{singleton action} if $\sigma(f)$ consists of all singleton sets for each $f \in \tau$.

\begin{example}
Let $G \in \mathcal G$ the variety of groups and $X$ a set. Asserting an action $G \ast X$ is compatible with $\mathcal G$ recovers the classic definition of group-action. For the binary group operation $f$, we write different notation for the actions by $a(f,1)(x,g) = x \circ g$ and $a(f,2)(g,x) = g \cdot x$. For the terms $f(x,f(y,z))$ and $f(f(x,y),z)$, the compatible sequences can be seen to be 
\[
C_{f(f(x,y),z)} = C_{f(x,f(y,z))} = \{(g_{1},g_{2},g_{3}), (g_{1},g_{2},x),(x,g_{1},g_{2}),(q_{1},x,q_{2}) : g_{1},g_{2},g_{3} \in G, x \in X \}.
\]
Then enforcing compatibility with associativity we calculate
\begin{align*}
g_1 \cdot (g_2 \cdot x) = a(f,2)(g_1,a(f,2)(g_2,x)) &= f(x,f(y,z))^{\ast}(g_1,g_2,x) \\
&= f(f(x,y),z)^{\ast}(g_1,g_2,x) = a(f,2)(f^{G}(g_1,g_2),x) = (g_1g_2) \cdot x.
\end{align*}
Similarly, the consequences for compatibility of a full action with the associative equation of $\mathcal G$ yields

\begin{enumerate}

	\item $(g_{1}g_{2})g_{3} = g_{1}(g_{2}g_{3})$ \quad \quad \quad \quad \quad \quad \quad \quad \quad \quad (associativity in the group $G \in \mathcal G$)

	\item $g_1 \cdot (g_2 \cdot x) = (g_1g_2) \cdot x$ \quad \quad \quad \quad \quad \quad \quad \quad \ \  (left-action of $G$ on $X$)
	
	\item $(x \circ g_{1}) \circ g_{2} = x \circ (g_{1}g_{2})$ \quad \quad \quad \quad \quad \quad \quad \quad (right-action of $G$ on $X$)
	
	\item $(g_{1} \cdot x) \circ g_{2} = g_{1} \cdot (x \circ g_{2})$

\end{enumerate}
By choosing an action with $\sigma(f) = \{2\}$ we would recover only consequences (1) and (2) above, or an action with $\sigma(f) = \{1\}$ would recover only (1) and (3) above.
\end{example}

\begin{example}
It is possible for an action to be vacuously compatible with a set of equations in the following sense. Take $\tau = \{f\}$ a single binary symbol and $\Sigma = \{f(x,y) = f(y,x)\}$. For an action $Q \ast A$ such that $\sigma(f) = \{2\}$, we have $C_{f(x,y)} \cap C_{f(y,x)} = Q^2$. If $Q \vDash f(x,y) = f(y,x)$, then the action is compatible with $\Sigma$ independently with how $a(f,1): Q \times A \rightarrow A$ is defined.
\end{example}

In the case of an action $Q \ast A$ on a partial-algebra $A$, there is a modification in the definition of a compatible sequence to reflect the fact that we have partial access to operations on the universe $A$. For $f \in \tau$, the corresponding partial-operation on $A$ may only be defined on a subset $\mathrm{dom} \, f \subseteq A^{\ar f}$; consequently, we define $C_{f} := \mathrm{dom} \, f \cup Q^{\ar f} \cup \bigcup_{s \in \sigma(f)} [Q,A]^{s}$ and
\[
f^{\ast}(\bar{a}) := \begin{cases}
   f^{Q}(\bar{a}) & , \bar{a} \in Q^{\ar f} \\
   a(f,s)(\bar{a}) & , \bar{a} \in [Q,A]^{s} \\
	 f^{A}(\bar{a}) & , \bar{a} \in \mathrm{dom} \, f
  \end{cases}
\]
for $\bar{a} \in C_{f}$. The rest of Definition~\ref{def:5} follows in a similar manner with appropriate allowances made for the partial-operations over $A$.

The notion of action is explicitly concerned with the operations of the signature which are binary or of higher arity. The nullary and unary operations will be encoded into a new structure comprising datum which will play the r\^{o}le of the normal subgroup associated to the kernel of the extension. The information in the new structure will have to compensate for the fact that congruences in general are not determined by a privileged class as is the case of normal subgroups and kernels of group homomorphisms.

\begin{definition}
Fix signature $\tau$ and ternary term symbol $m$. Define $A^{\alpha,\tau} = \left\langle  A,\alpha, \{f^{\Delta} : f \in \tau\} \right\rangle$ where
\begin{itemize}

	\item $A = \left\langle A, m \right\rangle$ is an algebra in the single operation symbol $m$;
	
	\item $\alpha \in \Con A$;
	
	\item $\{f^{\Delta} : f \in \tau\}$ is a sequence of operations $f^{\Delta} : A(\alpha)/\Delta_{\alpha \alpha} \times \delta(A)^{\ar f - 1} \rightarrow A(\alpha)/\Delta_{\alpha \alpha}$ where $\delta: A \rightarrow  A(\alpha)/\Delta_{\alpha \alpha}$ is the diagonal map.

\end{itemize}
\end{definition}

In order to make the treatment of the operations in reconstructed extensions more uniform, the partial operations $f^{\Delta}$ in the above definition are taken over the full signature $\tau$; however, we shall see that we are really concerned with only the nullary and unary operations of $\tau$. For higher arity operation symbols $f \in \tau$, we shall declare $f^{\Delta}$ to agree with action terms $a(f,1)$ and for unary and nullary operation symbols we shall take $f^{\Delta}$ to be independent of all but the first coordinate. This will be made formal in the definition of datum (Definition~\ref{def:datum}).

For a surjective map $\pi : A \rightarrow Q$ with $\alpha = \ker \pi$, there is a surjective map $\rho: A(\alpha) \rightarrow Q$ with $\hat{\alpha} = \ker \rho$. Note that a lifting $l: Q \rightarrow A$ for $\pi$ has the property that $(x,l(q)) \in \alpha \Leftrightarrow \rho\left(\begin{bmatrix} l(q) \\ x \end{bmatrix} \right)=q$. This differs with the notion from group theory in that $l$ is no longer taken as a right-inverse for $\rho$ as a set map, but that $\rho \circ (\delta \circ l) = \id$. The $\alpha$-trace $r: A \rightarrow A$ associated to the lifting is defined by $r = l \circ \rho$.

In analogy with the group case, an important distinction arises when an action $Q \ast K$ operates by an automorphism; that is, there is an induced homomorphism $\gamma: Q \rightarrow \Aut K$ such that $q \ast k = \gamma(q)(k)$ for all $q \in Q, k \in G$. This is what Definition~\ref{def:homact} addresses.

\begin{definition}\label{def:homact}
Fix $A^{\alpha,\tau}$, a ternary term $m$, an algebra $Q$ in the signature $\tau$ and surjective map $\rho : A(\alpha) \rightarrow Q$ with $\hat{\alpha} = \ker \rho$. For any $u \in A$, define the polynomial operation $x +_{u} y := m \left(x, \delta(u), y \right)$. An action $Q \ast A(\alpha)/\Delta_{\alpha \alpha}$ is \emph{homomorphic} with respect to $m$ if the following holds:
\begin{itemize}

	\item for all $f \in \tau$ with $\ar f \geq 2$ and $s \in \sigma(f)$, all $\vec{u} \in A^{s}$ and all $\vec{a},\vec{b} \in \left( A (\alpha)/\Delta_{\alpha \alpha} \right)^{s}$ such that $a_i \mathrel{\hat{\alpha}/\Delta_{\alpha \alpha}} \delta(u_i) \mathrel{\hat{\alpha}/\Delta_{\alpha \alpha}} b_i$ for $i \in s$,
	
	\item if we let $\vec{x},\vec{y},\vec{z} \in [Q, A (\alpha)/\Delta_{\alpha \alpha} ]^{s}$ such that $x_{i} = y_{i} = z_{i} \in Q$ for $i \not\in s$ and $x_{i} = a_{i}$, $y_{i} = b_{i}$ and $z_{i}= a_{i} +_{u_{i}} b_{i}$ for $i \in s$,

\end{itemize}
then we have 
\[
a(f,s)(\vec{z}) = a(f,s)(\vec{x}) +_{v} a(f,s)(\vec{y})
\] 
where $v=l(f^{A/\alpha}(q_1,\ldots,q_{i},u_1,\ldots,u_k))$ for any lifting $l:Q \rightarrow A$ associated to $\rho$.
\end{definition}

An action $Q \ast A$ is \emph{unary} if $\sigma(f) \subseteq [\ar f]^{\ast}$ consists of exactly all the singleton subsets. In the case of unary actions, the homomorphic property is more directly written as 
\[
a(f,i)(q_{1},\ldots,a +_{u} b,\ldots,q_{n}) = a(f,i)(q_{1},\ldots,a,\ldots,q_{n}) +_{v} a(f,i)(q_{1},\ldots,b,\ldots,q_{n})
\]
for $\ar f = n$ and $v = l(f^{A/\alpha}(q_{1},\ldots,u,\ldots,q_{n}))$.

\begin{definition}\label{def:homdatum}
The structure $A^{\alpha,\tau} = \left\langle  A,\alpha, \{f^{\Delta} : f \in \tau\} \right\rangle$ is \emph{homomorphic} with respect to $m$ if for each $f \in \tau$ and $a \mathrel{\hat{\alpha}/\Delta_{\alpha \alpha}} \delta(u) \mathrel{\hat{\alpha}/\Delta_{\alpha \alpha}} b$, $\vec{x} \in A^{\ar f - 1}$ we have 
\[
f^{\Delta} \left( a +_{u} b, \delta(\vec{x}) \right) = f^{\Delta} \left( a , \delta(\vec{x}) \right) +_{v} f^{\Delta} \left( b, \delta(\vec{x}) \right)
\]
where $v=l(f^{A/\alpha}(u, \vec{x}))$ for any lifting $l: A/\alpha \rightarrow A$ associated to $\rho : A(\alpha) \rightarrow Q$.
\end{definition}

\begin{remark}\label{remark: liftings}
Note Definitions~\ref{def:homact} and \ref{def:homdatum} do not depend on the lifting. If $l,l':Q \rightarrow A$ are liftings associated to $\rho$, then $(l(q),l'(q)) \in \alpha$ for all $q \in Q$ which implies $\delta(l(q)) = \delta(l'(q))$ since $\left( \begin{bmatrix} l(q) \\ l(q) \end{bmatrix}, \begin{bmatrix} l'(q) \\ l'(q) \end{bmatrix} \right)$ is a generator of $\Delta_{\alpha \alpha}$.
\end{remark}

When referring to homomorphic actions the choice of the operation $m$ will consistently be chosen to be the same operation given by $A^{\alpha,\tau}$. A ternary operation $m: A^{3} \rightarrow A$ is a \emph{ternary abelian group operation} if it is a Mal'cev operation which is \emph{self-commuting}; that is, 
\begin{align*}
A &\vDash m(x,y,y)=x \mathrel{\wedge} m(y,y,x)=x \\
A &\vDash m(m(x_{1},x_{2},x_{3}),m(y_{1},y_{2},y_{3}),m(z_{1},z_{2},z_{3})) = m(m(x_{1},y_{1},z_{1}),m(x_{2},y_{2},z_{2}),m(x_{3},y_{3},z_{3})).
\end{align*}
A ternary abelian group operation $m$ may also be referred to as an \emph{affine operation} since for any choice of $a \in A$, the definitions $x + y := m(x,a,y)$ and $-x := m(a,x,a)$ yields an abelian group $\left\langle A, +, -, a \right\rangle$ in which $m(x,u,z) = x - y + z$ \cite[Thm 7.34]{bergman}. We now arrive at the definition of datum and what it means for an extension to realize the datum.

\begin{definition}\label{def:datum}
Fix a signature $\tau$. A triple $\left( Q, A^{\alpha,\tau},\ast \right)$ is \emph{datum} if the following holds:
\begin{enumerate}
		
		\item[(D1)] $A^{\alpha,\tau}$ is homomorphic;
		
		\item[(D2)] $Q$ is an algebra in the signature $\tau$;
		
		\item[(D3)] The ternary operation symbol $m$ referenced in $A^{\alpha,\tau}$ also has an interpretation in $Q$ which is idempotent and there is a surjective homomorphism $\rho: A(\alpha) \rightarrow \left\langle Q, m \right\rangle$ such that $\hat{\alpha} = \ker \rho$;
		
		\item[(D4)] The action $Q \ast A(\alpha)/\Delta_{\alpha \alpha}$ is homomorphic with the following property: for any $f \in \tau$ with $n=\ar f \geq 2$ and $\rho \left(\begin{bmatrix} y_i \\ x_i \end{bmatrix} \right) =q_i$ for $i=1,\ldots,n$, we have
		\[
		f^{\Delta}\left(\begin{bmatrix} y_1 \\ x_1 \end{bmatrix}/\Delta_{\alpha \alpha},\delta(x_2),\ldots,\delta(x_n) \right) \mathrel{\hat{\alpha}/\Delta_{\alpha \alpha}} a(f,s)\left( \vec{z} \right)
		\]
for all $s \in \sigma(f)$ and $\vec{z} \in [Q, A(\alpha)/\Delta_{\alpha \alpha}]^{s}$ such that $z_{i} = q_{i}$ for $i \not\in s$ and $z_{i} = \begin{bmatrix} y_i \\ x_i \end{bmatrix}/\Delta_{\alpha \alpha}$ for $i \in s$.

\end{enumerate}
We say $(Q,A^{\alpha,\tau},\ast)$ is \emph{affine datum} if in addition
\begin{enumerate}

	\item[(AD1)] $\alpha \in \Con \left\langle A, m \right\rangle$ is an abelian congruence and $m$ is a ternary abelian group operation when restricted on each $\alpha$-class;
	
	\item[(AD2)] the action is unary and 
		\[
		f^{\Delta}\left(\begin{bmatrix} y_1 \\ x_1 \end{bmatrix}/\Delta_{\alpha \alpha},\delta(x_2),\ldots,\delta(x_n) \right) = a(f,1) \left( \begin{bmatrix} y_1 \\ x_1 \end{bmatrix}/\Delta_{\alpha \alpha}, q_{2},\ldots,q_{n} \right)
		\]
for any $f \in \tau$ with $n = \ar f \geq 2$.

\end{enumerate}
\end{definition}

\begin{remark}
The notion of datum requires comment. In the case of groups, from datum the $(\textbf{K},Q,\phi)$ one can always construct two extensions of $K$ by $Q$: the direct product $p_{2} : K \times Q \rightarrow Q$ with $\alpha_{K} = \ker p_{2}$ and the semidirect product $\pi : K \rtimes_{\phi} Q \rightarrow Q$ with $\alpha_{K} = \ker \pi$. In this way, we can treat the direct product and semidirect product as already explicitly provided by the datum; as a consequence, the datum provides a universe on which to define the algebras for possible extensions. The partial structure $A^{\alpha,\tau}$ includes the underlying set $A$ from which we can calculate the set $A(\alpha)/\Delta_{\alpha \alpha}$ which will serve as the universe of the reconstructed extensions.

It is not required that there is an interpretation of the signature $\tau$ on the set $A$ - only that there is a ternary operation $m$ which is defined on $A$ so that $A(\alpha)/\Delta_{\alpha \alpha}$ is calculated explicitly from the congruence $\alpha$ of the algebra $\left\langle A, m \right\rangle$. The membership in the congruence $\Delta_{\alpha \alpha}$ is then only determined by $m$. The idea is to preserve this particular aspect from the example where $\alpha$ is an abelian congruence of an algebra $A \in \mathcal V$ in a variety with weak-difference term; in this case, the universe $A(\alpha)/\Delta_{\alpha \alpha}$ can be reconstructed from the resulting datum by the rule $\begin{bmatrix} a \\ b \end{bmatrix} \Delta_{\alpha \alpha} \begin{bmatrix} c \\ d \end{bmatrix} \Leftrightarrow d=m(b,a,c)$ since $\Delta_{\alpha \alpha} \leq \hat{\alpha}$. This is data which resides in $A^{\alpha,\tau}$.
\end{remark}

\begin{definition}\label{def:realizes}
Let $(Q,A^{\alpha,\tau},\ast)$ be affine datum. An algebra $B$ \emph{realizes} the datum if there is an extension $\pi: B \rightarrow Q$ with $\beta = \ker \pi$, a bijection $i : B(\beta)/\Delta_{\beta \beta} \rightarrow A(\alpha)/\Delta_{\alpha \alpha}$ and a lifting $l: Q \rightarrow B$ such that for all $f \in \tau$:
\begin{enumerate}

	\item[(R1)] if $n = \ar f \geq 2$, $q_{1},\ldots,q_n \in Q$ and $x_{1},\ldots,x_{n} \in B(\beta)/\Delta_{\beta \beta}$, then
\[
a(f,i)\left(q_1,\ldots, i \circ x_{i},\ldots, q_{n} \right) = i \circ f^{B(\beta )/\Delta_{\beta \beta}} \left( \delta \circ l(q_1) ,\ldots, x_{i},\ldots,\delta \circ l(q_{n}) \right);
\]

	\item[(R2)] if $\ar f \leq 1$ and $x \in A(\alpha)/\Delta_{\alpha \alpha}$, then $f^{\Delta}(i ( x ) ) = i \circ f^{B(\beta )/\Delta_{\beta \beta}}(x)$.

\end{enumerate}
\end{definition}

For special choices of the operations in $T$, the following algebras will serve as our reconstruction of extensions which realize affine datum. This will be shown in Theorem~\ref{thm:12}

\begin{definition}\label{def:9}
Let $(Q,A^{\alpha,\tau},\ast)$ be affine datum in the signature $\tau$. Let $T= \{T_{f}: f \in \tau\}$ a sequence of functions such that $T_{f}: (Q)^{\ar f} \rightarrow A(\alpha)/\Delta_{\alpha \alpha}$. Fix a lifting $l:Q \rightarrow A$ associated to $\rho: A(\alpha) \rightarrow Q$. Define an algebra $A_{T}(Q,A^{\alpha,\tau},\ast)$ on the universe $A(\alpha)/\Delta_{\alpha \alpha}$ of the datum with operations
\begin{align*}
F_{f} &\left(\begin{bmatrix} a_1 \\ b_1 \end{bmatrix}/\Delta_{\alpha \alpha}, \ldots, \begin{bmatrix} a_n \\ b_n \end{bmatrix}/\Delta_{\alpha \alpha} \right) := f^{\Delta}\left(\begin{bmatrix} a_1 \\ b_1 \end{bmatrix}/\Delta_{\alpha \alpha},\delta(b_2),\ldots, \delta(b_n) \right)\\
&+_{u} \ \sum_{i=2}^{n} a(f,i)\left(q_1,\ldots,q_{i}, \begin{bmatrix} a_{i+1} \\ b_{i+1} \end{bmatrix}/\Delta_{\alpha \alpha}, \delta(b_{i+2}),\ldots,\delta(b_n) \right) \\
&+_{u} \ T_{f}(q_1,\ldots,q_n) \hspace{8.5cm} (f \in \tau).
\end{align*}
a left-associated composition where $u= l(f^{Q}(q_1,\ldots,q_n))$ and $\rho \left(\begin{bmatrix} a_i \\ b_i \end{bmatrix} \right)=q_i$.
\end{definition}

Definition~\ref{def:5}, Definition~\ref{def:tranequ} and Definition~\ref{def:2cocyle} will constitute our scheme which given membership $Q \in \mathcal U$ in a variety will guarantee the extension $A_{T}(Q,A^{\alpha,\tau},\ast) \in \mathcal U$:
\begin{align*}
\text{ identities in } \mathcal U &\Rightarrow \text{ expressions in } \{T_{f}: f \in \tau\} \text{ and action terms } Q \ast A(\alpha)/\Delta_{\alpha \alpha} \\
&\Rightarrow \text{ identities satisfied by } A_{T}(Q,A^{\alpha,\tau},\ast).
\end{align*}
Given a term $t(\bar{x})$ in the signature $\tau$, we will be interested in its interpretation in the algebra $A_{T}(Q,A^{\alpha,\tau},\ast)$. By using the homomorphism property of the action terms and compatibility of the ternary operation $m$ over the blocks of $\hat{\alpha}/\Delta_{\alpha \alpha}$, we can distribute the sum in Definition~\ref{def:9} at every point in the composition tree of $t^{A_{T}(Q,A^{\alpha,\tau},\ast)}$. The end result is a large sum in mixed action terms and functions in $T$ indexed by operation symbols of $\tau$. The next definition will describe a way to separate the terms in the sum which use some operations from $T$.

\begin{definition}\label{def:tranequ}
Let $(Q,A^{\alpha,\tau},\ast)$ be affine datum in the signature $\tau$ and $T= \{T_{f}: f \in \tau\}$ a sequence of functions $T_{f}: (Q)^{\ar f} \rightarrow A(\alpha)/\Delta_{\alpha \alpha}$. Let $m$ be the ternary operation on $A$ included in the datum and $x +_{u} y = m(x,\delta(u),y)$ the binary operation defined on $A(\alpha)/\Delta_{\alpha \alpha}$ for $u \in A$.

For each term $t$ in the signature $\tau$, we inductively define along the composition tree of the term a set $E_{t}$ of operations of the form $\nu : Q^{\ar t} \rightarrow A(\alpha)/\Delta_{\alpha \alpha}$. Then we define an operation $t^{\partial,T}: Q^{\ar t} \rightarrow A(\alpha)/\Delta_{\alpha \alpha}$ given by 
\begin{align}\label{eq:cocycleident}
t^{\partial,T}(\vec{q}) = \sum_{\nu \in E_{t}} \nu (\vec{q})
\end{align}
where the sum is the left-associated composition over $+_{u}$ where $u = l \left( t^{Q}(\vec{q}) \right)$ for any lifting $l:Q \rightarrow A$ associated to $\rho: A(\alpha) \rightarrow Q$. By Remark~\ref{remark: liftings}, the definition does not depend on the choice of the lifting $l$.

If $t=f \in \tau$, define $E_{t} = \{ T_{f} \}$. Suppose $t(\vec{y}) = f(\sigma(\vec{x}))$ is derived from a fundamental operation $f \in \tau$ by identifying or permuting variables according to the surjective map $\sigma : \mathrm{var} \ f \rightarrow \mathrm{var} \ t $. Define $E_{t} = \{T_{f} \circ \sigma \}$ where $T_{f} \circ \sigma : Q^{\ar t} \rightarrow A(\alpha)/\Delta_{\alpha \alpha}$ denotes the corresponding operation evaluated according to $\sigma$; that is, for each evaluation $\epsilon : \mathrm{var} \ t \rightarrow Q$ there is an evaluation $\epsilon' : \mathrm{var} \ f \rightarrow Q$ such that $\epsilon'(\vec{x}) = \epsilon \circ \sigma (\vec{x})$ and $(T_{f} \circ \sigma )(\epsilon'(\vec{y})) = T_{f}(\epsilon(\vec{x}))$.

Suppose $t(y_1,\ldots,y_n)=f(g_{1}(\bar{z}_{1}),\ldots,g_{m}(\bar{z}_{m}))$ where $f \in \tau$ with $\ar f=m$ and each $g_i(\bar{z}_i)$ is a term or variable. The set $E_{t}$ will consist of three different types of operations. For the first type, we take 
\[
\nu (\vec{q}) : = T_{f} \left( g^{Q}_{1}(\bar{q}_{1}),\ldots,g^{Q}_{m}(\bar{q}_{m}) \right)
\] 
where $\bar{q}_{i} \in Q^{\ar g_{i}}$ is the restriction of $\vec{q}$ to the variables of $g_{i}$. For the second type, it may be that $g_{1}$ is not a variable. We take operations of the form
\[
\nu (\vec{q}) := f^{\Delta} \left( \mu (\bar{q}_{1}),\delta \circ l (g^{Q}_{2}(\bar{q}_{2})),\ldots,\delta \circ l(g^{Q}_{m}(\bar{q}_{m})) \right).
\]
where $\mu \in E_{g_{1}}$ and $\vec{q}_{i}$ is the restriction of $\vec{q}$ to the variables of $g_{i}$. For the third type, for any $2 \leq k \leq m$ such that $g_{k}$ is not a variable, we take operations of the form
\[
\nu(\vec{q}) := a(f,k) \left(  g^{Q}_{1}(\bar{q}_1),\ldots, g^{Q}_{k-1}(\bar{q}_{k-1}), \mu(\vec{q}_{k}), g^{Q}_{k+1}(\delta(\bar{a}_{k+1})),\ldots, g^{Q}_{m}(\delta(\bar{a}_{m}))  \right)
\]
where $\mu \in E_{g_{k}}$ and $\vec{q}_{i}$ is the restriction of $\vec{q}$ to the variables of $g_{i}$.
\end{definition}

\begin{definition}\label{def:2cocyle}
Let $(Q,A^{\alpha,\tau},\ast)$ be affine datum and $\Sigma$ a set of equations in the same signature $\tau$. A sequence of operations $T=\{T_{f}: f \in \tau\}$ where $T_{f}: Q^{\ar f} \rightarrow A(\alpha)/\Delta_{\alpha \alpha}$ is a \emph{2-cocycle compatible with} $\Sigma$ if
\begin{enumerate}

	\item[(C1)] for all $f \in\tau$, $T_{f}(q_1,\ldots,q_{\ar f}) \mathrel{\hat{\alpha}/\Delta_{\alpha \alpha}} f^{\Delta} \left( \begin{bmatrix} a \\ b \end{bmatrix}/\Delta_{\alpha \alpha}, \delta \circ l (q_2),\ldots, \delta \circ l(q_{\ar f}) \right)$ where $\rho \left( \begin{bmatrix} a \\ b \end{bmatrix} \right) = q_{1}$;
	
	\item[(C2)] for all $t(\vec{x})=g(\vec{y}) \in \Sigma$ and evaluations $\epsilon : \mathrm{var} \, t \cup \mathrm{var} \, g \rightarrow Q$ we have 
	\[
	t^{\partial,T}(\epsilon(\vec{x})) = g^{\partial,T}(\epsilon(\vec{y})).
	\]

\end{enumerate}
\end{definition}

It follows immediately from the definition, that if $T$ is a 2-cocycle compatible with $\mathrm{Id} \, \mathcal V$, then $T$ is compatible for any variety $\mathcal U \geq \mathcal V$.

\begin{example}
Let $(Q,A^{\alpha,\tau},\ast)$ be affine datum in the signature $\tau$ and $l: Q \rightarrow A$ a lifting associated to $\rho: A(\alpha) \rightarrow Q$. For a single binary operation symbol $f$, enforcing compatibility with associativity yields
\begin{align}\label{eq:10}
f^{\Delta} \left(T_{f}(q_1,q_2), \delta \circ l(q_3) \right) +_{u} T_{f}(f^{Q}(q_1,q_2),q_3) &= a(f,1)\left( q_1,T_{f}(q_2,q_3) \right) +_{v} T_{f}(q_1,f^{Q}(q_2,q_3))
\end{align}
where $u = l \left( f(f(q_1,q_2),q_3) \right)$ and $v= l \left( f(q_1,f(q_2,q_3) \right)$.

In the case of an abelian normal subgroup $K \triangleleft G$ with $Q = G/K$, this specializes to the classic 2-cocycle identity in the following manner. By Lemma~\ref{lem:grp}, there is the isomorphism $\psi: G(\alpha_{K})/\Delta_{\alpha_{K} \alpha_{K}} \ni \begin{bmatrix} b \\ a \end{bmatrix}/\Delta_{\alpha_{K} \alpha_{K}} \longmapsto \left\langle  ab^{-1}, \pi(b) \right\rangle \in K \rtimes_{\phi} Q$. If we fix a lifting $l: Q \rightarrow G$ for the canonical surjection $\pi: G \rightarrow Q$, then this means each $\Delta_{\alpha_{K} \alpha_{K}}$-class is uniquely represented as $\begin{bmatrix} l(x) \\ kl(x) \end{bmatrix}/\Delta_{\alpha_{K} \alpha_{K}}$ for some $k \in K$, $x \in Q$. For the binary operation we set $T(x,y):=\begin{bmatrix} l(xy) \\ l(x)l(y) \end{bmatrix}/\Delta_{\alpha_{K} \alpha_{K}}$ and define action terms by 
\begin{align*}
a(\cdot,1) \left( \begin{bmatrix} l(y) \\ kl(y) \end{bmatrix}/\Delta_{\alpha_{K} \alpha_{K}},x \right) &:= \begin{bmatrix} l(y) \\ kl(y) \end{bmatrix}/\Delta_{\alpha_{K} \alpha_{K}} \cdot \delta \circ l(x) = \begin{bmatrix} l(y)l(x) \\ kl(y)l(x) \end{bmatrix}/\Delta_{\alpha_{K} \alpha_{K}} \longmapsto \left\langle  k, l(x)l(y) \right\rangle \\
a(\cdot,2) \left( x,\begin{bmatrix} l(y) \\ kl(y) \end{bmatrix}/\Delta_{\alpha_{K} \alpha_{K}} \right) &:= \delta \circ l(x) \cdot \begin{bmatrix} l(y) \\ kl(y) \end{bmatrix}/\Delta_{\alpha_{K} \alpha_{K}} = \begin{bmatrix} l(x)l(y) \\ l(x)kl(y) \end{bmatrix}/\Delta_{\alpha_{K} \alpha_{K}} \longmapsto \left\langle x \ast k, l(x)l(y) \right\rangle
\end{align*}
where $x \ast k = l(x)kl(x)^{-1}$ is the conjugation action. Recall, the group 2-cocyle is defined by $l(x)l(y)=f(x,y)l(xy)$.

Then using $v=l(xyz)=u$ and $m(x,y,z)=xy^{-1}z$ we apply $\psi$ to calculate
\begin{align*}
T(x,y) \cdot \delta(l(z)) +_{u} T(xy,z) &=  \left(\begin{bmatrix} l(xy) \\ l(x)l(y) \end{bmatrix} \begin{bmatrix} l(z) \\ l(z) \end{bmatrix} \right)/\Delta_{\alpha_{K} \alpha_{K}} +_{u}  \begin{bmatrix} l(xyz) \\ l(xy)l(z) \end{bmatrix}/\Delta_{\alpha_{K} \alpha_{K}}  \\
&= m \left(\begin{bmatrix} l(xy)l(z) \\ l(x)l(y)l(z) \end{bmatrix}, \begin{bmatrix} l(xyz) \\ l(xyz) \end{bmatrix} , \begin{bmatrix} l(xyz) \\ l(xy)l(z) \end{bmatrix} \right)/\Delta_{\alpha_{K} \alpha_{K}} \\
&= m \left(\begin{bmatrix} l(xy)l(z) \\ f(x,y)l(xy)l(z) \end{bmatrix}, \begin{bmatrix} l(xyz) \\ l(xyz) \end{bmatrix} , \begin{bmatrix} l(xyz) \\ f(xy,z)l(xyz) \end{bmatrix} \right)/\Delta_{\alpha_{K} \alpha_{K}} \\
&\longmapsto m \left( \left\langle  f(x,y),l(xyz) \right\rangle, \left\langle 0,l(xyz) \right\rangle, \left\langle f(xy,z),l(xyz) \right\rangle \right) \\
&= \left\langle  f(x,y),l(xyz) \right\rangle \cdot \left\langle  0,l(xyz) \right\rangle^{-1} \cdot \left\langle  f(xy,z),l(xyz) \right\rangle \\
&= \left\langle  f(x,y) + f(xy,z), l(xyz) \right\rangle
\end{align*}
and
\begin{align*}
a(\cdot,1)(x,T(y,z)) +_{v} T(x,yz) &= \begin{bmatrix} l(x)l(yz) \\ l(x)l(y)l(z) \end{bmatrix}/\Delta_{\alpha_{K} \alpha_{K}} +_{v} \begin{bmatrix} l(xyz) \\ l(x)l(yz) \end{bmatrix}/\Delta_{\alpha_{K} \alpha_{K}} \\
&= m \left(\begin{bmatrix} l(x)l(yz) \\ l(x)f(y,z)l(yz) \end{bmatrix}, \begin{bmatrix} l(xyz) \\ l(xyz) \end{bmatrix} , \begin{bmatrix} l(xyz) \\ f(x,yz)l(xyz) \end{bmatrix} \right)/\Delta_{\alpha_{K} \alpha_{K}} \\
&\longmapsto m \left( \left\langle  x \ast f(y,z),l(xyz) \right\rangle,\left\langle  0,l(xyz) \right\rangle, \left\langle f(x,yz),l(xyz) \right\rangle \right) \\
&= \left\langle  x \ast f(y,z),l(xyz) \right\rangle \cdot \left\langle  0,l(xyz) \right\rangle^{-1} \cdot \left\langle  f(x,yz),l(xyz) \right\rangle \\
&= \left\langle  x \ast f(y,z) + f(x,yz),l(xyz) \right\rangle.
\end{align*}
The equality in Eq.(\ref{eq:10}) and the above calculations yields
\[
f(x,y) + f(xy,z) = x \ast f(y,z) + f(x,yz)
\]
which is the group theoretic 2-cocycle identity.
\end{example}

We would like to consider the interpretation of a term in an algebra $A_{T}(Q,A^{\alpha,\tau},\ast)$ defined from affine datum. Inductively along the composition tree of the term, we can use the homomorphism property of the action and datum operations to distribute across the operations in Definition~\ref{def:9} for each fundamental symbol of the signature. While a term interpreted in the algebra $A_{T}(Q,A^{\alpha,\tau},\ast)$ with domain $A(\alpha)/\Delta_{\alpha \alpha}$ may have identified variables in different operations symbols in its composition tree, the interpretation of the operation symbols in Definition~\ref{def:9} depends on all the coordinates and have domains over $Q \cup \delta(A) \cup A(\alpha)/\Delta_{\alpha \alpha}$. So while the repeated variable $x$ in the term has a fixed evaluation $x \mapsto \begin{bmatrix} a \\ b \end{bmatrix}/\Delta_{\alpha \alpha}$, the interpretation of the term is expanded into sums of operations in which x has evaluations among the values $\left\{ \rho \left(  \begin{bmatrix} a \\ b \end{bmatrix} \right) ,\delta \circ l \circ \rho \left(  \begin{bmatrix} a \\ b \end{bmatrix} \right), \begin{bmatrix} a \\ b \end{bmatrix}/\Delta_{\alpha \alpha} \right\}$ depending in which coordinates the variable $x$ appears.

The next definitions relate the two different domains together. The first step is given a term, we produce a corresponding term with the same composition tree except for no repeated variables.

\begin{definition}\label{def:sigmaterms}
Let $f(\vec{x})$ be a term in the signature $\tau$. Let $f^{\sigma}$ be a term in the signature $\tau$ and variables $X_{\omega} = \{x_0,x_1,\ldots \}$ which has the same composition tree as $f$ except the leaves have no repeated variables; in addition, reading from left-to-right the variables of $f^{\sigma}$ are an initial segment of $X_{\omega}$. There is a surjective map $\sigma: \mathrm{var} \, f^{\sigma} \rightarrow \mathrm{var} \, f$ such that $f^{\sigma}(\sigma(\mathrm{var} \, f^{\sigma})) = f(\bar{x})$.
\end{definition}

Note for any evaluation $\epsilon : \mathrm{var} \, f \rightarrow A$, there is a corresponding evaluation $\epsilon^{\sigma} : \mathrm{var} \, f^{\sigma} \rightarrow A$ such that $\epsilon \circ \sigma = \epsilon^{\sigma}$.

Fix $\rho : A(\alpha) \rightarrow Q$ associated with affine datum $(Q,A^{\alpha,\tau},\ast)$ and a lifting $l: Q \rightarrow A$. Let $t(\vec{x})$ be a term in the same signature with $n=| \mathrm{var} \, f^{\sigma} |$ and $\epsilon : \mathrm{var} \, f \rightarrow A(\alpha)/\Delta_{\alpha \alpha}$ an evaluation. An evaluation $\mu : \mathrm{var} \, f^{\sigma} \rightarrow  Q \cup \delta(A) \cup A(\alpha)/\Delta_{\alpha \alpha}$ is \emph{consistent} with $\epsilon (\vec{x})$ if $\nu(x_{i}) \in \left\{ \rho \circ \epsilon^{\sigma}(x_{i}) ,\delta \circ l \circ \rho \circ \epsilon^{\sigma}(x_{i}), \epsilon^{\sigma}(x_{i}) \right\}$ for each $x_{i} \in \mathrm{var} \, f^{\sigma}$. Define $L(f,\epsilon(\vec{x})) = \{\mu ( \mathrm{var} \, f^{\sigma} ) \in C_{t^{\sigma}}: \mu \text{ is consistent with } \epsilon(\vec{x}) \}$. This is the set of evaluations which will allow us to describe equations in semidirect products realizing affine datum.

\begin{definition}
Let $(Q,A^{\alpha,\tau},\ast)$ be affine datum and $\Sigma$ a set of identities in the same signature. The action in the datum is \emph{weakly compatible} with $\Sigma$ if for all $f = g \in \Sigma$, 
\[
 \sum_{\mu \in L(f,\epsilon(\vec{x}))} (f^{\sigma})^{\ast}(\mu ( \mathrm{var} \, f^{\sigma} )) = \sum_{\mu \in L(g,\epsilon(\vec{x}))} (g^{\sigma})^{\ast}(\mu ( \mathrm{var} \, g^{\sigma} )). 
\]
The action is \emph{weakly compatible} with a variety $\mathcal V$ if it is weakly compatible with $\mathrm{Id} \, \mathcal V$.
\end{definition}

\begin{lemma}\label{lem:repterm}
Let $(Q,A^{\alpha,\tau},\ast)$ be affine datum in the signature $\tau$ and $t(\vec{x})$ a term. For any evaluation $\epsilon (\vec{x}) = \left( \begin{bmatrix} a_{1} \\ b_{1} \end{bmatrix}/\Delta_{\alpha \alpha}, \ldots, \begin{bmatrix} a_{n} \\ b_{n} \end{bmatrix}/\Delta_{\alpha \alpha} \right)$, the interpretation of the term $t$ in the algebra $A_{T}(Q,A^{\alpha,\tau},\ast)$ is represented by
\begin{align}
F_{t} \left( \begin{bmatrix} a_{1} \\ b_{1} \end{bmatrix}/\Delta_{\alpha \alpha}, \ldots, \begin{bmatrix} a_{n} \\ b_{n} \end{bmatrix}/\Delta_{\alpha \alpha} \right) =\sum_{\mu \in L(t,\epsilon(\vec{x}))} (t^{\sigma})^{\ast}(\mu ( \mathrm{var} \, t^{\sigma} ))  \ +_{u} \ t^{\partial, T}\left( \rho \circ \epsilon (\vec{x}) \right).
\end{align}
where $u = l \left( t^{Q}(\rho \circ \epsilon (\vec{x})) \right)$ for any lifting $l$ associated to the datum.
\end{lemma}
\begin{proof}
By induction on the composition tree of $t$. This is precisely what Definition~\ref{def:5}, Definition~\ref{def:tranequ} and Deifnition~\ref{def:sigmaterms} accomplish.
\end{proof}

The next theorem guarantees abelian congruences in varieties with a weak-difference term are sources of affine datum by decomposing an extension with an abelian kernel into appropriate datum. It will be convenient to work in terms of $\alpha$-traces rather than liftings associated to datum.

\begin{theorem}\label{thm:10}
Let $\mathcal V$ be a variety with a weak-difference term in the signature $\tau$. Let $A \in \mathcal V$ and surjective $\pi : A \rightarrow Q$ with $\alpha = \ker \pi \in \Con A$ abelian. Then there exists homomorphic $A^{\alpha,\tau}$, a 2-cocycle $T= \{T_{f}: f \in \tau\}$ compatible with $\mathrm{Id} \, \mathcal V$ and homomorphic action $Q \ast A(\alpha)/\Delta_{\alpha \alpha}$ compatible with $\mathrm{Id} \, \mathcal V$ constituting affine datum such that $A \approx A_{T}(Q,A^{\alpha,\tau},\ast)$.
\end{theorem}
\begin{proof}
We have $Q \approx A/\alpha$. Fix a lifting $l: Q \rightarrow A$ and associated $\alpha$-trace $r: A \rightarrow A$ so that $r = l \circ \pi$ and $\pi \circ l = \id$. We also see that $\rho: A(\alpha) \rightarrow Q$ defined by $\rho: \begin{bmatrix} a \\ b \end{bmatrix} \mapsto \pi(a)$ is a surjective homomorphism such that $\hat{\alpha} = \ker \rho$.

Define $\phi: A \rightarrow A_{T}(Q,A^{\alpha,\tau},\ast)$ by $\phi(x)= \begin{bmatrix} r(x) \\ x \end{bmatrix}/\Delta_{\alpha \alpha}$. Let $m$ be a weak-difference term of $\mathcal V$. Since $\alpha$ is abelian, $m$ is affine on $\alpha$-blocks. Note $\Delta_{\alpha \alpha} \leq \hat{\alpha}$. By Lemma~\ref{lem:20}(2), we see that $\begin{bmatrix} a \\ b \end{bmatrix} \Delta_{\alpha \alpha} \begin{bmatrix} c \\ d \end{bmatrix} \Leftrightarrow d=m(b,a,c)$; therefore, the universe of the algebra $A(\alpha)/\Delta_{\alpha \alpha}$ is uniquely reconstructed from $m$. It also follows that $\begin{bmatrix} a \\ b \end{bmatrix} \Delta_{\alpha \alpha} \begin{bmatrix} a \\ d \end{bmatrix} \Rightarrow b=d$. These facts guarantee that $\phi$ is bijective.

For each $f \in \tau$, define $T_{f}: (Q)^{\ar f} \rightarrow A(\alpha)/\Delta_{\alpha \alpha}$ by
\begin{align}
T_{f}(y_1,\ldots,y_{n}):= \begin{bmatrix} r(f(x_1,\ldots,x_n)) \\ f(r(x_1),\ldots,r(x_{n})) \end{bmatrix}/\Delta_{\alpha \alpha}
\end{align}
for any choice of $x_i \in A$ such that $\pi(x_i)=y_i \in Q$. Since $r(x)=r(y)$ if and only if $(x,y) \in \alpha$, $T_{f}$ is well-defined; in this way, we can also treat $T$ has a function with domain $A$ which depends only on the $\alpha$-classes. For each $f \in \tau$ and $1 \leq i \leq n=\ar f$, define the action $Q \ast A(\alpha)/\Delta_{\alpha \alpha}$ according to the rule
\begin{align}\label{eq:actiondef}
a(f,i)(q_{1},\ldots,q_{i-1},x,q_{i+1},\ldots,q_{n}) := f \left( \delta \circ l(q_1),\ldots,\delta \circ l(q_{i-1}),x,\delta \circ l(q_{i+1}),\ldots,\delta \circ l(q_{n}) \right).
\end{align}
for all $x \in A(\alpha)/\Delta_{\alpha \alpha}$, $q_1,\ldots,q_n \in Q$ and define
\begin{align}
f^{\Delta}(x,\delta(a_2),\ldots,\delta(a_{n})):= f(x_1,\delta(a_2),\ldots,\delta(a_{n}))
\end{align}
for all $x \in A(\alpha)/\Delta_{\alpha \alpha}$, $a_2,\ldots,a_n \in A$. It follows that the action and $A^{\alpha,\tau}$ are homomorphic since $\alpha$ is abelian and $m$ is Mal'cev on $\alpha$-blocks. The definitions also show that (AD2) from Definition~\ref{def:datum} is satisfied.

The fact that $\phi$ is a homomorphism is a result of the following expansion: if we first set $u_i = f(r(x_1),\ldots,r(x_i),x_{i+1},\ldots,x_n)$ we have
\begin{align*}
\begin{bmatrix} r(f(x_1,\ldots,x_n)) \\ f(x_1,\ldots,x_n) \end{bmatrix}/\Delta_{\alpha \alpha} &= \begin{bmatrix} m \big( f(r(x_1),\ldots,r(x_n)),f(r(x_1),\ldots,r(x_n)), r(f(x_1,\ldots,x_n)) \big) \\ m \big( f(x_1,\ldots,x_n), f(r(x_1),\ldots,r(x_n)),f(r(x_1),\ldots,r(x_n)) \big) \end{bmatrix}/\Delta_{\alpha \alpha} \\
&= \begin{bmatrix} f(r(x_1),\ldots,r(x_n)) \\ f(x_1,\ldots,x_n) \end{bmatrix}/\Delta_{\alpha \alpha} +_{u_n} \begin{bmatrix} r(f(x_1,\ldots,x_n)) \\ f(r(x_1),\ldots,r(x_n)) \end{bmatrix}/\Delta_{\alpha \alpha} \displaybreak[0]\\
&= \begin{bmatrix} m \big( f(r(x_1),x_2,\ldots,x_n), f(r(x_1),x_2,\ldots,x_n), f(r(x_1),\ldots,r(x_n)) \big) \\ m \big( f(x_1,\ldots,x_n), f(r(x_1),x_2,\ldots,x_n), f(r(x_1),x_2,\ldots,x_n) \big)  \end{bmatrix}/\Delta_{\alpha \alpha}  \displaybreak[0]\\
& \quad +_{u_n} T_{f}(\pi(x_1),\ldots,\pi(x_n)) \displaybreak[0]\\
&= \begin{bmatrix} f(r(x_1),x_2,\ldots,x_n) \\ f(x_1,x_2,\ldots,x_n) \end{bmatrix}/\Delta_{\alpha \alpha} +_{u_1} \begin{bmatrix} f(r(x_1),r(x_2),\ldots,r(x_n)) \\ f(r(x_1),x_2,\ldots,x_n) \end{bmatrix}/\Delta_{\alpha \alpha}  \displaybreak[0]\\
& \quad +_{u_n} T_{f}(\pi(x_1),\ldots,\pi(x_n)) \displaybreak[0]\\
&= f\left( \begin{bmatrix} r(x_1) \\ x_1 \end{bmatrix}/\Delta_{\alpha \alpha}, \delta(x_2),\ldots,\delta(x_n) \right) +_{u_1} \begin{bmatrix} f(r(x_1),r(x_2),x_3,\ldots,x_n) \\ f(r(x_1),x_2,x_3,\ldots,x_n) \end{bmatrix}/\Delta_{\alpha \alpha}  \displaybreak[0]\\
& \quad +_{u_2} \begin{bmatrix} f(r(x_1),r(x_2),r(x_3),\ldots,r(x_n)) \\ f(r(x_1),r(x_2),x_3,\ldots,x_n) \end{bmatrix}/\Delta_{\alpha \alpha} +_{u_n} T_{f}(\pi(x_1),\ldots,\pi(x_n)) \displaybreak[0]\\
&= f\left( \begin{bmatrix} r(x_1) \\ x_1 \end{bmatrix}/\Delta_{\alpha \alpha}, \delta(x_2),\ldots,\delta(x_n) \right) \displaybreak[0]\\
& \quad +_{u_1} a(f,1)\left( \pi(x_1),\begin{bmatrix} r(x_2) \\ x_2 \end{bmatrix}/\Delta_{\alpha \alpha}, \delta(x_3),\ldots,\delta(x_n) \right) \displaybreak[0]\\
& \quad +_{u_2} \begin{bmatrix} f(r(x_1),r(x_2),r(x_3),\ldots,r(x_n)) \\ f(r(x_1),r(x_2),x_3,\ldots,x_n) \end{bmatrix}/\Delta_{\alpha \alpha} +_{u_n} T_{f}(\pi(x_1),\ldots,\pi(x_n)) \displaybreak[0]\\
&\vdots \displaybreak[0]\\
&= f^{\Delta}\left( \begin{bmatrix} r(x_1) \\ x_1 \end{bmatrix}/\Delta_{\alpha \alpha}, \delta(x_2),\ldots,\delta(x_n) \right) +_{u_n} \displaybreak[0]\\
& \quad \sum_{i=2}^{n} a(f,i)\left(\pi(x_1),\ldots,\pi(x_i),\begin{bmatrix} r(x_i) \\ x_i \end{bmatrix}/\Delta_{\alpha \alpha},\delta(x_{i+2}),\ldots,\delta(x_n) \right) \displaybreak[0]\\
& \quad +_{u_n} T_{f}(\pi(x_1),\ldots,\pi(x_n)) \displaybreak[0]\\
&= F_{f}\left( \begin{bmatrix} r(x_1) \\ x_1 \end{bmatrix}, \ldots, \begin{bmatrix} r(x_n) \\ x_n \end{bmatrix} \right)
\end{align*}
since $\delta(u_i)=\delta(u_j)=\delta(l(f(\pi(x_1),\ldots,\pi(x_n))))$ for $i \neq j$.

We now show both the action and 2-cocycle $T$ are compatible with $\mathrm{Id} \, \mathcal V$. Note that in the expansion for $f \in \tau$ previously calculated, 
\begin{align}\label{eq:13}
F_{f}\left( \begin{bmatrix} r(x_1) \\ x_1 \end{bmatrix}/\Delta_{\alpha \alpha}, \ldots, \begin{bmatrix} r(x_n) \\ x_n \end{bmatrix}/\Delta_{\alpha \alpha} \right) &= \begin{bmatrix} r(f(x_1,\ldots,x_n)) \\ f(x_1,\ldots,x_n) \end{bmatrix}/\Delta_{\alpha \alpha}  \\
&= \begin{bmatrix} f(r(x_1),\ldots,r(x_n)) \\ f(x_1,\ldots,x_n) \end{bmatrix}/\Delta_{\alpha \alpha} +_{u} \ T_{f}(\pi(x_1),\ldots,\pi(x_n)) \\
&= f\left( \begin{bmatrix} r(x_1) \\ x_1 \end{bmatrix},\ldots, \begin{bmatrix} r(x_n) \\ x_n \end{bmatrix} \right)/\Delta_{\alpha \alpha}  +_{u} \ T_{f}(\pi(x_1),\ldots,\pi(x_n)) \label{eq:14}
\end{align}
in the last line Eq.(\ref{eq:14}) the $\Delta_{\alpha \alpha}$-class of the first term is expanded using the action. This is a reflection of the fact that in the algebra $A(\alpha)/\Delta_{\alpha \alpha}$ the action represents a ``twisting'' of the product structure.

Take $t(\bar{x}) = g(\bar{y}) \in \mathrm{Id} \, \mathcal V$. Fix an assignment $\epsilon: \mathrm{var} \ t \cup \mathrm{var} \ g \rightarrow A$. By the isomorphism $\phi$ we have $t^{A_{T}(Q,A^{\alpha,\tau},\ast)}(\phi \circ \epsilon (\bar{x})) = g^{A_{T}(Q,A^{\alpha,\tau},\ast)}(\phi \circ \epsilon (\bar{y}))$ since $A \in \mathcal V$. We can use Eq.(\ref{eq:13}), Eq.(\ref{eq:14}) and the homomorphism property of the action to recursively expand the interpretation of the term $t$ in the algebra $A_{T}(Q,A^{\alpha,\tau},\ast)$ in order to write
\begin{align}
t^{A_{T}(Q,A^{\alpha,\tau},\ast)}(\phi \circ \epsilon (\vec{x})) &= t^{A(\alpha)/\Delta_{\alpha \alpha}}\left(\phi \circ \epsilon (\vec{x}) \right) \ +_{u} \ t^{\partial,T}(\pi \circ \epsilon (\bar{x})) \\
&= t^{A(\alpha)/\Delta_{\alpha \alpha}}\left( \begin{bmatrix} r(a_1) \\ a_1 \end{bmatrix}, \ldots, \begin{bmatrix} r(a_n) \\ a_{n} \end{bmatrix} \right)/\Delta_{\alpha \alpha} \ +_{u} \ t^{\partial,T}(\pi \circ \epsilon (\vec{x})) \label{eq:15}
\end{align}
where $u=l(t^{Q}(\pi \circ \epsilon(\vec{x})))$ and $\left( \begin{bmatrix} r(a_1) \\ a_1 \end{bmatrix}/\Delta_{\alpha \alpha}, \ldots, \begin{bmatrix} r(a_n) \\ a_{n} \end{bmatrix}/\Delta_{\alpha \alpha} \right) = \phi \circ \epsilon (\vec{x})$. The second term in Eq.(\ref{eq:15}) incorporates all the appearances of the transfers $T_{f}$. By comparison with Lemma~\ref{lem:repterm}, we see that
\begin{align*}
 t^{A(\alpha)/\Delta_{\alpha \alpha}}\left( \begin{bmatrix} r(a_1) \\ a_1 \end{bmatrix}, \ldots, \begin{bmatrix} r(a_n) \\ a_{n} \end{bmatrix} \right)/\Delta_{\alpha \alpha}  = \sum_{\mu \in L(t,\epsilon(\vec{x}))} (t^{\sigma})^{\ast}( \mu ( \mathrm{var} \, t^{\sigma}) ).
\end{align*}
A similar calculation produces
\begin{align*}
g^{A_{T}(Q,A^{\alpha,\tau},\ast)}(\phi \circ \epsilon (\bar{y})) &= g^{A(\alpha)/\Delta_{\alpha \alpha}}\left( \begin{bmatrix} r(b_1) \\ b_1 \end{bmatrix}, \ldots, \begin{bmatrix} r(b_m) \\ b_{m} \end{bmatrix} \right)/\Delta_{\alpha \alpha} \ +_{v} \ g^{\partial,T}(\epsilon (\pi \circ \vec{y})) \\
&= \sum_{\nu \in L(g,\epsilon(\vec{x}))} (g^{\sigma})^{\ast}( \nu ( \mathrm{var} \, g^{\sigma}) ) \ + \ g^{\partial,T}(\epsilon (\pi \circ \vec{y})).
\end{align*}
where $v = l(g^{Q}(\pi \circ \epsilon(\bar{y})))$ and $\left( \begin{bmatrix} r(b_1) \\ b_1 \end{bmatrix}/\Delta_{\alpha \alpha}, \ldots, \begin{bmatrix} r(b_m) \\ b_{m} \end{bmatrix}/\Delta_{\alpha \alpha} \right) = \phi \circ \epsilon (\vec{y})$.

Since $A(\alpha) \in \mathcal V$, we conclude that
\begin{align}\label{eq:16}
\sum_{\mu \in L(t,\epsilon(\vec{x}))} (t^{\sigma})^{\ast}( \mu ( \mathrm{var} \, t^{\sigma}) )  = t^{\sigma}\left( \begin{bmatrix} r(a_1) \\ a_1 \end{bmatrix}, \ldots, \begin{bmatrix} r(a_n) \\ a_{n} \end{bmatrix} \right)/\Delta_{\alpha \alpha} &= g^{\sigma'}\left( \begin{bmatrix} r(b_1) \\ b_1 \end{bmatrix}, \ldots, \begin{bmatrix} r(b_m) \\ b_{m} \end{bmatrix} \right)/\Delta_{\alpha \alpha} \\
&= \sum_{\nu \in L(g,\epsilon(\vec{x}))} (g^{\sigma})^{\ast}( \nu ( \mathrm{var} \, g^{\sigma}) ).
\end{align}
This shows the action is compatible with $\mathrm{Id} \, \mathcal V$. Since $Q \in \mathcal V$, we have $t^{Q}(\pi \circ \epsilon(\bar{x})) = g^{Q}(\pi \circ \epsilon(\bar{y}))$ which implies $u=v$. Using this we then have 
\begin{align*}
t^{\partial,T}(\pi \circ \epsilon (\vec{x})) &= t^{A_{T}(Q,A^{\alpha,\tau},\ast)}(\phi \circ \epsilon (\vec{x})) \ -_{u} \ t^{A(\alpha)/\Delta_{\alpha \alpha}}\left(\phi \circ \epsilon (\vec{x}) \right)\\
&= g^{A_{T}(Q,A^{\alpha,\tau},\ast)}(\phi \circ \epsilon (\vec{y})) \ -_{v} \ g^{A(\alpha)/\Delta_{\alpha \alpha}}\left(\phi \circ \epsilon (\vec{y}) \right) = g^{\partial,T}(\epsilon (\pi \circ \vec{y}))
\end{align*}
which shows that $T$ is a 2-cocycle compatible with $\mathrm{Id} \, \mathcal V$ by Definition~\ref{def:2cocyle}.
\end{proof}

\begin{remark}
If $l$ is lifting associated to $\rho: A(\alpha) \rightarrow Q$ from affine datum, then it useful to observe
\[
\delta(l(f^{Q}(\bar{x}))) = f^{A(\alpha)/\Delta_{\alpha \alpha}}(\delta(\bar{x}))
\]
which follows from the property $(a,l(q)) \in \alpha \Longleftrightarrow \rho \left( \begin{bmatrix} a \\ l(q)\end{bmatrix} \right)=q$.
\end{remark}

The next theorem is complimentary to Theorem~\ref{thm:10} in that it starts with affine datum and membership $Q \in \mathcal U$ for a variety with compatible equational theory and reconstructs an extension in $\mathcal U$ which determines the given datum.

\begin{theorem}\label{thm:12}
Let $(Q,A^{\alpha,\tau},\ast)$ be affine datum in the signature $\tau$. Assume the action is weakly compatible with $\mathrm{Id} \, \mathcal U$ and $T= \{T_{f}: f \in \tau\}$ is a 2-cocycle compatible $\mathrm{Id} \, \mathcal U$. Then there is an extension $\pi : A_{T}(Q,A^{\alpha,\tau},\ast) \rightarrow Q$ which realizes the datum with $A_{T}(Q,A^{\alpha,\tau},\ast) \in \mathcal U$.
\end{theorem}
\begin{proof}
Since the action is weakly compatible with $\mathrm{Id} \, \mathcal U$, we have $Q \in \mathcal U$. Define $\pi:A_{T}(Q,A^{\alpha,\tau},\ast) \rightarrow Q$ by $\pi \left( \begin{bmatrix} a \\ b \end{bmatrix}/\Delta_{\alpha \alpha} \right) := \rho \left( \begin{bmatrix} a \\ b \end{bmatrix} \right)$. By Definition~\ref{def:2cocyle}.(C1) and Definition~\ref{def:datum}.(D3)-(D4) we see that $\pi$ is a surjective homomorphism and $\ker \pi = \hat{\alpha}/\Delta_{\alpha \alpha}$. Fix a lifting $l: Q \rightarrow A$ for $\rho$ and attendant $\alpha$-trace $r:A \rightarrow A$.

We show $A_{T}(Q,A^{\alpha,\tau},\ast) \in \mathcal U$. Take $t(\bar{x}) = g(\bar{y}) \in \mathrm{Id} \ \mathcal U$. Let $\epsilon: \mathrm{var} \ t \cup \mathrm{var} g \rightarrow A(\alpha)/\Delta_{\alpha \alpha}$ be an assignment. If we set $u=l(t^{Q}(\pi \circ \epsilon(\bar{x})))$ and $v=l(g^{Q}(\pi \circ \epsilon(\bar{y})))$, then $Q \in \mathcal V$ implies $u=v$. Then because the action and $T$ are both separately $\mathcal U$-compatible, by Lemma~\ref{lem:repterm} we have
\begin{align*}
t^{A_{T}(Q,A^{\alpha,\tau},\ast)}(\epsilon(\bar{x})) &= \sum_{\mu \in L(t,\epsilon(\vec{x}))} (t^{\sigma})^{\ast}( \mu ( \mathrm{var} \, t^{\sigma}) )  \ +_{u} \ t^{\partial, T}\left( \pi \circ \epsilon (\vec{x}) \right) \\
&= \sum_{\nu \in L(g,\epsilon(\vec{x}))} (g^{\sigma})^{\ast}( \nu ( \mathrm{var} \, g^{\sigma}) )  \ +_{u} \ g^{\partial, T}\left( \pi \circ \epsilon (\vec{y}) \right) = g^{A_{T}(Q,A^{\alpha,\tau},\ast)}(\epsilon(\bar{y})).
\end{align*}
This shows the algebra $A_{T}(Q,A^{\alpha,\tau},\ast)$ satisfies $\mathrm{Id} \ \mathcal U$ and so $A_{T}(Q,A^{\alpha,\tau},\ast) \in \mathcal U$.

We will put an isomorphic algebraic structure on the set $A$. Define the bijection $\phi:A \rightarrow A_{T}(Q,A^{\alpha,\tau},\ast)$ by $\phi(a):=\begin{bmatrix} r(a) \\ a \end{bmatrix}/\Delta_{\alpha \alpha}$ and a new algebra $\check{A} = \left\langle A, \{\check{f}: f \in \tau \} \right\rangle$ with operations $\check{f}(a_1,\ldots,a_n):=\phi^{-1}\left(F_{f}\left(\phi(a_1),\ldots,\phi(a_n) \right) \right)$ for $f \in \tau$. It is immediate that $A \approx A_{T}(Q,A^{\alpha,\tau},\ast)$ and $\ker (\pi \circ \phi) = \alpha$. We now show $A_{T}(Q,A^{\alpha,\tau},\ast)$ realizes the datum. In order to verify Definition~\ref{def:realizes}, we take $1 \leq i \leq n=\ar f$ and evaluate
\begin{align}
\check{f} & \left( \delta(r(x_{1})),\ldots,\delta(r(x_{i-1})),\begin{bmatrix} r(x_{i}) \\ x_{i} \end{bmatrix}/\Delta_{\alpha \alpha},\delta(r(x_{i+1})),\ldots,\delta(r(x_{n})) \right) \\
&= \begin{bmatrix} \check{f}(r(x_1),\ldots,r(x_{i-1}),r(x_{i}),r(x_{i+1}),\ldots,r(x_n)) \\  \check{f}(r(x_1),\ldots,r(x_{i-1}),x_{i},r(x_{i+1}),\ldots,r(x_n)) \end{bmatrix}/\Delta_{\alpha \alpha} \\ \label{eq:20}
&= \begin{bmatrix} \phi^{-1} \left( F_{f} \left( \phi(r(x_1)),\ldots,\phi(r(x_{i-1})),\phi(r(x_{i})),\phi(r(x_{i+1})),\ldots,\phi(r(x_n)) \right) \right) \\ \phi^{-1}\left( F_{f} \left( \phi(r(x_1)),\ldots,\phi(r(x_{i-1})),\phi(x_{i}),\phi(r(x_{i+1})),\ldots,\phi(r(x_n)) \right) \right)  \end{bmatrix}/\Delta_{\alpha \alpha}.
\end{align}
We will evaluate the operations $F_{f}$ on the above tuples. Let us write $u=l(f^{Q}(q_1,\ldots,q_n))$ where $q_i = \pi \left(\begin{bmatrix} r(x_i) \\ x_i \end{bmatrix}/\Delta_{\alpha \alpha}\right)$ and note $q_i = \pi \left(\delta(r(x_i)) \right) = \pi \left( \delta(x_i) \right)$. First, for any $1 \leq j < n=\ar f$, by the homomorphism property of the action we always have
\begin{align*}
a(f,j) \left( q_1,\ldots,q_{j-1},\delta(x_{j}),q_{j+1},\ldots,q_{n} \right) &= a(f,j) \left( q_1,\ldots,q_{j-1},\delta(x_{j}),q_{j+1},\ldots,q_{n} \right) \\
& \quad \quad +_{u} \ a(f,j) \left( q_1,\ldots,q_{j-1},\delta(x_{j}),q_{j+1},\ldots,q_{n} \right)
\end{align*}
which implies $a(f,j)(q_1,\ldots,q_j,\delta(x_{j+1}),\ldots,\delta(x_{n}) = \delta(u)$; similarly, $f^{\Delta} \left( \delta(x_{1}),\delta(x_{2}),\ldots,\delta(x_n) \right) = \delta(u)$. Second, note by Definition~\ref{def:datum}.D3 it follows that we can write
\[
a(f,j) \left( q_1,\ldots,q_{j-1},\begin{bmatrix} r(x_{j}) \\ x_{j} \end{bmatrix}/\Delta_{\alpha \alpha},q_{j+1},\ldots,q_{n} \right) = \begin{bmatrix} u \\ a_{j} \end{bmatrix}/\Delta_{\alpha \alpha}
\]
for some $a_{j} \in A$; similarly,  $T_{f}(q_1,\ldots,q_n) = \begin{bmatrix} u \\ b \end{bmatrix}/\Delta_{\alpha \alpha}$ for some $b \in A$. Then if we recall $r \circ r = r$ for an $\alpha$-trace, we have
\begin{align*}
F_{f} & \left( \phi(r(x_1)),\ldots,\phi(r(x_{i-1})),\phi(x_{i}),\phi(r(x_{i+1})),\ldots,\phi(r(x_n)) \right)  \\
&=  f^{\Delta} \left( \begin{bmatrix} r(x_1) \\ r(x_1) \end{bmatrix}/\Delta_{\alpha \alpha}, \delta(r(x_2)),\ldots,\delta(r(x_{i})),\delta(x_{i+1}),\ldots,\delta(x_n) \right) \\
& \quad +_{u} \  \sum_{j=2}^{i-1} a(f,j)\left(q_1,\ldots,q_{j-1},\delta(r(x_{j})),q_{j+1},\ldots,q_{i},\ldots,q_{n} \right)  \\
& \quad +_{u} \ a(f,i)\left(q_1,\ldots,q_{i-1},\begin{bmatrix} r(x_{i}) \\ x_{i} \end{bmatrix}/\Delta_{\alpha \alpha},q_{i+1},\ldots,q_{n} \right)  \\
& \quad +_{u} \sum_{j=i+1}^{n} a(f,j)\left(q_1,\ldots,q_{j-1},\delta(r(x_{j})),q_{j+1},\ldots,q_{n} \right)   \ +_{u} \ T_{f}(q_1,\ldots,q_n) \\
&= a(f,i)\left(q_1,\ldots,q_{i-1},\begin{bmatrix} r(x_{i}) \\ x_{i} \end{bmatrix}/\Delta_{\alpha \alpha},q_{i+1},\ldots,q_{n} \right) \ +_{u} \ T_{f}(q_1,\ldots,q_n) \\
&= \begin{bmatrix} u \\ a_{i} \ \hat{+}_{u} \  b \end{bmatrix}/\Delta_{\alpha \alpha}
\end{align*}
where we have written $x \ \hat{+}_{u} \ y = m(x,u,y)$ for the induced sum on $A$. Note $\hat{\alpha}/\Delta_{\alpha \alpha}$ abelian implies the sum $a_{i+2} \ \hat{+}_{u} \cdots \hat{+}_{u} \ a_n \ \hat{+}_{u} \ b$ is unique up to association. In the same manner we have,
\begin{align*}
F_{f} \left( \phi(r(x_1)),\ldots,\phi(r(x_{i})),\phi(x_{i+1}),\phi(x_{i+2}),\ldots,\phi(x_n) \right) = T_{f}(q_1,\ldots,q_n) = \begin{bmatrix} u \\ b \end{bmatrix}/\Delta_{\alpha \alpha}.
\end{align*}
Putting the above together with Eq.(\ref{eq:20}) we see that
\begin{align*}
\check{f} & \left( \delta(r(x_{1})),\ldots,\delta(r(x_{i-1})),\begin{bmatrix} r(x_{i}) \\ x_{i} \end{bmatrix}/\Delta_{\alpha \alpha},\delta(r(x_{i+1})),\ldots,\delta(r(x_{n})) \right) \displaybreak[0]\\
&= \begin{bmatrix} \phi^{-1}\left( \begin{bmatrix} u \\ b \end{bmatrix}/\Delta_{\alpha \alpha} \right) \\ \phi^{-1}\left( \begin{bmatrix} u \\ a_{i} \ \hat{+}_{u} \ b \end{bmatrix}/\Delta_{\alpha \alpha} \right)  \end{bmatrix}/\Delta_{\alpha \alpha} \displaybreak[0]\\
&= \begin{bmatrix} b \\ a_{i} \ \hat{+}_{u} \ b  \end{bmatrix}/\Delta_{\alpha \alpha} \displaybreak[0]\\
&= \begin{bmatrix}  b \\  m(a_{i},u, b) \end{bmatrix}/\Delta_{\alpha \alpha} \displaybreak[0]\\
&= \begin{bmatrix} u \\ a_{i} \end{bmatrix}/\Delta_{\alpha \alpha} = a(f,i)(q_1,\ldots,q_{i-1},\begin{bmatrix} r(x_{i}) \\ x_{i} \end{bmatrix}/\Delta_{\alpha \alpha},q_{i+1},\ldots,q_n)
\end{align*}
which shows $\check{A} \approx A_{T}(Q,A^{\alpha,\tau},\ast)$ realizes the datum.
\end{proof}

We say the variety $\mathcal U$ \emph{contains} the datum $(Q,A^{\alpha,\tau},\ast)$ if the action $\ast$ is weakly compatible with $\mathrm{Id} \, \mathcal U$. The following is a characterization of the internal semidirect product by retractions for algebras which realize affine datum; in particular, this holds for any algebra with an abelian congruence in a variety with a weak-difference term. Note a retraction $r: A \rightarrow A$ is always a $\ker r$-trace.

\begin{proposition}\label{prop:6}
Let $A$ be an algebra which realizes affine datum $(A^{\alpha,\tau}, Q, \ast)$. Let $\pi: A \rightarrow A/\alpha$ be the canonical homomorphism for $\alpha$. The following are equivalent:
\begin{enumerate}

	\item $A \approx A(\alpha)/\Delta_{\alpha \alpha};$
	
	\item there is a homomorphism $l: A/\alpha \rightarrow A$ such that $\pi \circ l = \id_{A/\alpha}$;
	
	\item there is a retraction $r: A \rightarrow A$ with $\ker r = \alpha$.

\end{enumerate}	
\end{proposition}
\begin{proof}
$(2)\Leftrightarrow (3)$: Suppose $l: A/\alpha \rightarrow A$ is a homomorphism such that $\pi \circ l = \id_{A/\alpha}$. Define $r=l \circ \pi$. Then $r$ is a homomorphism and $r^{2} = l \circ \pi \circ l \circ \pi = l \circ \id_{A/\alpha} \circ \pi = l \circ \pi = r$; thus, $r$ is a retraction.

If we assume $r: A \rightarrow A$ is a retraction, then $r^{2}(x)=r(x)$ implies $(x,r(x)) \in \ker r = \alpha$; thus, $r$ is a $\alpha$-trace. Define $l: A/\alpha \rightarrow A$ by $l(q) = r(x)$ for any $x \in A$ such that $\pi(x)=x$. Since $r$ is an $\alpha$-trace, $l$ is well-defined. Take $q_i \in A/\alpha$ and $x_{i} \in A$ such that $\pi(x_i) = q_i$ for $i=1,\ldots,n=\ar f$. Then $\pi \circ r(f(x_1,\ldots,x_n)) = f(\pi \circ r(x_1),\ldots,\pi \circ r(x_n)) = f(\pi(x_1),\ldots,\pi(x_n))=f(q_1,\ldots,q_n)$. By definition we have $l(f(q_1,\ldots,q_n))=r(f(x_1,\ldots,x_n)) = f(r(x_1),\ldots,r(x_n))=f(l(q_1),\ldots,l(q_n))$ which shows $l$ is a homomorphism.

$(3)\Rightarrow (1)$: Assume $r: A \rightarrow A$ is a retraction. Notice $r$ is also an $\alpha$-trace. By Theorem~\ref{thm:12}, we have $A \approx A_{T}(Q,A^{\alpha,\tau},\ast)$ defined using an $\alpha$-trace $r$ and associated lifting $l$. By the construction, we have $T_{f}(\pi(x_1),\ldots,\pi(x_n)) = \begin{bmatrix} l(f(\pi(x_1),\ldots,\pi(x_n))) \\ f(l \circ \pi(x_1),\ldots,l \circ \pi(x_n)) \end{bmatrix}/\Delta_{\alpha \alpha} = \delta(f(r(x_1),\ldots,r(x_n)))$ since $l$ is a homomorphism. Then taking $u = l(f(\pi(x_1),\ldots,\pi(x_n)) = f(r(x_1),\ldots,r(x_n))$ we have in $A_{T}(Q,A^{\alpha,\tau},\ast)$
\begin{align*}
F_{f}\left( \begin{bmatrix} r(x_1) \\ x_1 \end{bmatrix}/\Delta_{\alpha \alpha},\ldots, \begin{bmatrix} r(x_n) \\ x_n \end{bmatrix}/\Delta_{\alpha \alpha} \right) &= \begin{bmatrix} r(f(x_1,\ldots,x_n)) \\ f(r(x_1),\ldots,r(x_n)) \end{bmatrix}/\Delta_{\alpha \alpha} \displaybreak[0]\\
&= \begin{bmatrix} f(r(x_1),\ldots,r(x_n)) \\ f(x_1,\ldots,x_n) \end{bmatrix}/\Delta_{\alpha \alpha} +_{u} \begin{bmatrix} r(f(x_1,\ldots,x_n)) \\ f(r(x_1),\ldots,r(x_n)) \end{bmatrix}/\Delta_{\alpha \alpha} \displaybreak[0]\\
&= \begin{bmatrix} f(r(x_1),\ldots,r(x_n)) \\ f(x_1,\ldots,x_n) \end{bmatrix}/\Delta_{\alpha \alpha} +_{u} T_{f}(\pi(x_1),\ldots,\pi(x_n)) \displaybreak[0]\\
&= f\left(\begin{bmatrix} r(x_1) \\ x_1 \end{bmatrix},\ldots,\begin{bmatrix} r(x_n) \\ x_n \end{bmatrix} \right)/\Delta_{\alpha \alpha}
\end{align*}
which shows $A \approx A_{T}(Q,A^{\alpha,\tau},\ast) \approx A(\alpha)/\Delta_{\alpha \alpha}$.
\end{proof}

The constructions have required a fixed lifting to define, but up to isomorphism do not depend on the particular choice. Making different choices for the liftings leads to an equivalence on extensions realizing datum which is defined by a combinatorial condition on 2-cocycles.

\begin{proposition}\label{prop:10}
Suppose $\pi : A \rightarrow Q$ with $\alpha = \ker \pi$ is an extension realizing affine datum. If $T$ is a 2-cocycle defined by the lifting $l$ and $T'$ is a 2-cocycle defined by the lifting $r'$, then there exists a map $h: Q \rightarrow A(\alpha)/\Delta_{\alpha \alpha}$ such that
\begin{align*}
T'_{f}(\bar{x}) -_{u} T_{f}(\bar{x}) &= f^{\Delta}(h(x_1),\delta(l(x_2)),\ldots,\delta(l(x_n))) -_{u} h(f^{Q}(x_1,\ldots,x_n)) \\
&+_{u} \sum_{i=2}^{n} a(f,i)\left(x_1,\ldots,x_{i-1},h(x_{i}),x_{i+1},\ldots,x_n \right) \quad \quad \quad \quad \quad \quad  (f \in \tau)
\end{align*}
where $u=l(f^{Q}(x_1,\ldots,x_n))$.
\end{proposition}
\begin{proof}
The action is defined as in Eq.(\ref{eq:actiondef}) from Theorem~\ref{thm:10}. Define $h(x) := \begin{bmatrix} l(x) \\ l'(x) \end{bmatrix}/\Delta_{\alpha \alpha}$. We first show
\begin{align}\label{eq:17}
T'_{f}(\bar{x}) -_{u} T_{f}(\bar{x}) = f^{A(\alpha)/\Delta_{\alpha \alpha}} \left( h(\bar{x}) \right) -_{u} h(f^{Q}(\bar{x})) \quad \quad \quad \quad \quad (f \in \tau).
\end{align}
If $l,l': Q \rightarrow A$ are liftings associated with $\pi$ such that $r=l \circ \pi$ and $r'=l' \circ \pi$, then recall the 2-cocycles are defined by
$T'_{f}(\bar{x}) = \begin{bmatrix} l'(f(\bar{x})) \\ f(l'(\bar{x})) \end{bmatrix}/\Delta_{\alpha \alpha}$ and $T_{f}(\bar{x}) = \begin{bmatrix} l(f(\bar{x})) \\ f(l(\bar{x})) \end{bmatrix}\Delta_{\alpha \alpha}$. If we set $v=f(l(x_1),\ldots,l(x_n))$, then note $\begin{bmatrix} u \\ u \end{bmatrix} \Delta_{\alpha \alpha} \begin{bmatrix} v \\ v \end{bmatrix}$. Then we can expand
\begin{align*}
T_{f}'(\bar{x}) -_{u} T_{f}(\bar{x}) &= \begin{bmatrix} l'(f(\bar{x})) \\ f(l'(\bar{x})) \end{bmatrix}/\Delta_{\alpha \alpha} -_{u} \begin{bmatrix} l(f(\bar{x})) \\ f(l(\bar{x})) \end{bmatrix}/\Delta_{\alpha \alpha}  \displaybreak[0]\\
&= \begin{bmatrix} l'(f(\bar{x})) \\ f(l'(\bar{x})) \end{bmatrix}/\Delta_{\alpha \alpha} +_{u} \begin{bmatrix} f(l(\bar{x})) \\ l(f(\bar{x})) \end{bmatrix}/\Delta_{\alpha \alpha} \displaybreak[0]\\
&= m\left( \begin{bmatrix} f(l(\bar{x})) \\ f(l'(\bar{x})) \end{bmatrix}, \begin{bmatrix} f(l(\bar{x})) \\ f(l(\bar{x})) \end{bmatrix}, \begin{bmatrix} l'f((\bar{x})) \\ f(l(\bar{x})) \end{bmatrix} \right)/\Delta_{\alpha \alpha} +_{u} \begin{bmatrix} f(l(\bar{x})) \\ l(f(\bar{x})) \end{bmatrix}/\Delta_{\alpha \alpha} \displaybreak[0]\\
&= \begin{bmatrix} f(l(\bar{x})) \\ f(l'(\bar{x})) \end{bmatrix}/\Delta_{\alpha \alpha} +_{v} \begin{bmatrix} l'(f(\bar{x})) \\ f(l(\bar{x})) \end{bmatrix}/\Delta_{\alpha \alpha}  +_{u} \begin{bmatrix} f(l(\bar{x})) \\ l(f(\bar{x})) \end{bmatrix}/\Delta_{\alpha \alpha} \displaybreak[0]\\
&= \begin{bmatrix} f(l(\bar{x})) \\ f(l'(\bar{x})) \end{bmatrix}/\Delta_{\alpha \alpha} +_{v} m \left(\begin{bmatrix} l'(f(\bar{x})) \\ f(l(\bar{x})) \end{bmatrix} , \begin{bmatrix} l(f(\bar{x})) \\ l(f(\bar{x})) \end{bmatrix}, \begin{bmatrix} f(l(\bar{x})) \\ l(f(\bar{x})) \end{bmatrix}\right)/\Delta_{\alpha \alpha} \displaybreak[0]\\
&= \begin{bmatrix} f(l(\bar{x})) \\ f(l'(\bar{x})) \end{bmatrix}/\Delta_{\alpha \alpha} +_{u} \begin{bmatrix} l'(f(\bar{x})) \\ l(f(\bar{x})) \end{bmatrix}/\Delta_{\alpha \alpha} \displaybreak[0]\\
&= f^{A(\alpha)/\Delta_{\alpha \alpha}} \left(\begin{bmatrix} l(\bar{x}) \\ l'(\bar{x}) \end{bmatrix} \right)/\Delta_{\alpha \alpha} -_{u} \begin{bmatrix} l(f(\bar{x})) \\ l'(f(\bar{x})) \end{bmatrix}/\Delta_{\alpha \alpha} = f^{A(\alpha)/\Delta_{\alpha \alpha}}(h(\bar{x})) -_{u} h(f^{Q}(\bar{x})).
\end{align*}
In a similar manner, if we set $u_i = f(l'(x_1),\ldots,l'(x_i),l(x_{i+1}),\ldots,l(x_n))$, then using realization we can expand
\begin{align*}
T_{f}'(\bar{x}) -_{u} T_{f}(\bar{x}) &= f(h(\bar{x})) -_{u} h(f(\bar{x})) \displaybreak[0]\\
&= f\left( \begin{bmatrix} l(x_1) \\ l'(x_1) \end{bmatrix}, \ldots,\begin{bmatrix} l(x_n) \\ l'(x_n) \end{bmatrix} \right)/\Delta_{\alpha \alpha} -_{u} h(f(\bar{x})) \displaybreak[0]\\
&= m\left( \begin{bmatrix} f(l(x_1),l(x_2),\ldots,l(x_n)) \\ f(l'(x_1),l(x_2),\ldots,l(x_n)) \end{bmatrix}, \begin{bmatrix} u_1 \\ u_1 \end{bmatrix}, \begin{bmatrix} f(l'(x_1),l(x_2),\ldots,l(x_n)) \\ f(l'(x_1),l'(x_2),\ldots,l'(x_n)) \end{bmatrix} \right)/\Delta_{\alpha \alpha} \\
& \quad -_{u} h(f(\bar{x})) \displaybreak[0]\\
&= f^{\Delta} \left(h(x_1),\delta(l(x_2)),\ldots,\delta(l(x_n)) \right) +_{u_1} f\left( \begin{bmatrix} l'(x_1) \\ l'(x_1) \end{bmatrix}/\Delta_{\alpha \alpha},h(x_2),\ldots, h(x_n) \right) -_{u} h(f(\bar{x})) \displaybreak[0]\\
&= f^{\Delta} \left(h(x_1),\delta(l(x_2)),\ldots,\delta(l(x_n)) \right) +_{u_1}  f\left( \begin{bmatrix} l'(x_1) \\ l'(x_1) \end{bmatrix}/\Delta_{\alpha \alpha},h(x_2),\delta(l(x_3)),\ldots, \delta(l(x_n)) \right) \displaybreak[0]\\
& \quad +_{u_2} f\left( \begin{bmatrix} l'(x_1) \\ l'(x_1) \end{bmatrix}/\Delta_{\alpha \alpha},\begin{bmatrix} l'(x_2) \\ l'(x_2) \end{bmatrix}/\Delta_{\alpha \alpha},h(x_3),\ldots, h(x_n) \right) -_{u} h(f(\bar{x})) \displaybreak[0]\\
&= f^{\Delta} \left(h(x_1),\delta(l(x_2)),\ldots,\delta(l(x_n)) \right) +_{u}  a(f,1)(x_1,h(x_2),x_3,\ldots,x_n ) \\
& \quad +_{u} f\left( \begin{bmatrix} l'(x_1) \\ l'(x_1) \end{bmatrix}/\Delta_{\alpha \alpha},\begin{bmatrix} l'(x_2) \\ l'(x_2) \end{bmatrix}/\Delta_{\alpha \alpha},h(x_3),\ldots, h(x_n) \right) -_{u} h(f(\bar{x})) \displaybreak[0]\\
&\vdots \displaybreak[0]\\
&= f^{\Delta}(h(x_1),\delta(l(x_2)),\ldots,\delta(l(x_n))) -_{u} h(f(\bar{x})) \displaybreak[0]\\
&\quad +_{u} \ \sum_{i=2}^{n} a(f,i)\left(x_1,\ldots,x_{i-1},h(x_{i}), x_{i+1},\ldots, x_n \right)
\end{align*}
since each $\delta(u_i)=\delta(u)$ and $\delta(l(x)) = \delta(l'(x))$.
\end{proof}

\begin{definition}\label{def:2cobound}
Let $(Q,A^{\alpha,\tau},\ast)$ be affine datum in the signature $\tau$. A sequence of operations $G=\{G_{f}: f \in \tau\}$  where $G_{f}: Q^{\ar f} \rightarrow A(\alpha)/\Delta_{\alpha \alpha}$ is a \emph{2-coboundary} of the datum if there is a function $h: Q \rightarrow A(\alpha)/\Delta_{\alpha \alpha}$ such that for any lifting $l:Q \rightarrow A$ associated to the datum
\begin{enumerate}

	\item[(B1)] $h(x) \mathrel{\hat{\alpha}/\Delta_{\alpha \alpha}} \delta \circ l(x)$;

	\item[(B2)] for each $f \in \tau$,
\begin{align*}
G_{f}(x_1,\ldots,x_n) &= f^{\Delta}(h(x_1),\delta(l(x_2)),\ldots,\delta(l(x_n))) -_{u} h(f^{Q}(x_1,\ldots,x_n)) \\
&+_{u} \sum_{i=2}^{n} a(f,i) ( x_1,\ldots,x_{i-1},h(x_{i}),x_{i+1},\ldots,x_n )
\end{align*}

\end{enumerate}
where $u = l(f(x_1,\ldots,x_n))$.
\end{definition}

The function $h : Q \rightarrow A(\alpha)/\Delta_{\alpha \alpha}$ referenced in the above definition is said to witness the 2-coboundary.

\begin{definition}
Let $(Q,A^{\alpha,\tau},\ast)$ be affine datum and $\mathcal U$ a variety in the same signature which contains the datum. The set of 2-cocycles compatible with $\mathcal U$ is denoted by $Z^{2}_{\mathcal U}(Q,A^{\alpha,\tau},\ast)$. The set of 2-coboundaries of the datum is denoted by $B^{2}(Q,A^{\alpha,\tau},\ast)$.
\end{definition}

Notice the notation for the class of 2-coboundaries omits a subscript for the variety $\mathcal U$. We shall see in the next lemma that 2-coboundaries are compatible with an variety containing the datum.

\begin{lemma}\label{lem:5}
Let $(Q,A^{\alpha,\tau},\ast)$ be affine datum and $\mathcal U$ a variety in the same signature which contains the datum. For any lifting $l:Q \rightarrow A$ associated to the datum, the operations
\begin{align*}
(T_{f} + T'_{f})(\bar{x}) &:= T_{f}(\bar{x}) +_{l(f(\bar{x}))} T'_{f}(\bar{x})  \quad \quad \quad \quad \quad (f \in \tau, \bar{x} \in Q^{\ar f}) \\
(h + h')(x) &:= h(x) +_{l(x)} h'(x)  \quad \quad \quad \quad \quad \quad \quad (x \in Q)
\end{align*}
makes $Z^{2}_{\mathcal U}(Q,A^{\alpha,\tau},\ast)$ into an abelian group with subgroup $B^{2}(Q,A^{\alpha,\tau},\ast) \leq Z^{2}_{\mathcal U}(Q,A^{\alpha,\tau},\ast)$.
\end{lemma}
\begin{proof}
From previous remarks, note the definition of the operation is not dependent on the lifting. The fact that the operation defines an abelian group follows from noting that the ternary operation $m$ of the datum is affine on the congruence blocks of $\hat{\alpha}/\Delta_{\alpha \alpha}$. The sum of 2-cocycles compatible with $\mathrm{Id} \, \mathcal U$ is again a 2-cocycle compatible with $\mathrm{Id} \, \mathcal U$ follows from the homomorphism property of the action which distributes the action terms over the witnesses of the respective 2-cocycles. These facts also show properties $(B1)$ and $(B2)$ are preserved by the sum and so 2-coboundaries form an abelian group.

What requires a little explanation is why a 2-coboundary is a 2-cocycle compatible with $\mathcal U$. Let $G = \{G_{f}: f \in \tau \}$ be a 2-coboundary. For any choice of 2-cocycle $\{T_{f}: f \in \tau \}$ compatible with $\mathcal U$, Theorem~\ref{thm:12} provides an algebra $\mathcal U \ni A_{T}(Q,A^{\alpha,\tau},\ast) \stackrel{\pi}{\rightarrow} Q$ on the universe $A(\alpha)/\Delta_{\alpha \alpha}$ which realizes the datum. Note there is at least one such 2-cocyle which is provided by the datum. The condition $(B2)$ now takes the form
\begin{align}\label{eq:21}
G_{f}(x_1,\ldots,x_n) = F^{A_{T}(Q,A^{\alpha,\tau},\ast)}_{f}(h(x_1),\ldots,h(x_n)) \, -_{u} \, T_{f}(x_1,\ldots,x_n) \, -_{u} \, h(f^{Q}(x_1,\ldots,x_n))
\end{align}
where $u=l(f(x_1,\ldots,x_n))$. We have preserved the superscripts here to indicate the realization. Consider a term $t(\bar{x})$ and assignment $\epsilon: \mathrm{var} \ t \rightarrow A(\alpha)/\Delta_{\alpha \alpha}$. By induction on the composition tree of a term $t(\bar{x})$ with Eq.(\ref{eq:21}) as the base case, we can evaluate
\begin{align}\label{eq:22}
t^{\partial,G}(\pi \circ \epsilon(\bar{x})) = t^{A_{T}(Q,A^{\alpha,\tau},\ast)}(h \circ \pi \circ \epsilon(\bar{x})) \, -_{v} \, t^{\partial,T}(\pi \circ \epsilon(\bar{x})) \, -_{v} \, h(t^{Q}(\pi \circ \epsilon (\bar{x})))
\end{align}
where $v=l(t^{Q}(\pi \circ \epsilon (\bar{x})))$. Then using Eq.(\ref{eq:22}), it follows from the fact that $A_{T}(Q,A^{\alpha,\tau},\ast) \in \mathcal U$, $T$ is 2-cocycle compatible with $\mathcal U$ and $Q \in \mathcal U$, that $G$ is also a 2-cocycle compatible with $\mathcal U$.
\end{proof}

\begin{definition}
Let $(Q,A^{\alpha,\tau},\ast)$ be affine datum and $\mathcal U$ a variety in the same signature which contains the datum. The \emph{second cohomology group relative to} $\mathcal U$ is the quotient group
\[
H^{2}_{\mathcal U}(Q,A^{\alpha,\tau},\ast) := Z^{2}_{\mathcal U}(Q,A^{\alpha,\tau},\ast)/B^{2}(Q,A^{\alpha,\tau},\ast).
\]
\end{definition}

\begin{definition}\label{def:4}
Let $(Q,A^{\alpha,\tau},\ast)$ be affine datum. Two extensions $A$ and $A'$ realizing the datum are \emph{equivalent} if there is a 2-cocycle $T$ associated to $A$ and 2-cocycle $T'$ associated to $A'$ such that $T' - T \in B^{2}(Q,A^{\alpha,\tau},\ast)$.
\end{definition}

If $A$ realizes the datum, we write $[A]$ for the equivalence class determined by Definition~\ref{def:4}. The \emph{trivial 2-cocycle} for datum $(Q,A^{\alpha,\tau},\ast)$ has operations $T_{f}(x_{1},\ldots,x_{\ar f}) = \delta \circ l (f^{Q}(x_{1},\ldots,x_{ar f}))$ for $f \in \tau$ and any lifting $l: Q \rightarrow A$ associated to the datum. When the context is clear, we denote the trivial 2-cocycle by $T=0$.

\begin{theorem}\label{thm:cohomology}
Let $(Q,A^{\alpha,\tau},\ast)$ be affine datum and $\mathcal U$ a variety in the same signature which contains the datum. The set of equivalence classes of extensions in $\mathcal U$ realizing the datum with the operation $[A_{T}(Q,A^{\alpha,\tau},\ast)] + [A_{T'}(Q,A^{\alpha,\tau},\ast)]:= [A_{T + T'}(Q,A^{\alpha,\tau},\ast)]$ is an abelian group isomorphic with $H^{2}_{\mathcal U}(Q,A^{\alpha,\tau},\ast)$. The zero of the abelian group is the class of the semidirect product $[A_{0}(Q,A^{\alpha,\tau},\ast)]$.   
\end{theorem}

For affine datum $(Q,A^{\alpha,\tau},\ast)$, let $\mathcal L(Q,A^{\alpha,\tau},\ast)$ be the set of varieties in the signature $\tau$ which contain the datum ordered by inclusion. It is a complete sublattice of the lattice of varieties in the signature $\tau$. Let $Z^{2}(Q,A^{\alpha,\tau},\ast)$ denote the abelian group generated by 2-cocycles compatible with some $\mathcal U \in \mathcal L(Q,A^{\alpha,\tau},\ast)$. Since $B^{2}(Q,A^{\alpha,\tau},\ast) \leq Z^{2}_{\mathcal U}(Q,A^{\alpha,\tau},\ast)$ for all $\mathcal U \in L(Q,A^{\alpha,\tau},\ast)$ by Lemma~\ref{lem:5}, we can define the second cohomology group of the datum
\[
H^{2}(Q,A^{\alpha,\tau},\ast) := Z^{2}(Q,A^{\alpha,\tau},\ast)/B^{2}(Q,A^{\alpha,\tau},\ast).
\]
For any algebra $A$, $\Sub A$ denotes the lattice of subalgebras ordered by inclusion.

\begin{proposition}\label{prop:14}
Let $(Q,A^{\alpha,\tau},\ast)$ be affine datum in the signature $\tau$. The map
\[
\Psi: \mathcal L(Q,A^{\alpha,\tau},\ast) \rightarrow \Sub H^{2}(Q,A^{\alpha,\tau},\ast)
\]
defined by $\Psi(\mathcal U) := H^{2}_{\mathcal U}(Q,A^{\alpha,\tau},\ast)$ is a meet-homomorphism which is an upper adjoint for a Galois connection $(-,\Psi)$ between $\Sub H^{2}(Q,A^{\alpha,\tau},\ast)$ and $\mathcal L(Q,A^{\alpha,\tau},\ast)$. The subset of varieties generated by their algebras realizing the datum is join-complete. 
\end{proposition}
\begin{proof}
It is easy to see that $\psi$ is monotone which yields the immediate inclusions
\begin{align*}
\psi(\mathcal U_1 \wedge \mathcal U_{2}) \leq \psi(\mathcal U_1) \wedge \psi(\mathcal U_2) \quad \quad \text{and} \quad \quad \psi(\mathcal U_1) \vee \psi(\mathcal U_2) \leq \psi(\mathcal U_1 \vee \mathcal U_{2}).
\end{align*}
For the reverse inclusion for the meet operation, we use the description of the operation in terms of the class operators $\mathcal U_{1} \wedge \mathcal U_{2} = \Mod \left( \mathrm{Id} \, \mathcal U_{1} \cup \mathrm{Id} \, \mathcal U_{2} \right)$. Any $[T] \in \psi(\mathcal U_1) \wedge \psi(\mathcal U_2)$, is both $\mathcal U_{1}$-compatible and $\mathcal U_{2}$-compatible; according to Theorem~\ref{thm:12}, the algebra $A_{T}(Q,A^{\alpha,\tau},\ast) \in \mathcal U_1 \wedge \mathcal U_{2}$. Then by Theorem~\ref{thm:10} we have $[T] \in H^{2}_{\mathcal U_1 \wedge \mathcal U_{2}}(Q,A^{\alpha,\tau},\ast) = \psi(\mathcal U_1 \wedge \mathcal U_{2})$. We conclude $\psi$ is a meet-homomorphism.

The lower adjoint to $\Psi$ is the monotone map $\theta: \Sub H^{2}(Q,A^{\alpha,\tau},\ast) \longrightarrow \mathcal L(Q,A^{\alpha,\tau},\ast)$ defined by $\theta(E) := \bigvee_{[T] \in E} \mathcal V(A_{T})$. It is not to hard to see that $\theta \circ \Psi (\mathcal U) \leq \mathcal U$ and $\Psi \circ \theta (E) \geq E$. It follows that $\Psi \circ \theta$ is a closure operator and the closed sets are precisely the cohomology groups corresponding to varieties; that is, of the form $\Psi(\mathcal U)$. 
\end{proof}

The combinatorial equivalence on extensions defined on their associated 2-cocycles can also be given by special isomorphisms between the extensions

\begin{theorem}\label{thm:13}
Let $(A(\alpha),Q,\ast)$ be affine datum in the signature $\tau$. Let $\pi: A \rightarrow Q$ and $\pi' : A' \rightarrow Q$ be extensions realizing the datum. Then $A$ and $A'$ are equivalent if and only if there is an isomorphism $\gamma: A \rightarrow A'$ such that
\begin{enumerate}

	\item $\pi' \circ \gamma = \pi$, and
	
	\item $\gamma = m(\gamma \circ r, r , \id)$ for all $\alpha$-traces $r: A \rightarrow A$.

\end{enumerate}
\end{theorem}
\begin{proof}
First, assume $A$ and $A'$ are equivalent extensions realizing the datum. We make take $\mathcal U = \mathcal V(A,A')$ and note $Q \in \mathcal U$ and $\mathcal U$ contains the datum since it contains algebras which realize the datum. By Theorem~\ref{thm:12}, we have isomorphisms $\phi : A \rightarrow A_{T}(Q,A^{\alpha,\tau},\ast)$ and $\phi' : A' \rightarrow A(\alpha,\ast,T')$ by the explicit construction detailed in that theorem. By Definition~\ref{def:4}, there is a 2-coboundary defined by $h: Q \rightarrow A(\alpha)/\Delta_{\alpha \alpha}$ such that
\begin{align*}
T_{f}(\bar{x}) - T'_{f}(\bar{x}) &= f^{\Delta} \left( h(x_1),\delta(\bar{l}(x_2)),\ldots,\delta(\bar{l}(x_n)) \right) -_{u} h(f^{Q}(x_1,\ldots,x_n)) \\
&+_{u} \sum_{i=2}^{n} a(f,i)(x_1,\ldots,x_{i},h(x_{i}),x_{i+1},\ldots,x_n ) \quad \quad (f \in \tau, n=\ar f)
\end{align*}
where $u = f(x_1,\ldots,x_n)$ and $\bar{l}:Q \rightarrow A$ is a lifting for the datum. Let $r:A \rightarrow A$ and $r':A' \rightarrow A'$ be $\alpha$-traces used to define the extensions $A$ and $A'$, respectively. Let $l: Q \rightarrow A$ and $l:Q \rightarrow A'$ be the associated liftings for $r$ and $r'$, respectively.

Note every element in $A_{T}(Q,A^{\alpha,\tau},\ast)$ has a unique representation of the form $\begin{bmatrix} r(x) \\ x \end{bmatrix}/\Delta_{\alpha \alpha}$ with $x \in A$. Define $\bar{\gamma}: A_{T}(Q,A^{\alpha,\tau},\ast) \rightarrow A'(\alpha,\ast,T')$ by $\bar{\gamma}: \begin{bmatrix} r(x) \\ x \end{bmatrix}/\Delta_{\alpha \alpha} \longmapsto \begin{bmatrix} r(x) \\ x \end{bmatrix}/\Delta_{\alpha \alpha} +_{r(x)} h \circ \rho \left(\begin{bmatrix} r(x) \\ x \end{bmatrix}\right)$ and $\gamma = \phi'^{-1} \circ \bar{\gamma} \circ \phi : A \rightarrow A'$. Since $\phi(\ker \pi) = \phi'(\ker \pi') = \hat{\alpha}/\Delta_{\alpha \alpha}$, we see that
\begin{align*}
\pi' \circ \phi'^{-1} \circ \bar{\gamma} \left( \begin{bmatrix} r(x) \\ x \end{bmatrix}/\Delta_{\alpha \alpha} \right) &= \pi'\left( \begin{bmatrix} r(x) \\ x \end{bmatrix}/\Delta_{\alpha \alpha} +_{r(x)} h \circ \rho \left(\begin{bmatrix} r(x) \\ x \end{bmatrix}\right) \right) \\
&= m \left( \rho \left(\begin{bmatrix} r(x) \\ x \end{bmatrix}\right) , \rho \left(\begin{bmatrix} r(x) \\ r(x) \end{bmatrix}\right), \rho \left(\begin{bmatrix} r(x) \\ x \end{bmatrix}\right) \right) = \rho \left( \begin{bmatrix} r(x) \\ x \end{bmatrix} \right) = \pi (x)
\end{align*}
by idempotence in $Q$. This verifies $\pi = \pi' \circ \gamma$ and surjectivity follows readily. We check $\bar{\gamma}$ is a homomorphism. For simplicity, set $q_i = \rho \left(\begin{bmatrix} r(x_i) \\ x_i \end{bmatrix}\right)$. Calculating for $f \in \tau$ and $\bar{x} \in A^{\ar f}$,
\begin{align*}
\bar{\gamma} \left( F_{f}\left( \begin{bmatrix} r(\bar{x}) \\ \bar{x} \end{bmatrix}/\Delta_{\alpha \alpha} \right) \right) &= \bar{\gamma} \left( \begin{bmatrix} r(f(\bar{x})) \\ f(\bar{x}) \end{bmatrix}/\Delta_{\alpha \alpha} \right) \displaybreak[0]\\
&= \begin{bmatrix} r(f(\bar{x})) \\ f(\bar{x}) \end{bmatrix}/\Delta_{\alpha \alpha} +_{r(f(\bar{x}))} h \circ \rho \left( \begin{bmatrix} r(f(\bar{x})) \\ f(\bar{x}) \end{bmatrix} \right) \displaybreak[0]\\
&= \begin{bmatrix} r(f(\bar{x})) \\ f(\bar{x}) \end{bmatrix}/\Delta_{\alpha \alpha} +_{r(f(\bar{x}))} h(f^{Q}(q_1,\ldots,q_n)) \displaybreak[0]\\
&= f^{\Delta} \left( \begin{bmatrix} r(x_1) \\ x_1 \end{bmatrix}\Delta_{\alpha \alpha},\delta(x_2),\ldots,\delta(x_n) \right) \displaybreak[0]\\
& \quad +_{r(f(\bar{x}))} \sum_{i=2}^{n} a(f,i)\left(q_1,\ldots,q_{i-1},\begin{bmatrix} r(x_{i}) \\ x_{i} \end{bmatrix}/\Delta_{\alpha \alpha},q_{i+1},\ldots,q_{n}) \right) \displaybreak[0]\\
& \quad +_{r(f(\bar{x}))} T_{f}(\bar{q}) +_{r(f(\bar{x}))} h(f^{Q}(q_1,\ldots,q_n)) \displaybreak[0]\\
&= F_{f}\left( \begin{bmatrix} r(\bar{x}) \\ \bar{x} \end{bmatrix}/\Delta_{\alpha \alpha} \right) +_{r(f(\bar{x}))} h(f^{Q}(q_1,\ldots,q_n)).
\end{align*}
Since $\left( h \circ \rho \left(\begin{bmatrix} r(x) \\ x \end{bmatrix}\right) , \begin{bmatrix} r(x) \\ x \end{bmatrix}/\Delta_{\alpha \alpha} \right) \in \hat{\alpha}/\Delta_{\alpha \alpha}$, we can write each
\begin{align}\label{eqn:777}
h \circ \rho \left(\begin{bmatrix} r(x_i) \\ x_i \end{bmatrix}\right) = \begin{bmatrix} r(x_i) \\ a_i  \end{bmatrix}/\Delta_{\alpha \alpha} = \begin{bmatrix} r'(x_i) \\ m(a_i,r(x_i),r'(x_i)) \end{bmatrix}/\Delta_{\alpha \alpha} = \begin{bmatrix} r'(x_i) \\ w_i  \end{bmatrix}/\Delta_{\alpha \alpha}
\end{align}
for some $a_i \in A$ and $w_i = m(a_i,r(x_i),r'(x_i))$. Then $\gamma \left( \begin{bmatrix} r(x_i) \\ x_i \end{bmatrix}/\Delta_{\alpha \alpha} \right) = \begin{bmatrix} r'(x_i) \\ z_i \end{bmatrix}/\Delta_{\alpha \alpha}$ where $z_i = m(x_i,r(x_i),w_i)$. Then 
\begin{align*}
F_{f} \left( \bar{\gamma} \left( \begin{bmatrix} r(\bar{x}) \\ \bar{x}  \end{bmatrix}/\Delta_{\alpha \alpha} \right) \right) &= F_{f}\left( \begin{bmatrix} r'(x_1) \\ z_1 \end{bmatrix}/\Delta_{\alpha \alpha}, \ldots, \begin{bmatrix} r'(x_n) \\ z_n \end{bmatrix}/\Delta_{\alpha \alpha} \right) \displaybreak[0]\\
&= f^{\Delta}\left( \begin{bmatrix} r'(x_1) \\ z_1 \end{bmatrix}/\Delta_{\alpha \alpha}, \delta(z_2),\ldots,\delta(z_n) \right) +_{r(f(\bar{z}))}  \displaybreak[0]\\
& \quad \sum_{i=2}^{n} a(f,i) \left( q_1,\ldots,q_{i-1}, \begin{bmatrix} r'(x_{i}) \\ z_{i} \end{bmatrix}/\Delta_{\alpha \alpha} , q_{i+1},\ldots,q_{n}  \right) \displaybreak[0]\\
& \quad +_{r(f(\bar{z}))} T'_{f}(q_1,\ldots,q_n) \displaybreak[0]\\
&= f^{\Delta} \left( \begin{bmatrix} r(x_1) \\ x_1 \end{bmatrix}/\Delta_{\alpha \alpha}, \delta(z_2),\ldots,\delta(z_n) \right) +_{r(f(\bar{x}))} f^{\Delta} \left( h(q_1), \delta(z_2),\ldots,\delta(z_n) \right) \displaybreak[0]\\
& \quad +_{r(f(\bar{z}))} \sum_{i=2}^{n} a(f,i)\left(q_1,\ldots,q_{i-1}, \begin{bmatrix} r(x_{i}) \\ x_{i} \end{bmatrix}/\Delta_{\alpha \alpha} , q_{i+1},\ldots,q_n \right) +_{r(f(\bar{x}))} \displaybreak[0]\\
& \quad \sum_{i=2}^{n} a(f,i)\left(q_1,\ldots,q_{i-1}, h(q_{i}),q_{i+1},\ldots,q_n \right) +_{r(f(\bar{z}))} T'_{f}(q_1,\ldots,q_n) \displaybreak[0]\\
&= f^{\Delta} \left( \begin{bmatrix} r(x_1) \\ x_1 \end{bmatrix}/\Delta_{\alpha \alpha}, \delta(z_2),\ldots,\delta(z_n) \right) +_{r(f(\bar{z}))} \displaybreak[0]\\
&\quad \sum_{i=2}^{n} a(f,i)\left(q_1,\ldots,q_{i-1}, \begin{bmatrix} r(x_{i}) \\ x_{i} \end{bmatrix}/\Delta_{\alpha \alpha} , q_{i+1},\ldots, q_n \right) +_{r(f(\bar{x}))} \displaybreak[0]\\
& \quad  T_{f}(q_1,\ldots,q_n) +_{u} h(f^{Q}(q_1,\ldots,q_n)) \displaybreak[0]\\
&= F_{f}\left( \begin{bmatrix} r(\bar{x}) \\ \bar{x} \end{bmatrix}/\Delta_{\alpha \alpha} \right) +_{r(f(\bar{x}))} h(f^{Q}(q_1,\ldots,q_n)).
\end{align*}
Since $\delta(z_i) = \delta(x_i)$ and $\delta(r(f(\bar{x}))) = \delta(r(f(\bar{z})))=\delta(u)$, we conclude that $\bar{\gamma}$, and thus $\gamma$, is a homomorphism.

Now assume $\left( \begin{bmatrix} r(x) \\ x \end{bmatrix}/\Delta_{\alpha \alpha}, \begin{bmatrix} r(y) \\ y \end{bmatrix}/\Delta_{\alpha \alpha} \right) \in \ker \bar{\gamma}$. Note we have $h \circ \rho \left(\begin{bmatrix} r(x) \\ x \end{bmatrix}\right) = \begin{bmatrix} r(x) \\ a  \end{bmatrix}/\Delta_{\alpha \alpha}$ and $h \circ \rho \left(\begin{bmatrix} r(y) \\ y \end{bmatrix}\right) = \begin{bmatrix} r(y) \\ b  \end{bmatrix}/\Delta_{\alpha \alpha}$ for some $a,b \in A$. Then
\begin{align}\label{eq:6}
\begin{bmatrix} r(x) \\ m(x,r(x),a) \end{bmatrix}/\Delta_{\alpha \alpha} = \bar{\gamma} \left( \begin{bmatrix} r(x) \\ x \end{bmatrix}/\Delta_{\alpha \alpha} \right) = \bar{\gamma} \left( \begin{bmatrix} r(y) \\ y \end{bmatrix}/\Delta_{\alpha \alpha} \right) =  \begin{bmatrix} r(y) \\ m(y,r(y),b) \end{bmatrix}/\Delta_{\alpha \alpha}.
\end{align}
Then $\pi = \pi' \circ \gamma$ yields $\left( \begin{bmatrix} r(x) \\ x \end{bmatrix}/\Delta_{\alpha \alpha}, \begin{bmatrix} r(y) \\ y \end{bmatrix}/\Delta_{\alpha \alpha} \right) \in \ker \pi = \hat{\alpha}/\Delta_{\alpha \alpha}$ which implies $r(x) = r(y)$, and so $a=b$ since the 2-coboundary $h$ only depends on $Q$. Then Eq (\ref{eq:6}) implies $m(x,r(x),a) = m(y,r(y),b)=m(y,r(x),a)$. Since $m$ is affine on $\alpha$-blocks, we must have $x=y$; thus, $\bar{\gamma}$ is injective.

For the second condition on the isomorphism $\gamma$, note $\rho \left( \begin{bmatrix} r(x) \\ r(x) \end{bmatrix} \right) = \rho \left( \begin{bmatrix} r(x) \\ x \end{bmatrix} \right)$ implies $h \circ \rho \left( \begin{bmatrix} r(x) \\ r(x) \end{bmatrix} \right) = h \circ \rho \left( \begin{bmatrix} r(x) \\ x \end{bmatrix} \right) = \begin{bmatrix} r(x) \\ a \end{bmatrix}/\Delta_{\alpha \alpha}$ for some $a \in A$ with $(a,r(x)) \in \alpha$. Then
\begin{align*}
\begin{bmatrix} r(x) \\ \gamma(r(x)) \end{bmatrix}/\Delta_{\alpha \alpha} &= \begin{bmatrix} r(x) \\ \phi'^{-1} \left( \begin{bmatrix} r(x) \\ r(x) \end{bmatrix}/\Delta_{\alpha \alpha} +_{r(x)} h \circ \rho \left( \begin{bmatrix} r(x) \\ r(x) \end{bmatrix}\right) \right) \end{bmatrix}/\Delta_{\alpha \alpha} \displaybreak[0]\\
&= \begin{bmatrix} r(x) \\ \phi'^{-1} \left( h \circ \rho \left( \begin{bmatrix} r(x) \\ r(x) \end{bmatrix}\right) \right) \end{bmatrix}/\Delta_{\alpha \alpha} \displaybreak[0]\\
&= \begin{bmatrix} r(x) \\ \phi'^{-1} \left( \begin{bmatrix} r(x) \\ a \end{bmatrix}/\Delta_{\alpha \alpha} \right) \end{bmatrix}/\Delta_{\alpha \alpha} \displaybreak[0]\\
&= \begin{bmatrix} r(x) \\ m(a,r(x),r'(x)) \end{bmatrix}/\Delta_{\alpha \alpha} \displaybreak[0]\\
&= \begin{bmatrix} x \\ m( m(a,r(x),r'(x)),r(x),x ) \end{bmatrix}/\Delta_{\alpha \alpha} \displaybreak[0]\\
&= \begin{bmatrix} x \\ m( m(x,r(x),a),r(x),r'(x) ) \end{bmatrix}/\Delta_{\alpha \alpha} \displaybreak[0]\\
&= \begin{bmatrix} x \\ \phi'^{-1} \left( \begin{bmatrix} r(x) \\ m(x,r(x),a) \end{bmatrix}/\Delta_{\alpha \alpha} \right) \end{bmatrix}/\Delta_{\alpha \alpha} \displaybreak[0]\\
&= \begin{bmatrix} x \\ \phi'^{-1} \left( \begin{bmatrix} r(x) \\ x \end{bmatrix}/\Delta_{\alpha \alpha} +_{r(x)} h \circ \rho \left( \begin{bmatrix} r(x) \\ x \end{bmatrix} \right) \right) \end{bmatrix}/\Delta_{\alpha \alpha} = \begin{bmatrix} x \\ \gamma(x) \end{bmatrix}/\Delta_{\alpha \alpha}. 
\end{align*}

Conversely, assume there is a homomorphism $\gamma: A \rightarrow A'$ which satisfies the conditions. The condition $\pi=\pi' \circ \gamma$ implies $\gamma \circ l : Q \rightarrow A'$ is a lifting for $\pi': A' \rightarrow Q$. Condition (2) implies that for any $f \in \tau$ and $\bar{x} \in A^{\ar f}$, if we apply the operation
\[
 \begin{bmatrix} f(\bar{x}) \\ \gamma(f(\bar{x})) \end{bmatrix} = \begin{bmatrix} f(\bar{x}) \\ f(\gamma(\bar{x})) \end{bmatrix} \mathrel{\Delta_{\alpha \alpha}} \begin{bmatrix} f(l(\bar{x})) \\ f(\gamma(l(\bar{x}))) \end{bmatrix}  = \begin{bmatrix} f(l(\bar{x})) \\ \gamma(f(l(\bar{x}))) \end{bmatrix}
\]
and then also by substitution
\[
\begin{bmatrix} f(\bar{x}) \\ \gamma(f(\bar{x})) \end{bmatrix} \mathrel{\Delta_{\alpha \alpha}} \begin{bmatrix} l(f(\bar{x})) \\ \gamma(l(f(\bar{x}))) \end{bmatrix}.
\]
From the above we conclude
\[
\begin{bmatrix} \gamma \circ l (f(\bar{x})) \\ l \circ f(\bar{x})  \end{bmatrix} \mathrel{\Delta_{\alpha \alpha}} \begin{bmatrix} \gamma(f(l(\bar{x}))) \\ f(l(\bar{x})) \end{bmatrix},
\]
and since $\Delta_{\alpha \alpha} \leq \hat{\alpha}$ we have
\begin{align}
 \begin{bmatrix} \gamma \circ l (f(\bar{x})) \\ f(\gamma \circ l(\bar{x})) \end{bmatrix} = \begin{bmatrix} \gamma \circ l (f(\bar{x})) \\ \gamma(f(l(\bar{x}))) \end{bmatrix} \mathrel{\Delta_{\alpha \alpha}} \begin{bmatrix} l \circ f(\bar{x}) \\ f(l(\bar{x})) \end{bmatrix}.
\end{align}
Now define $T^{\gamma}_{f}(\bar{x}) := \begin{bmatrix} \gamma \circ l(f(\bar{x})) \\ f(\gamma \circ l (\bar{x})) \end{bmatrix}$ for $f \in \tau$. It follows that $T^{\gamma}$ is 2-cocycle for $A'$. By a similar argument as in Proposition~\ref{prop:10}, there is $h$ such that
\[
T'_{f}(\bar{x}) - T^{\gamma}_{f}(\bar{x}) = f(h(\bar{x}))  - h(f(\bar{x})) \quad \quad \quad \quad \quad \quad \quad (f \in \tau)
\]
which can be expanded out to represent a 2-coboundary. The second condition guarantees that $T^{\gamma} = T$ and so $A$ and $A'$ are equivalent.
\end{proof}

\begin{definition}
Let $(A^{\alpha, \tau},Q,\ast)$ be affine datum. A function $d: Q \rightarrow A(\alpha)/\Delta_{\alpha \alpha}$ is a 1-\emph{cocycle} if for any $f \in \tau, \ar f = n$,
\[
d(f^{Q}(x_1,\ldots,x_n)) = f^{\Delta}(d(x_1),\delta(l(x_2)),\ldots,\delta(l(x_n))) +_{u} \sum_{i=2}^{n} a(f,i) ( x_1,\ldots,x_{i-1},d(x_{i}),x_{i+1},\ldots,x_n )
\]
where $l: Q \rightarrow A$ is any lifting associated to the datum. The set of 1-cocycles is denoted by $Z^{1}(A^{\alpha, \tau},Q,\ast)$
\end{definition}

We will also refer to 1-cocycles as \emph{derivations} of the datum.

\begin{lemma}
Let $(A^{\alpha, \tau},Q,\ast)$ be affine datum. For any lifting $l:Q \rightarrow A$ associated to the datum, the operation $(d+d')(x):= d(x) +_{l(x)} d'(x)$ makes $Z^{1}(A^{\alpha, \tau}Q,\ast)$ into an abelian group.
\end{lemma}
\begin{proof}
See the first paragraph of Lemma~\ref{lem:5}.
\end{proof}

Let $(A^{\alpha, \tau},Q,\ast)$ be affine datum. If $\pi: A \rightarrow Q$ is an extension realizing the datum, let $\mathrm{Stab}(\pi:A \rightarrow Q)$ denote the set of \emph{stabilizing automorphisms}; that is, the automorphism $\gamma \in \Aut A$ which satisfy conditions (1) and (2) in Theorem~\ref{thm:13}. The following is our characterization of the derivations of datum by stabilizing automorphisms.

\begin{theorem}\label{thm:stabilize}
Let $(Q,A^{\alpha,\tau},\ast)$ be affine datum. Let $\pi: A \rightarrow Q$ be an extension realizing the datum. Then $Z^{1}(Q,A^{\alpha,\tau},\ast) \approx \mathrm{Stab}(\pi:A \rightarrow Q)$.
\end{theorem}
\begin{proof}
Let $\phi: A \rightarrow A(\alpha,\ast, T)$ witness the isomorphism from Theorem~\ref{thm:12}. Take $\gamma \in \mathrm{Stab}(\pi:A \rightarrow Q)$ and $\alpha$-trace $r$. Since $\pi \circ \gamma = \pi$, we have $(x,\gamma(x)) \in \alpha = \ker \pi$ which implies $r(\gamma(x))=r(x)$. Then by condition (2) in Theorem~\ref{thm:13} we can write
\begin{align*}
\phi \circ \gamma(x) = \begin{bmatrix}  r(\gamma(x)) \\ \gamma(x) \end{bmatrix}/\Delta_{\alpha \alpha} = \begin{bmatrix}  r(x) \\ \gamma(x) \end{bmatrix}/\Delta_{\alpha \alpha} &= \begin{bmatrix}  r(x) \\ x \end{bmatrix}/\Delta_{\alpha \alpha} +_{r(x)} \begin{bmatrix}  x \\ \gamma(x) \end{bmatrix}/\Delta_{\alpha \alpha} \\
&= \begin{bmatrix}  r(x) \\ x \end{bmatrix}/\Delta_{\alpha \alpha} +_{r(x)} \begin{bmatrix}  r(x) \\ \gamma(r(x)) \end{bmatrix}/\Delta_{\alpha \alpha} \\
&= \phi(x) +_{r(x)} d_{\gamma}(\pi(x))
\end{align*}
where we have defined $d_{\gamma}(x) := \begin{bmatrix}  l(x) \\ \gamma(l(x)) \end{bmatrix}/\Delta_{\alpha \alpha} : Q \rightarrow A(\alpha)/\Delta_{\alpha \alpha}$ for the lifting associated to $r$. Note $r$ and $\gamma \circ r$ are both $\alpha$-traces. Then the argument in Proposition~\ref{prop:10} shows for all $f \in \tau$ and $\bar{x} \in Q^{\ar f}$,
\begin{align*}
\delta(u) = T_{f}(\bar{x}) -_{u} T_{f}(\bar{x}) &= f(d_{\gamma}(x_1),\delta(l(x_2)),\ldots,\delta(l(x_n))) -_{u} d_{\gamma}(f(x_1,\ldots,x_n)) \\
&+_{u} \ \sum_{i=2}^{n} a(f,i)\left(x_1,\ldots,x_{i-1},d_{\gamma}(x_{i}),x_{i+1},\ldots,x_n)\right)
\end{align*}
where $u=l(f(\bar{x}))$; therefore, $d_{\gamma}$ is a derivation. We should note $\phi \circ \gamma(r(x)) = d_{\gamma}(\pi(x))$.

We show $\mathrm{Stab}(\pi:A \rightarrow Q)$ is closed under composition. Take $\gamma,\gamma' \in \mathrm{Stab}(\pi:A \rightarrow Q)$; clearly, condition (1) in Theorem~\ref{thm:13} holds. We first calculate
\begin{align}\label{eq:28}
\phi \circ (\gamma' \circ \gamma)(x) = \phi(\gamma(x)) +_{r(\gamma(x))} d_{\gamma'}(\pi(\gamma(x))) = \phi(x) +_{r(x)} d_{\gamma}(\pi(x)) +_{r(x)} d_{\gamma'}(\pi(x)).
\end{align}
Since $d_{\gamma}$ is a derivation, we have $\pi \circ d_{\gamma} = \id$ which implies we can write $d_{\gamma}(x) = \begin{bmatrix} l(x) \\ \beta(x) \end{bmatrix}/\Delta_{\alpha \alpha}$ for some function $\beta: Q \rightarrow A$; similarly, $d_{\gamma'}(x) = \begin{bmatrix} l(x) \\ \beta'(x) \end{bmatrix}/\Delta_{\alpha \alpha}$ for some function $\beta': Q \rightarrow A$. Then using Eq.~\eqref{eq:28}, we have
\begin{align*}
\begin{bmatrix} x \\ (\gamma' \circ \gamma)(x) \end{bmatrix}/\Delta_{\alpha \alpha} &= \begin{bmatrix} x \\ \phi^{-1} \left( \phi(x) +_{r(x)} d_{\gamma}(\pi(x)) +_{r(x)} d_{\gamma'}(\pi(x)) \right) \end{bmatrix}/\Delta_{\alpha \alpha} \\
&= \begin{bmatrix} x \\  \phi^{-1} \left( \begin{bmatrix} m(r(x),r(x),m(r(x),r(x),r(x))) \\ m(x,r(x),m(\beta'(\pi(x)),r(x),\beta(\pi(x)))) \end{bmatrix}/\Delta_{\alpha \alpha} \right) \end{bmatrix}/\Delta_{\alpha \alpha} \\
&= \begin{bmatrix} x \\ m(x,r(x),m(\beta'(x),r(x),\beta(\pi(x)))) \end{bmatrix}/\Delta_{\alpha \alpha} \\
&= m \left( \begin{bmatrix} x \\ x \end{bmatrix}, \begin{bmatrix} r(x) \\ r(x) \end{bmatrix}, m \left( \begin{bmatrix} r(x) \\ \beta'(\pi(x)) \end{bmatrix},\begin{bmatrix} r(x) \\ r(x) \end{bmatrix}, \begin{bmatrix} r(x) \\ \beta(\pi(x)) \end{bmatrix} \right) \right)/\Delta_{\alpha \alpha} \\
&= \begin{bmatrix} r(x) \\ \beta'(\pi(x)) \end{bmatrix}/\Delta_{\alpha \alpha} +_{r(x)} \begin{bmatrix} r(x) \\ \beta(\pi(x)) \end{bmatrix}/\Delta_{\alpha \alpha} \\
&= d_{\gamma'}(\pi(x)) +_{r(x)} d_{\gamma}(\pi(x)) \\
&= \phi \circ (\gamma' \circ \gamma)(r(x)) = \begin{bmatrix} r(x) \\ (\gamma' \circ \gamma)(r(x)) \end{bmatrix}/\Delta_{\alpha \alpha}.
\end{align*}
since $r(\gamma'(\gamma(r(x)))) = r(\gamma(r(x))) = r(r(x))=r(x)$. This shows $\gamma' \circ \gamma$ satisfies condition (2) and so $\gamma' \circ \gamma \in \mathrm{Stab}(\pi:A \rightarrow Q)$.

Define $\Psi : \mathrm{Stab}(\pi:A \rightarrow Q) \rightarrow Z^{1}(Q,A^{\alpha,\tau},\ast)$ by $\Psi(\gamma) := d_{\gamma}$ where $\phi \circ \gamma(r(x)) = d_{\gamma}(\pi(x))$. It is easy to see that $\Psi$ is injective. To show surjectivity, take a derivation $d$ and define $\gamma_{d}$ by $\phi \circ \gamma_{d}(x) := \phi(x) +_{r(x)} d(\pi(x))$. The proof of Theorem~\ref{thm:13} shows $\gamma_{d}$ is a stabilizing automorphism. We show it satisfies condition (2) above. Since $d$ is a derivation, we can write $d(x) = \begin{bmatrix} l(x) \\ \beta(x) \end{bmatrix}/\Delta_{\alpha \alpha}$ for some function $\beta: Q \rightarrow A$. We calculate
\begin{align*}
\begin{bmatrix} x \\ \gamma_{d}(x) \end{bmatrix}/\Delta_{\alpha \alpha} = \begin{bmatrix} x \\ \phi^{-1} \left( \phi(x) +_{r(x)} d(\pi(x)) \right) \end{bmatrix}/\Delta_{\alpha \alpha} &= \begin{bmatrix} x \\ \phi^{-1} \left( \begin{bmatrix} m \big( r(x),r(x),r(x) \big) \\ m \big(x,r(x),\beta(\pi(x))\big) \end{bmatrix}/\Delta_{\alpha \alpha} \right) \end{bmatrix}/\Delta_{\alpha \alpha} \\
&= \begin{bmatrix} x \\ m(x,r(x),\beta(\pi(x))) \end{bmatrix}/\Delta_{\alpha \alpha} \\
&= m \left( \begin{bmatrix} x \\ x \end{bmatrix}, \begin{bmatrix} r(x) \\ r(x) \end{bmatrix},\begin{bmatrix} r(x) \\ \beta(\pi(x)) \end{bmatrix} \right)/\Delta_{\alpha \alpha} \\
&= \begin{bmatrix} r(x) \\ \beta(\pi(x)) \end{bmatrix}/\Delta_{\alpha \alpha} = \begin{bmatrix} r(x) \\ \gamma_{d}(r(x)) \end{bmatrix}/\Delta_{\alpha \alpha}
\end{align*}
This implies $m(\gamma_{d}(r(x)),r(x),x)=\gamma_{d}(x)$ and so $\gamma_{d} \in \mathrm{Stab}(\pi:A \rightarrow Q)$. We see that $\Psi(\gamma_{d})=d$ since $\phi \circ \gamma(r(x)) = \phi(r(x)) +_{r(x)} d(\pi(r(x))) = d(\pi(x))$.

To finish the theorem, we show $\Psi$ is a homomorphism. This follows from Eq.~\eqref{eq:28} since
\[
\phi \circ (\gamma' \circ \gamma)(r(x)) = d_{\gamma'}(\pi(x)) +_{r(x)} d_{\gamma}(\pi(x)) = (d_{\gamma'} + d_{\gamma})(\pi(x))
\]
implies $\Psi(\gamma' \circ \gamma) = d_{\gamma'} + d_{\gamma}$.
\end{proof}

Suppose $\pi: A \rightarrow Q$ with $\alpha = \ker \pi$ is an extension realizing affine datum. By Theorem~\ref{thm:10} we have the isomorphism $\phi: A \rightarrow A_{T}(Q,A^{\alpha,\tau},\ast)$ given by $a \longmapsto \begin{bmatrix} r(a) \\ a \end{bmatrix}/\Delta_{\alpha \alpha}$ for any $\alpha$-trace $r: A \rightarrow A$. For the semidirect product we have the homomorphisms
\begin{align*}
\rho: A(\alpha)/\Delta_{\alpha \alpha} \rightarrow Q &, \quad \quad \begin{bmatrix} a \\ b \end{bmatrix}/\Delta_{\alpha \alpha} \longmapsto \pi(a) \\
\kappa: A(\alpha)/\Delta_{\alpha \alpha} \rightarrow A(\alpha)/\Delta_{\alpha 1} &, \quad \quad \begin{bmatrix} a \\ b \end{bmatrix}/\Delta_{\alpha \alpha} \longmapsto \begin{bmatrix} a \\ b \end{bmatrix}/\Delta_{\alpha 1}
\end{align*}
where $\kappa$ is the canonical homomorphism. We then take the solution in the product diagram
\[
\sigma: A(\alpha)/\Delta_{\alpha \alpha} \rightarrow A(\alpha)/\Delta_{\alpha 1} \times Q, \quad \quad \begin{bmatrix} a \\ b \end{bmatrix}/\Delta_{\alpha \alpha} \longmapsto \left\langle \begin{bmatrix} a \\ b \end{bmatrix}/\Delta_{\alpha 1}  , \pi(a) \right\rangle.
\]
It is easy to see that the same definition also induces a homomorphism
\[
\sigma: A_{T}(Q,A^{\alpha,\tau},\ast) \rightarrow A(\alpha)/\Delta_{\alpha 1} \otimes^{\kappa \circ T} Q
\]
with the following properties:
\begin{enumerate}

	\item $\ker \kappa \circ \ker \rho = 1$ and $\ker \sigma = \ker \kappa \wedge \ker \rho$ where $\ker \kappa = \Delta_{\alpha 1}/\Delta_{\alpha \alpha}$, $\ker \rho = \hat{\alpha}/\Delta_{\alpha \alpha}$;
	
	\item if $\mathcal V(A)$ has a weak-difference term and $\alpha$ is left-central, then $0 = \phi([\alpha,1]) = \ker \sigma$ and $\sigma$ is a subdirect embedding;
	
	\item if $\mathcal V(A)$ has a difference term and $\alpha$ is central, then $\sigma$ is surjective.

\end{enumerate}
In this way Proposition~\ref{thm:extension} follows from Theorem~\ref{thm:10} if we assume $\mathcal V$ has a difference term. In the following, we recover what appears to be a folklore result.

\begin{lemma}\label{lem:centrality}
Let $\mathcal V$ be a variety with a difference term and $A \in \mathcal V$. If $\alpha \in \Con A$ is central, then
\[
A(\alpha)/\Delta_{\alpha \alpha} \approx A(\alpha)/\Delta_{\alpha 1} \times A/\alpha.
\]
\end{lemma}
\begin{proof}
By Corollary~\ref{cor:1} and Theorem~\ref{thm:10}, there is a 2-cocycle $T$ and action $\ast$ such that $A_{T}(Q,A^{\alpha,\tau},\ast) \approx A \approx A(\alpha)/\Delta_{\alpha 1} \otimes^{\kappa \circ T} A/\alpha$. If we take the trivial 2-cocycle $T=0$, then $A(\alpha)/\Delta_{\alpha \alpha} \approx A_{0}(Q,A^{\alpha,\tau},\ast)  \approx A(\alpha)/\Delta_{\alpha 1} \times A/\alpha$.
\end{proof}

\begin{definition}\label{def:trivialaction}
Let $(Q,A^{\alpha,\tau},\ast)$ be affine datum. The action is \emph{trivial} if the following holds: for any $f \in \tau$ with $\ar f > 1$ and $I \subseteq \{ 1,\ldots,\ar f \}$, for any $\bar{c},\bar{c}' \in A^{\ar f}$ such that $c_{i} = c_{i}'$ for $i \in I$ and $\rho(\delta(c_{i})) = q_{i}, \rho(\delta(c_{i}')) = q_{i}$, and for any $\alpha$-trace $r:A \rightarrow A$ with associated lifting $l$, then  
\begin{align*}
\sum_{i \in I} a(f,i) &\left(q_{1},\ldots, q_{i-1}, \begin{bmatrix} r(c_{i}) \\ c_{i} \end{bmatrix}/\Delta_{\alpha \alpha}, q_{i+1},\ldots,q_{n} \right) \\
&= \delta(u) +_{u'} \sum_{i \in I}a(f,i) \left(q_{1}',\ldots, q_{i-1}', \begin{bmatrix} r(c_{i}') \\ c_{i}' \end{bmatrix}/\Delta_{\alpha \alpha}, q_{i+1}',\ldots,q_{n}' \right)
\end{align*}
where $u=l \left( f^{Q}(q_{1},\ldots,q_{n}) \right)$ and $u'=l \left( f^{Q}(q_{1},\ldots,q_{n}) \right)$.
\end{definition}

For an algebra $A$, a congruence $\alpha \in \Con A$ is \emph{right-central} if $[1,\alpha]=0$ and \emph{left-central} if $[\alpha,1]=0$.

\begin{proposition}\label{prop:centralaction}
Let $(Q,A^{\alpha,\tau},\ast)$ be affine datum contained in some variety with a difference term. The action is trivial if and only if every extension in a variety with a difference term realizing the datum is a central extension.
\end{proposition}
\begin{proof}
Assume every extension in a variety with a difference term realizing the datum is a central extension. By assumption, there is at least one such extension. Taking any such extension $\pi: A \rightarrow Q$ with $A \in \mathcal V$ a variety with a difference term. By Theorem~\ref{thm:12}, we have $A \approx A_{T}(Q,A^{\alpha,\tau},\ast)$ for some compatible 2-cocycle. Fix an $\alpha$-trace $r: A \rightarrow A$ with associated lifting $l$. Take $f \in \tau$ with $n= \ar f > 1$ and choose $I \subseteq \{1,\ldots,n \}$. Choose $\bar{c},\bar{c}' \in A^{n}$ such that $c_{i} = c_{i}'$ for $i \in I$. Let $\pi(c_{i})=q_{i}$ and $\pi(c_{i}')=q_{i}'$ for $i=1,\ldots,n$. Then by realization we have for $i \in I$
\begin{align*}
a(f,i)\left( q_1,\ldots,q_{i-1},\begin{bmatrix} r(c_{i}) \\ c_{i} \end{bmatrix}/\Delta_{\alpha \alpha},q_{i+1},\ldots,q_{n}\right) &= \begin{bmatrix} u \\ a_i \end{bmatrix}/\Delta_{\alpha \alpha} \\
a(f,i)\left(q'_1,\ldots,q'_{i-1},\begin{bmatrix} r(c_{i}') \\ c_{i}' \end{bmatrix}/\Delta_{\alpha \alpha},q_{i+1}',\ldots,q_{n}' \right) &= \begin{bmatrix} u' \\ b_i \end{bmatrix}/\Delta_{\alpha \alpha}.
\end{align*}
where $u=f(r(c_1),\ldots,r(c_{n}))$, $u'=f(r(c_{1}'),\ldots,r(c_{n}'))$ and $a_i=f(r(c_{1}),\ldots,r(c_{i}),c_{i},r(c_{i+1}),\ldots,r(c_{n}))$, $b_i=f(r(c_{1}'),\ldots,r(c_{i}'),c_{i}',r(c_{i+1}'),\ldots,r(c_{n}'))$. Note the diagonal is a single $\Delta_{\alpha 1}$-class; that is, $\begin{bmatrix} x \\ x \end{bmatrix} \Delta_{\alpha 1} \begin{bmatrix} y \\ y \end{bmatrix}$ for all $x,y \in A$. Denote the $\Delta_{\alpha 1}$-class of the diagonal by $\hat{\delta}$. Then by passing to the quotient to $A(\alpha)/\Delta_{\alpha 1}$ we have
\begin{align*}
\begin{bmatrix} u \\ \sum_{i \in I}^{u} a_{i}  \end{bmatrix}/\Delta_{\alpha 1} &= \sum_{i \in I}^{u} \begin{bmatrix} u \\ a_{i} \end{bmatrix}/\Delta_{\alpha \alpha} \\
&= \sum_{i \in I}^{u} f \left( \delta(c_{1}),\ldots,\delta(c_{i-1}),\begin{bmatrix} r(c_{i}) \\ c_{i} \end{bmatrix}/\Delta_{\alpha \alpha}, \delta(c_{i+1}),\ldots,\delta(c_{n}) \right)/(\Delta_{\alpha 1}/\Delta_{\alpha \alpha}) \\
&= \sum_{i \in I}^{u} f\left( \hat{\delta},\ldots,\hat{\delta},\begin{bmatrix} r(c_{i}) \\ c_{i} \end{bmatrix},\hat{\delta},\ldots,\hat{\delta} \right)/\Delta_{\alpha 1} \\
&=  \sum_{i \in I}^{u} f \left( \delta(c_{1}'),\ldots,\delta(c_{i-1}'),\begin{bmatrix} r(c_{i}') \\ c_{i}' \end{bmatrix}/\Delta_{\alpha \alpha}, \delta(c_{i+1}'),\ldots,\delta(c_{n}') \right)/(\Delta_{\alpha 1}/\Delta_{\alpha \alpha})     \\
&= \sum_{i \in I}^{u'} \begin{bmatrix} u' \\ b_{i} \end{bmatrix}/\Delta_{\alpha 1} \\
&= \begin{bmatrix} u' \\ \sum_{i \in I}^{u'} b_{i}  \end{bmatrix}/\Delta_{\alpha 1}
\end{align*}
Since $\alpha$ is central, by Lemma~\ref{lem:20}(1) we conclude $\sum_{i \in I}^{u} a_{i} = \left( \sum_{i \in I}^{u'} b_{i} \right) +_{u'} u$. Write $I = \{ i_{1},\ldots,i_{k} \}$. Using the difference term identity $m(u',u',u)=u$ we see that
\begin{align*}
\sum_{i \in I}^{u} \begin{bmatrix} u \\ a_{i} \end{bmatrix}/\Delta_{\alpha \alpha} = \begin{bmatrix} u \\ a_{i_1} +_{u} \cdots +_{u} a_{i_k} \end{bmatrix}/\Delta_{\alpha \alpha} &=  \begin{bmatrix} u' +_{u'} u   \\ \left( b_{i_1} +_{u'} \cdots +_{u'} b_{i_k} \right) +_{u'} u \end{bmatrix}/\Delta_{\alpha \alpha}  \\
&= \begin{bmatrix} u' \\ \left( b_{i_1} +_{u'} \cdots +_{u'} b_{i_k} \right)  \end{bmatrix}/\Delta_{\alpha \alpha} +_{u'} \delta(u) \\
&= \sum_{i \in I}^{u'} \begin{bmatrix} u' \\ b_{i} \end{bmatrix}/\Delta_{\alpha \alpha} +_{u'} \delta(u)
\end{align*}

Now assume the action is trivial. Consider an extension $A \approx A_{T}(Q,A^{\alpha,\tau},\ast)$ realizing the datum in a variety with a difference term. Fix an $\alpha$-trace $r$ and associated lifting $l$. We directly verify the term condition for centrality in any extension realizing the datum. Take $f \in \tau$ with $1 \leq k < n=\ar f$ and $\bar{a},\bar{b} \in A^{k}$, $\bar{c},\bar{d} \in A^{n-k}$ such that $(c_i,d_i) \in \alpha$ for $i=1,\ldots,n-k$. Assume
\begin{align}\label{eq:26}
F_{f}\left(\begin{bmatrix} r(\bar{a}) \\ \bar{a} \end{bmatrix}/\Delta_{\alpha \alpha},\begin{bmatrix} r(\bar{c}) \\ \bar{c} \end{bmatrix}/\Delta_{\alpha \alpha}\right) = F_{f}\left(\begin{bmatrix} r(\bar{a}) \\ \bar{a} \end{bmatrix}/\Delta_{\alpha \alpha},\begin{bmatrix} r(\bar{d}) \\ \bar{d} \end{bmatrix}/\Delta_{\alpha \alpha}\right).
\end{align}
By realization, we can combine the terms in the definition of the operation $F_{f}$ to rewrite Eq.(\ref{eq:26}) as
\begin{align*}
f\left( \begin{bmatrix} r(\bar{a}) \\ \bar{a} \end{bmatrix}/\Delta_{\alpha \alpha},\delta(\bar{c}) \right) &+_{u} f\left(\delta(r(\bar{a})), \begin{bmatrix} r(\bar{c}) \\ \bar{c} \end{bmatrix}/\Delta_{\alpha \alpha} \right) +_{u} T_{f}(\pi(\bar{a}),\pi(\bar{c}))  \\
&= f\left( \begin{bmatrix} r(\bar{a}) \\ \bar{a} \end{bmatrix}/\Delta_{\alpha \alpha},\delta(\bar{d}) \right) +_{u} f\left(\delta(r(\bar{a})), \begin{bmatrix} r(\bar{d}) \\ \bar{d} \end{bmatrix}/\Delta_{\alpha \alpha} \right) +_{u} T_{f}(\pi(\bar{a}),\pi(\bar{d}))
\end{align*}
Then $(c_i,d_i) \in \alpha \Rightarrow \delta(c_i) = \delta(d_i) \text{ and } \pi(c_i) = \pi(d_i)$. This implies that in the above expression the first terms on the left-side and right-side and the 2-cocycle terms are equal. After canceling we conclude
\begin{align}\label{eq:27}
f\left(\delta(r(\bar{a})), \begin{bmatrix} r(\bar{c}) \\ \bar{c} \end{bmatrix}/\Delta_{\alpha \alpha} \right) = f\left(\delta(r(\bar{a})), \begin{bmatrix} r(\bar{d}) \\ \bar{d} \end{bmatrix}/\Delta_{\alpha \alpha} \right).
\end{align}
We can write $u = l(f(\pi(\bar{a}),\pi(\bar{c}))) = l(f(\pi(\bar{a}),\pi(\bar{d})))$ and $u' = l(f(\pi(\bar{b}),\pi(\bar{c}))) = l(f(\pi(\bar{b}),\pi(\bar{d})))$. Then we can use realization to re-expand Eq.(\ref{eq:27}) into action terms and apply triviality to yield 
\begin{align*}
f &\left(\delta(r(\bar{b})) , \begin{bmatrix} r(\bar{c}) \\ \bar{c} \end{bmatrix}/\Delta_{\alpha \alpha} \right) \displaybreak[0]\\
&= \sum_{i=1}^{n-k} a(f,i)\left(\pi(\bar{b}),\pi(c_{1}),\ldots,\pi(c_{i-1}),\begin{bmatrix} r(c_i) \\ c_i \end{bmatrix}/\Delta_{\alpha \alpha},\pi(c_{i+1}),\ldots,\pi(c_{n-k}) \right) \displaybreak[0]\\
&= \sum_{i=1}^{n-k} a(f,i)\left(\pi(\bar{a}),\pi(c_{1}),\ldots,\pi(c_{i-1}),\begin{bmatrix} r(c_i) \\ c_i \end{bmatrix}/\Delta_{\alpha \alpha},\pi(c_{i+1}),\ldots,\pi(c_{n-k}) \right) +_{u} \delta(u') \displaybreak[0]\\
&= f\left(\delta(r(\bar{a})), \begin{bmatrix} r(\bar{c}) \\ \bar{c} \end{bmatrix}/\Delta_{\alpha \alpha} \right) +_{u} \delta(u') \displaybreak[0]\\
&= f\left(\delta(r(\bar{a})), \begin{bmatrix} r(\bar{d}) \\ \bar{d} \end{bmatrix}/\Delta_{\alpha \alpha} \right) +_{u} \delta(u') \displaybreak[0]\\
&= \sum_{i=1}^{n-k} a(f,i)\left(\pi(\bar{a}),\pi(d_{1}),\ldots,\pi(d_{i-1}),\begin{bmatrix} r(d_i) \\ d_i \end{bmatrix}/\Delta_{\alpha \alpha},\pi(d_{i+1}),\ldots,\pi(d_{n-k}) \right) +_{u} \delta(u') \displaybreak[0]\\
&= \sum_{i=1}^{n-k} a(f,i)\left(\pi(\bar{b}),\pi(d_{1}),\ldots,\pi(d_{i-1}),\begin{bmatrix} r(d_i) \\ d_i \end{bmatrix}/\Delta_{\alpha \alpha},\pi(d_{i+1}),\ldots,\pi(d_{n-k}) \right)  \displaybreak[0]\\
&= f\left(\delta(r(\bar{b})), \begin{bmatrix} r(\bar{d}) \\ \bar{d} \end{bmatrix}/\Delta_{\alpha \alpha} \right).
\end{align*}
Again, $\delta(c_i) = \delta(d_i) \text{ and } \pi(c_i) = \pi(d_i)$ imply
\begin{align*}
f\left( \begin{bmatrix} r(\bar{b}) \\ \bar{b} \end{bmatrix}/\Delta_{\alpha \alpha},\delta(\bar{c}) \right) = f\left( \begin{bmatrix} r(\bar{b}) \\ \bar{b} \end{bmatrix}/\Delta_{\alpha \alpha},\delta(\bar{d}) \right) \quad \quad \text{and} \quad \quad T_{f}(\pi(\bar{b}),\pi(\bar{c})) = T_{f}(\pi(\bar{b}),\pi(\bar{d})).
\end{align*}
Putting the last three systems of equations together we conclude that
\begin{align*}
F_{f}\left(\begin{bmatrix} r(\bar{b}) \\ \bar{b} \end{bmatrix}/\Delta_{\alpha \alpha},\begin{bmatrix} r(\bar{c}) \\ \bar{c} \end{bmatrix}/\Delta_{\alpha \alpha}\right) = F_{f}\left(\begin{bmatrix} r(\bar{b}) \\ \bar{b} \end{bmatrix}/\Delta_{\alpha \alpha},\begin{bmatrix} r(\bar{d}) \\ \bar{d} \end{bmatrix}/\Delta_{\alpha \alpha}\right).
\end{align*}
which shows $[1,\alpha]=0$ in $A \approx A_{T}(Q,A^{\alpha,\tau},\ast)$.
\end{proof}

\begin{remark}\label{rmk:1}
The proof of necessity in Proposition~\ref{prop:centralaction} does not require the hypothesis of a Mal'cev condition in the variety generated by any extension of the datum; that is, if the action is trivial, then the kernel in every extension realizing the datum is right-central.
\end{remark}

\begin{remark}
Suppose $(Q,A^{\alpha,\tau},\ast)$ be affine datum with trivial action realized by $A$. Let $\kappa: A(\alpha)/\Delta_{\alpha \alpha} \rightarrow A(\alpha)/\Delta_{\alpha 1}$ be the canonical homomorphism. Then 
\[
\kappa \circ a(f,i)(y_1,\ldots,y_{i-1},x,y_{i+1},\ldots,y_{n})
\] 
depends only on the i-th coordinate. 
\end{remark}

In the case of affine datum with trivial action, the second-cohomology $H^{2}_{\mathcal U}(Q,A^{\alpha,\tau},\ast)$ is an abelian group of equivalence classes of right-central extensions which realize the datum; however, to guarantee that we recover every central extension realizing the datum we must rely on Proposition~\ref{prop:centralaction} together with Theorem~\ref{thm:cohomology} in the case of varieties with a difference term.

\begin{theorem}\label{thm:centralcohom}
Let $(Q,A^{\alpha,\tau},\ast)$ be affine datum with trivial action and $\mathcal U$ a variety with a difference term containing the datum. The abelian group $H^{2}_{\mathcal U}(Q,A^{\alpha,\tau},\ast)$ is in bijective correspondence with the set of equivalence classes of central extensions in $\mathcal U$ realizing the datum. 
\end{theorem}

We offer a tentative definition for a $1^{\mathrm{st}}$-cohomology group associated to affine datum. The abelian group of derivations is isomoprhic to the group of  stabilizing automorphisms of an extension. In the case of groups, the principal derivations correspond to the stabilizing automorphisms of the semidirect product which act by conjugation. The approach to principal derivation for affine datum follows in this line and finds application in Wires \cite{wiresVI} for a low-dimensional Hochschild-Serre sequence associated to a general extension with an additional affine action.

For any algebra $A$ and $\alpha \in \Con A$, unary functions $f,h: A \rightarrow A$ are $\alpha$-twins if there is a term $t(x,\bar{y})$ and $\bar{c} \mathrel{\alpha} \bar{d}$ such that $f(x)=t(x,\bar{c})$ and $h(x)=t(x,\bar{d})$. The set of $\alpha$-twins of the identity is denoted by $\mathrm{Tw}_{\alpha} A$; that is, $\gamma \in \mathrm{Tw}_{\alpha} A$ if there there is a term $t(x,\bar{y})$ and $\bar{c} \mathrel{\alpha} \bar{d}$ such that $\gamma(x)=t(x,\bar{c})$ and $x=t(x,\bar{d})$. We can restrict to a subset  
\[
\mathrm{Tw}_{\alpha,F} A = \{\gamma \in \mathrm{Tw}_{\alpha} A: ( \exists x \in A), \gamma(x)=x \}
\]
of those $\alpha$-twins of the identity which have a fixed-point. In general, $\mathrm{Tw}_{\alpha} A$ is closed under composition and $\mathrm{Tw} A$ is closed under conjugation by automorphisms of $A$, but $\mathrm{Tw}_{\alpha,F} A$ is neither. Given affine datum $(Q,A^{\alpha,\tau},\ast)$ we consider the set of \emph{principal stabilizing automorphisms}
\[
\mathrm{PStab}(Q,A^{\alpha,\tau},\ast) = \mathrm{Tw}_{\hat{\alpha}/\Delta_{\alpha \alpha},F} A(\alpha)/\Delta_{\alpha \alpha} \cap \mathrm{Stab}(\pi: A(\alpha)/\Delta_{\alpha \alpha} \rightarrow Q ).
\]

\begin{definition}
Let $(Q,A^{\alpha,\tau},\ast)$ be affine datum. A map $d: Q \rightarrow A(\alpha)/\Delta_{\alpha \alpha}$ is a \emph{principal derivation} if $d(\pi(x)) +_{l(x)} x \in \mathrm{PStab}(Q,A^{\alpha,\tau},\ast)$ for any lifting $l$ of $\pi: A(\alpha)/\Delta_{\alpha \alpha} \rightarrow Q$.
\end{definition}

If we simplify notation and write the extension of the datum $A \approx A(\alpha)/\Delta_{\alpha \alpha} \stackrel{\pi}{\rightarrow} Q$ isomorphic to the semidirect product witnessed by $\phi$ from Theorem~\ref{thm:12}, then a principal derivation $d$ takes the form
\[
\phi(\gamma(x)) = d(\pi(x)) +_{r(x)} \begin{bmatrix} r(x) \\ x \end{bmatrix}/\Delta_{\alpha \alpha}
\]
for some $\gamma \in \mathrm{PStab}(Q,A^{\alpha,\tau},\ast)$ and any $\alpha$-trace $r: A \rightarrow A$, and are just those derivations which correspond under the isomorphism of Theorem~\ref{thm:stabilize} to principal stabilizing automorphisms. Denote by $\mathrm{PDer}(Q,A^{\alpha,\tau},\ast)$ the subgroup of $Z^{1}(Q,A^{\alpha,\tau},\ast)$ generated by the principal derivations.

\begin{definition}
Let $(Q,A^{\alpha,\tau},\ast)$ be affine datum. The $1^{\mathrm{st}}$-cohomology of the datum is defined as the quotient group
\[
H^{1}(Q,A^{\alpha,\tau},\ast):= Z^{1}(Q,A^{\alpha,\tau},\ast)/\mathrm{PDer}(Q,A^{\alpha,\tau},\ast).
\]
\end{definition}

\begin{lemma}
Let $(Q,A^{\alpha,\tau},\ast)$ be affine datum with trivial action. Then $H^{1}(Q,A^{\alpha,\tau},\ast) \approx \mathrm{Stab}(\pi: A(\alpha)/\Delta_{\alpha \alpha} \rightarrow Q)$.
\end{lemma}
\begin{proof}
By Remark~\ref{rmk:1}, the kernel in the semidirect product realizing the datum is right-central. For $\gamma \in \mathrm{PStab}(Q,A^{\alpha,\tau},\ast)$, we have $\gamma(x)=t(x,\bar{c})$, $x=t(x,\bar{d})$ for some term with $\bar{c} \mathrel{\alpha} \bar{d}$ and there exists $a \in A$ such that $a=\gamma(a)$. The matrix
\[
\begin{bmatrix} a & \gamma(x) \\ a & x \end{bmatrix} = \begin{bmatrix} \gamma(a) & \gamma(x) \\ a & x \end{bmatrix} = \begin{bmatrix} t(a,\bar{c}) & t(x,\bar{c}) \\ t(a,\bar{d}) & t(x,\bar{d}) \end{bmatrix} \in M(\alpha,1)
\]
implies $(\gamma(x),x) \in [1,\alpha]=0$; thus, the subgroup generated by principal stabilizing automorphism is trivial. Then Theorem~\ref{thm:stabilize} yields $H^{1}(Q,A^{\alpha,\tau},\ast) \approx \mathrm{Stab}(\pi: A(\alpha)/\Delta_{\alpha \alpha} \rightarrow Q)$.
\end{proof}

According to Proposition~\ref{prop:centralaction}, the hypothesis of the previous lemma holds for central extensions in varieties with a difference term.

\begin{definition}
Let $(Q,A^{\alpha,\tau},\ast)$ be affine datum with trivial action and $Q$ an abelian algebra. For a variety $\mathcal V$ in the same signature as the datum, $\mathrm{Ext}_{\mathcal V}(Q,A^{\alpha,\tau},\ast)$ denote the set of equivalence classes of $\mathcal V$-compatible 2-cocycles which represent extensions realizing the datum which are abelian algebras.
\end{definition}

The definition is well-defined since equivalence of extensions of datum is finer than isomorphism, and isomorphism preserves the abelian property of an algebra. The following is a far-reaching generalization of the classical observation from group theory where abelian extensions are characterized by symmetric 2-cocycles; in the case of more general varieties with a weak-difference term, the abelian extensions are characterized by the 2-cocycle identities (see Definition~\ref{def:2cocyle}(C2)) corresponding to the axiomatization of the abelian subvariety.

\begin{corollary}\label{cor:21}
Let $(Q,A^{\alpha,\tau},\ast)$ be affine datum with trivial action and $Q$ an abelian algebra. Let $\mathcal V$ be a variety with a weak-difference in the same signature as the datum. Then either $\mathrm{Ext}_{\mathcal V}(Q,A^{\alpha,\tau},\ast)$ is empty or $\mathrm{Ext}_{\mathcal V}(Q,A^{\alpha,\tau},\ast) \leq H^{2}_{\mathcal V}(Q,A^{\alpha,\tau},\ast)$.
\end{corollary}
\begin{proof}
Since $\mathcal V$ has a weak-difference term, the abelian algebras of $\mathcal V$ form a subvariety $\mathcal A$. If $\mathcal A$ does not contain the datum, then
$\mathrm{Ext}_{\mathcal V}(Q,A^{\alpha,\tau},\ast)$ is empty; otherwise, $\mathcal A \in \mathcal L(Q,A^{\alpha,\tau},\ast)$ and the embedding follows from Proposition~\ref{prop:14} after noting $\mathrm{Ext}_{\mathcal V}(Q,A^{\alpha,\tau},\ast) = H^{2}_{\mathcal A}(Q,A^{\alpha,\tau},\ast)$.
\end{proof}

\vspace{0.3cm}

\section{Varieties with a difference term redux}\label{sec:4}

In the case of varieties with a difference term and algebras with an idempotent element, the representations of extensions realizing affine datum in Theorem~\ref{thm:10} can sometimes take a form closely analogous to the group case with abelian normal subgroups in Lemma~\ref{lem:grp} and general central extensions in Corollary~\ref{cor:1}. In the variety of groups $\mathcal G$, the term $m(x,y,z) = x  y^{-1} z$ satisfies several identities; in particular, 
\begin{align}\label{eqn:malcev}
\mathcal G \vDash x=m(x,y,y) \mathrel{\wedge} x=m(y,y,x).
\end{align}
The existence of a ternary term $m$ satisfying the identities in Eq~\eqref{eqn:malcev} (a Mal'cev term) is an equational characterization for a variety having permuting congruences \cite{malcev}; however, the Mal'cev term derived from the group operations satisfies an additional identity
\begin{align}\label{eqn:coolterm}
\mathcal G \vDash x = m(m(x,y,z),z,y).
\end{align}
not satisfied by general congruence permutable varieties. This failure can be rather crudely measured by the congruence $\epsilon = \Cg^{A} \left( \left\{ \left\langle \ x \ , \ m(m(x,y,z),z,y) \ \right\rangle  : x,y,z \in A \right\} \right)$ in an algebra $A$ since $A/\epsilon$ satisfies Eq~\eqref{eqn:coolterm}. The representation in Theorem~\ref{thm:idemabelian} is in terms of the congruence $\epsilon$ relativized to a fixed class of an abelian congruence and choice of section.

Fix a variety $\mathcal V$ with a difference term $m$ and take an algebra $A \in \mathcal V$ and congruence $\alpha \in \Con A$; in addition, suppose $u \in A$ is an idempotent element and let $I_{\alpha} := u/\alpha$ denote the congruence class of $\alpha$ containing $u$ which is a subalgebra of $A$. Given a section $l: A/\alpha \rightarrow A$ of the canonical epimorphism $\pi : A \rightarrow A/\alpha$, let $\epsilon(l)$ be the congruence generated by the set of pairs 
\begin{align}\label{eq:2ndgener}
&\left\{ \left\langle \ x \ , \ m \big( m(x, u,l \circ \pi(y)),l \circ \pi(y), u \big) \ \right\rangle  : x \in I_{\alpha}, y \in A \right\} \\
&\left\{ \left\langle \ x \ , \ m \big( m(x, l \circ \pi(x), u),u,l \circ \pi(x) \big) \ \right\rangle  : x \in A \right\}.
\end{align}
Since $m$ is a difference term, $m(m(x, u,l \circ \pi(y)),l \circ \pi(y), u ) \mathrel{\alpha} m(m(u,u,l \circ \pi(x)),l\circ \pi(x),u) = u $ and $m(m(y,l \circ \pi(y),u),u,l \circ \pi(y) ) \mathrel{\alpha} m(m(y,y,u),u,l \circ \pi(y)) = l \circ \pi(y)$ for $x \in I_{y}, y \in A$; thus, the generating set of $\epsilon(l)$ is contained in $\alpha$. We also see from the centralizer condition that 
\begin{align*}
\begin{bmatrix} u & u \\ m(m(x,u,l(y)),l(y),u) & x \end{bmatrix} = \begin{bmatrix} m(m(u,u,l(y)),l(y),u) & m(m(u,u,u),u,u) \\ m(m(x,u,l(y)),l(y),u) & m(m(x,u,u),u,u) \end{bmatrix} \in M(\alpha,1)
\end{align*}
and similarly for the other set of generators in Eq~\eqref{eq:2ndgener}; thus, $\epsilon(l) \leq [\alpha,1]$.

The next step is to modify the construction of the algebra $B \otimes^{T} Q$ from Section~\ref{sec:2} to include an ``action''. Let $B$ and $Q$ be algebras in the same signature $\tau$ and suppose $x + y$ is a binary operation on $B$. A 2-cocycle appropriate to the pair $(Q,B)$ is a sequence $T = \{ T_{f} : f \in \tau \}$ with $T_{f} : Q^{\ar f} \rightarrow B$. An \emph{action} $Q \ast B$ is a sequence $\ast = \{ a(f,s) : f \in \tau, s \in \sigma(f) \subseteq [\ar f]^{\ast} \}$ with each $a(f,s) : Q^{\ar f} \times B^{\ar f} \rightarrow B$. It is a singleton action if $\sigma(f) = \{ \{i\} : i \in [ \ar f] \}$ the singleton sets of the coordinates of $f$. Given an appropriate 2-cocycle $T$ and action $Q \ast B$, the algebra $B \rtimes_{\ast,T} Q$ is defined over the universe of the direct product $B \times Q$ where each operation symbol $f \in \tau$ is interpreted by the rule
\[
f^{I_{\alpha} \rtimes_{\ast,T} Q} ( \left\langle a_{1}, x_{1} \right\rangle, \ldots, \left\langle a_{n}, x_{n} \right\rangle ) := \left\langle \ \sum_{f \in \tau, s \in \sigma(f)} a(f,s)(\vec{x},\vec{a}) + T_{f}(\vec{x}) \ , \ f^{Q}(\vec{x}) \ \right\rangle .
\]
In what follows, we will be interested in the case of singleton actions and when $x+y$ is the operation of an abelian group.

\begin{theorem}\label{thm:idemabelian}
Let $A \in \mathcal V$ a variety with a difference term and $u \in A$ an idempotent element. If $\alpha \in \Con A$ is abelian, then for each section $l: A/\alpha \rightarrow A$ there is a singleton action $A/\alpha \ast I_{\alpha/\epsilon(l)}$ homomorphic with respect to $x +_{u} y$ and a 2-cocycle $T$ appropriate to $(A/\alpha , I_{\alpha/\epsilon(l)})$ such that $A/\epsilon(l) \approx I_{\alpha/\epsilon(l)} \rtimes_{\ast,T} A/\alpha$.
\end{theorem}
\begin{proof}
Write $Q = A/\alpha$ and $\pi : A \rightarrow Q$ the canonical epimorphism with $\alpha = \ker \pi$. Fix a section $l:Q \rightarrow A$ of $\pi$ such that $l \circ \pi (u)=u$. Let $m$ be a difference term for $\mathcal V$. Since $\epsilon(l) \leq \alpha$, then $\alpha' = \alpha/\epsilon(l)$ is an abelian congruence of $A' = A/\epsilon(l)$ and $Q \approx A'/\alpha'$. Then we take the section $l' = \pi' \circ l$ of the canonical epimorphism $\pi': A' \rightarrow Q$ and note $u' = u/\alpha'$ is idempotent. Since $\alpha'$ is an abelian congruence, $I_{\alpha'}$ is an abelian subalgebra; consequently, $m(x,u',y) := x +_{u'} y$ is the binary operation of an abelian group with identity $u'$ when restricted to $I_{\alpha'}$.

We define a set map $\psi : I_{\alpha'} \times Q \rightarrow A'(\alpha')/\Delta_{\alpha' \alpha'}$ by $\psi (a,x) := \begin{bmatrix} l(x) \\ m(a,u,l(x)) \end{bmatrix}/\Delta_{\alpha' \alpha'}$. Suppose $\psi(a,x) = \psi(b,y)$. Since $\Delta_{\alpha' \alpha'} \leq \hat{\alpha'}$, we have $(l(x) ,l(y)) \in \alpha$ which implies $x=y$. Then in $A'$ we have $a = m(m(a,u,l(x)),l(x),u) = m(m(b,u,l(x)),l(x),u)=b$ ; therefore, $\psi$ is injective. For surjectivity, observe that 
\begin{align*}
\begin{bmatrix} b \\ a \end{bmatrix}/\Delta_{\alpha' \alpha'} = \begin{bmatrix} l \circ \pi(b) \\ m(a,b,l \circ \pi(b) ) \end{bmatrix}/\Delta_{\alpha' \alpha'} &= \begin{bmatrix} l \circ \pi(b) \\ m(a,b,l \circ \pi(b) ) \end{bmatrix}/\Delta_{\alpha' \alpha'} \\
&= \begin{bmatrix} l \circ \pi(b) \\ m \Big( m \big( m(a,b,l \circ \pi(b)),l \circ \pi(b),u \big), u, l \circ \pi(b) \Big) \end{bmatrix}/\Delta_{\alpha' \alpha'} \\
&= \psi \Big( m \big( m(a,b,l \circ \pi(b)),l \circ \pi(b),u \big) , l \circ \pi(b) \Big) \\
&= \psi \Big( m(a,b,u) , l \circ \pi(b) \Big) 
\end{align*}
where we have used Lemma~\ref{lem:usefulbit1} in the last line. It follows that $\psi^{-1} \left( \begin{bmatrix} b \\ a \end{bmatrix}/\Delta_{\alpha \alpha} \right) = \left\langle  m(a,b,u) , \pi(b) \right\rangle$.

According to Theorem~\ref{thm:10}, $A' \approx A'_{T'}(Q,A'^{\alpha',\tau},\ast')$ for affine datum and compatible 2-cocycle $T'$ which are given by
\begin{align*}
T'_{f}(x_{1},\ldots,x_{n}):= \begin{bmatrix}  l(f^{Q}(x_{1},\ldots,x_{n})) \\ f^{A'}(l(x_{1}),\ldots,l(x_{n}))  \end{bmatrix}/\Delta_{\alpha' \alpha'} & &(f \in \tau , n=\ar f)
\end{align*}
and action terms
\begin{align*}
a'(f,i)(x_{1},\ldots, x_{i-1},\begin{bmatrix} b \\ a \end{bmatrix}/\Delta_{\alpha' \alpha'} ,x_{i+1},\ldots, x_{n}):=  \begin{bmatrix} f^{A'}( l(x_{1}),\ldots, l(x_{i-1}),b,l(x_{i+1}),\ldots,l(x_{n}) )  \\ f^{A'}( l(x_{1}),\ldots,l(x_{i-1}),a,l(x_{i+1}),\ldots,l(x_{n}) )  \end{bmatrix}/\Delta_{\alpha' \alpha'}
\end{align*}
for $f \in \tau , n=\ar f$ where
\begin{align*}
f^{\Delta} \left( \begin{bmatrix} b \\ a \end{bmatrix}/\Delta_{\alpha' \alpha'},\delta \circ l(x_{2}),\ldots,\delta \circ l(x_{n}) \right) = a'(f,1) \left( \begin{bmatrix} b \\ a \end{bmatrix}/\Delta_{\alpha' \alpha'},x_{2},\ldots,x_{n} \right).
\end{align*}
The next step is to define action terms and a 2-cocycle $T$ appropriate to $(Q,I_{\alpha'})$ such that $\psi$ becomes an isomorphism $\psi : A'_{T'}(Q,A'^{\alpha',\tau},\ast') \rightarrow I_{\alpha'} \rtimes_{\ast,T} Q$; clearly, this is explicitly given by
\begin{align}\label{eq:408}
T_{f}(x_{1},\ldots,x_{n}):= m \big( f^{A'}(l(x_{1}),\ldots,l(x_{n})) , l(f^{Q}(x_{1},\ldots,x_{n})) , u \big) 
\end{align}
with action terms
\begin{align}\label{eq:409}
a(f,i)(\vec{x},\vec{a}) &= m \Big( f^{A'}(l(x_{1}),\ldots,m(a_{i},u,l(x_{i})),\ldots,l(x_{n})) ,f^{A'}(l(x_{1}),\ldots,l(x_{n})) , u \Big) 
\end{align}
so that 
\begin{align}
\psi \big( a(f,i)(\vec{x},\vec{a}),f^{Q}(\vec{x}) \big) = a'(f,i)(x_{1},\ldots,\psi(a_{i},x_{i}) ,\ldots,x_{n})
\end{align} 
and
\begin{align}
\psi \big( T_{f}(\vec{x}),f^{Q}(\vec{x}) \big) = T'_{f}(\vec{x}).
\end{align}
If we write $p_{i} = f^{A'}(l(x_{1}),\ldots,m(a_{i},u,l(x_{i})),\ldots,l(x_{n}))$ and $w = f^{A'}(l(x_{1}),\ldots,l(x_{n}))$, then using the definitions of the action and 2-cocycle terms and Lemma~\ref{lem:usefulbit1} we have
\begin{align*}
F_{f} \left( \psi(a_{1},x_{1}),\ldots, \psi(a_{n},x_{n}) \right) &= \sum_{i=1}^{n} a'(f,i)(x_{1},\ldots,\psi(a,x_{i}) ,\ldots,x_{n}) +_{w} T'_{f}(\vec{x}) \displaybreak[0]\\
&= \sum_{i=1}^{n} \begin{bmatrix} w \\ p_{i} \end{bmatrix}/\Delta_{\alpha' \alpha'} +_{w} \begin{bmatrix} l(f^{Q}(\vec{x})) \\ w \end{bmatrix}/\Delta_{\alpha' \alpha'}  \displaybreak[0]\\
&= \begin{bmatrix} w \\ p_{1} +_{w} \cdots +_{w} p_{n} \end{bmatrix}/\Delta_{\alpha' \alpha'} +_{w} \begin{bmatrix} l(f^{Q}(\vec{x})) \\ w \end{bmatrix}/\Delta_{\alpha' \alpha'} \displaybreak[0]\\
&= \begin{bmatrix} l(f^{Q}(\vec{x})) \\ p_{1} +_{w} \cdots +_{w} p_{n} \end{bmatrix}/\Delta_{\alpha' \alpha'} \displaybreak[0]\\
&= \psi \Big( m \big( p_{1} +_{w} \cdots +_{w} p_{n} , l(f^{Q}(\vec{x})), u \big) , f^{Q}(\vec{x}) \Big) \displaybreak[0]\\
&= \psi \Big( m \big( m( p_{1} +_{w} \cdots +_{w} p_{n}, w, u ), u , m(w, l(f^{Q}(\vec{x})), u) \big), f^{Q}(\vec{x}) \Big) \displaybreak[0]\\
&= \psi \Big( m( p_{1} +_{w} \cdots +_{w} p_{n}, w, u )  +_{u} T_{f}(\vec{x}) , f^{Q}(\vec{x}) \Big) \displaybreak[0]\\
&= \psi \Big( m \big( m ( p_{1} +_{w} \cdots +_{w} p_{n-1},w,u ) , u, m(p_{n},w,u) \big)  +_{u} T_{f}(\vec{x}) , f^{Q}(\vec{x}) \Big) \displaybreak[0]\\
&= \psi \Big(  m ( p_{1} +_{w} \cdots +_{w} p_{n-1},w,u ) +_{u} m(p_{n},w,u)  +_{u} T_{f}(\vec{x}) , f^{Q}(\vec{x}) \Big) \displaybreak[0]\\ 
&\vdots \displaybreak[0]\\
&= \psi \Big(  m(p_{1},w,u) +_{u} \cdots +_{u} m(p_{n},w,u)  +_{u} T_{f}(\vec{x}) , f^{Q}(\vec{x}) \Big) \displaybreak[0]\\ 
&= \psi \Big(  \sum_{i=1}^{n} a(f,i)(\vec{x},\vec{a}) +_{u} T_{f}(\vec{x}) , f^{Q}(\vec{x}) \Big) \displaybreak[0]\\ 
&= \psi \circ f^{I_{\alpha'} \rtimes_{\ast,T} Q} \left( \left\langle a_{1},x_{1} \right\rangle,\ldots,\left\langle a_{1},x_{1} \right\rangle \right);
\end{align*}
thus, $\psi$ is an isomorphism. 
\end{proof}

We say a difference term $m$ for variety $\mathcal V$ is \emph{weakly-associative} if it satisfies the identity $x=m(m(x,y,z),z,y)$. One way to enforce triviality $\epsilon(l)=0$ is for the variety to have a weakly-associative difference term.

\begin{corollary}\label{cor:AAA}
Let $\mathcal V$ be a variety with a weakly-associative difference term $m$. Let $A \in \mathcal V$ and $u \in A$ an idempotent element. If $\alpha \in \Con A$ is abelian, then there is a singleton action $\ast$ and 2-cocycle $T$ appropriate to $(A/\alpha , I_{\alpha})$ such that 
\begin{align}\label{eqn:amazing}
A \approx I_{\alpha} \rtimes_{\ast,T} A/\alpha.
\end{align}
\end{corollary}

The following generalizes the group example from lemma~\ref{lem:grp}.

\begin{proposition}
Let $A \in \mathcal V$ a variety with a weakly-associative difference term. Suppose $u \in A$ is an idempotent element and $\alpha \in \Con A$ is abelian. 
\begin{enumerate}
	
	\item If $\alpha \leq \beta$ and $\epsilon(l)=0$ for a section of $\beta$, then there is a singleton action $\ast$ such that $A(\alpha)/\Delta_{\alpha \beta} \approx I_{\alpha/[\alpha,\beta]} \rtimes_{\ast} A/\beta$.
	
	\item If $\alpha \in \Con A$ is central, then $A(\alpha)/\Delta_{\alpha 1} \approx I_{\alpha}$ and there is a 2-cocycle $T$ such that $A \approx I_{\alpha} \otimes^{T} A/\alpha$.
	
	\item If $\alpha \in \Con A$ is central and $\beta \in \Con A$, then $\kappa_{\alpha 1} \left( A(\alpha \wedge \beta) \right) \approx I_{\alpha \wedge \beta}$.

\end{enumerate}
\end{proposition}
\begin{proof}
(1) Set $Q= A/\beta$ and define $\psi: I_{\alpha} \times Q \rightarrow A(\alpha)/\Delta_{\alpha \beta}$ by $\psi(a,x) := \begin{bmatrix} l(x) \\ m(a,u,l(x)) \end{bmatrix}/\Delta_{\alpha \beta}$. If $((a,x),(b,y)) \in \ker \psi$, then we must have $(l(x),l(y)) \in \beta$ which yields $x=y$. Then by Lemma~\ref{lem:20}(1) we have $m(a,u,l(x)) \mathrel{[\alpha,\beta]} m(b,u,l(x))$ and so $a=m(m(a,u,l(x)),l(x),u) \mathrel{[\alpha,\beta]} m(m(b,u,l(x)),l(x),u)=b$ since $\epsilon(l)=0$; thus, $\ker \psi \leq (I_{\alpha} \cap [\alpha,\beta]) \times 0_{Q}$. For the reverse inclusion, take $a,b \in I_{\alpha}, x \in Q$ with $(a,b) \in [\alpha,\beta]$. By Lemma~\ref{lem:20}(2), there is $v \in A$ such that $\begin{bmatrix} v \\ a \end{bmatrix} \Delta_{\alpha \beta} \begin{bmatrix} v \\ b \end{bmatrix}$. Then
\begin{align*}
\begin{bmatrix} v  \\ a \end{bmatrix} \Delta_{\alpha \beta} \begin{bmatrix} v  \\ b \end{bmatrix} \ , \ \begin{bmatrix} v  \\ u \end{bmatrix} \Delta_{\alpha \beta} \begin{bmatrix} v  \\ u \end{bmatrix} \ , \ \begin{bmatrix} l(x)  \\ l(x) \end{bmatrix} \Delta_{\alpha \beta} \begin{bmatrix} l(x)  \\ l(x) \end{bmatrix} \quad \Rightarrow  \begin{bmatrix} l(x) \\ m(a,u,l(x)) \end{bmatrix} \Delta_{\alpha \beta} \begin{bmatrix} l(x) \\ m(b,u,l(x)) \end{bmatrix} ;
\end{align*}
altogether, $\ker \psi = (I_{\alpha} \cap [\alpha,\beta]) \times 0_{Q}$. This induces a bijection between $I_{\alpha/[\alpha,\beta]} \times Q$ and $A(\alpha)/\Delta_{\alpha \beta}$. We then follow the same definitions Eq~\eqref{eq:408} and Eq~\eqref{eq:409} in Theorem~\ref{thm:idemabelian} to define the action terms $\ast$ and 2-cocycle $T$ on $I_{\alpha/[\alpha,\beta]}$ so that $\psi$ is an isomorphism. Recall $A(\alpha)/\Delta_{\alpha \beta} \rightarrow A/\beta$ is a semidirect product with homomorphic section given by the diagonal embedding. Then by the definition in Eq~\eqref{eq:408} we see that each $T_{f}(\vec{x})=u/[\alpha,\beta]$ which is the zero in the induced abelian group operation $x +_{u/[\alpha,\beta]} y$ on $I_{\alpha/[\alpha,\beta]}$; thus, $I_{\alpha/[\alpha,\beta]} \rtimes_{\ast} A/\beta \approx A(\alpha)/\Delta_{\alpha \beta}$.

(2) Since $\alpha$ is central, $\epsilon(l) \leq [\alpha,1]=0$ for all choices of section. Then by part (1) we have $A(\alpha)/\Delta_{\alpha 1} \approx I_{\alpha/[\alpha,1]} \rtimes_{\ast} A/1 \approx I_{\alpha}$. The representation follows from Corollary~\ref{cor:1}. 

(3) Write $\Psi: A(\alpha)/\Delta_{\alpha 1} \rightarrow I_{\alpha}$ for the isomorphism in (2). Passing to the subalgebra $A(\alpha \wedge \beta) \leq A(\alpha)$ through the homomorphisms we have $\Psi \circ \kappa_{\alpha 1}(A(\alpha \wedge \beta)) = \Psi\left( \left\{ \begin{bmatrix} b \\ a \end{bmatrix}/\Delta_{\alpha 1} : (a,b) \in \alpha \wedge \beta \right\} \right) = \{m(a,b,u) : (a,b) \in \alpha \wedge \beta \} = I_{\alpha \wedge \beta}$.
\end{proof}

If $\mathcal V$ is a variety with a weakly-associative difference term $m$ and has idempotent elements, then by Corollary~\ref{cor:AAA} extensions with abelian kernels can be reconstructed as
\begin{align}\label{eqn:reprep}
A \approx I_{\alpha} \rtimes_{\ast, T} Q
\end{align}
from the pair of algebras $(Q,I_{\alpha})$, singleton actions $\ast$ and choices of 2-cocycle $T$. The condition $\mathcal V \vDash x=m(m(x,y,z),z,y)$ is quite restrictive but is satisfied by varieties of group expansions such as groups with multiple operators in Higgins \cite{higgins} which encompass varieties of $R$-modules expanded by multilinear operators investigated in Wires \cite{wiresVI}; for example, the representation in Eq~\eqref{eqn:reprep} holds for loops, skew-braces, Liebniz algebras, Rota-Baxter groups or dendriform algebras to name just a few well-researched examples. For such algebras, the isomorphism in Corollary~\ref{cor:AAA} provides a translation of the previously developed cohomology machinery for affine datum to machinery adapted to the representation in Eq~\eqref{eqn:amazing}. Since we will sometimes consider the results in \cite{wiresV,wiresVI,wiresVII} applied to such varieties it will be useful to have this machinery explicitly stated for future reference.

Two extensions $\pi: A \rightarrow Q$, $\pi' : A' \rightarrow Q$ with idempotents $u \in A,u' \in A'$ are \emph{equivalent} if there is an isomorphism $\phi: A \rightarrow A'$ such that $\pi = \pi' \circ \phi$ and restricts to an isomorphism $\phi: I_{\ker \pi} \rightarrow I_{\ker \pi'}$ of the congruence classes. Extensions in $\mathcal V$ with an abelian kernel and idempotent element are equivalent to extensions of the form
\begin{align}\label{eqn:semidirectidemdiff} 
p_{2} : I \rtimes_{\ast,T} Q \rightarrow Q
\end{align}
for \emph{datum} $(Q,I,\ast)$ which consists of a pair of algebras $(Q,I)$ in $\mathcal V$ with idempotent elements $(v,u) \in Q \times I$ in which $I$ is abelian, a homomorphic singleton action $Q \ast I$ and a 2-cocycle $T$ appropriate to $(Q,I)$; in addition, the action and 2-cocycle terms satisfy
\begin{enumerate}

	\item[(T1)] $a(f,i)(\vec{x},(a_{1},\ldots,a_{i-1},u,a_{i+1},\ldots,a_{n})) = u$, 
	
	\item[(T2)] $a(f,i)((v,\ldots,v),\vec{a}) = f^{I}(u,\ldots,a_{i},\ldots,u)$,
	
	\item[(T3)] $T_{f}(v,\ldots,v)=u$

\end{enumerate}
for $f \in \tau$. Semidirect products are equivalent to extensions in the form in Eq~\eqref{eqn:semidirectidemdiff} for $T$ trivial; that is, $T_{f} \equiv u$ for all $f \in \tau$. Another consequence of (T1)-(T3) is that the map $I \ni a \longmapsto \left\langle a, v \right\rangle \in I \rtimes_{\ast, T} Q$ is an embedding. Given an extension $\pi: A \rightarrow Q$ with $\alpha=\ker \pi$ and idempotent $u \in A$, the representation in Eq~\eqref{eqn:semidirectidemdiff} is determined by a choice of section $l: Q \rightarrow A$ with $l \circ \pi(u) = u$ which defines 
\begin{align*}
A \ni a \longmapsto \left\langle m(a,l \circ \pi(a),u) , \pi(a) \right\rangle \in I_{\alpha} \rtimes_{\ast, T} Q
\end{align*}
and the 2-cocycle and action terms by 
\begin{enumerate}

	\item $T_{f}(x_{1},\ldots,x_{n}):= m^{A} \big( f^{A}(l(x_{1}),\ldots,l(x_{n})), l(f^{Q}(x_{1},\ldots,x_{n})), u \big)$,

	\item $a(f,i)(\vec{x},\vec{a}) := m^{A} \big( f^{A'}(l(x_{1}),\ldots,m(a_{i},u,l(x_{i})),\ldots,l(x_{n})) ,f^{A'}(l(x_{1}),\ldots,l(x_{n})), u \big)$.
 
\end{enumerate}
To capture all such extensions in $\mathcal V$, we must restrict to those actions and 2-cocycles which ensure $I \rtimes_{\ast, T} Q \in \mathcal V$. Given a term $t$, there is term $t^{\ast}$ in the multisorted language $\ast = \{a(f,i): f \in \tau, i \in [\ar f] \} \cup \{ x +_{u} y \}$ and a term $t^{\partial, T}$ in the multisorted language $\ast \cup T \cup \{x +_{u} y \}$ such that 
\begin{align}
t^{ I \rtimes_{\ast,T} Q}\left( \left\langle a_{1},x_{1} \right\rangle,\ldots,\left\langle a_{n},x_{n} \right\rangle \right) = \left\langle t^{\ast}(\vec{a},\vec{x}) +_{u} t^{\partial,T}(\vec{x}) , t^{Q}(\vec{x}) \right\rangle
\end{align}
The action $Q \ast I$ is $\mathcal V$-\emph{compatible} if $t^{\ast} = s^{\ast}$ for all $t=s \in \mathrm{Id} \, \mathcal V$ and the 2-cocycle $T$ is $\mathcal V$-\emph{compatible} if $t^{\partial,T} = s^{\partial,T}$ for all $t=s \in \mathrm{Id} \, \mathcal V$. For $Q,I \in \mathcal V$ with $I$ abelian, the semidirect product $I \rtimes_{\ast} Q \in \mathcal V$ if and only if $Q \ast I$ is $\mathcal V$-compatible. For datum $(Q,I,\ast)$ with $Q,I \in \mathcal V$, $I$ abelian and $\mathcal V$-compatible action satisfying (T1), all extensions in $\mathcal V$ realizing the datum with an idempotent element are equivalent to extensions of the form Eq~\eqref{eqn:semidirectidemdiff} for $\mathcal V$-compatible 2-cocycles which satisfy (T2).

A 2-\emph{coboundary} for datum $(Q,I,\ast)$ is a sequence of functions $G=\{ G_{f}: f \in \tau \}$ \emph{witnessed} by $h : Q \rightarrow I$ such that
\begin{enumerate}

	\item $h(v)=u$,
	
	\item $G_{f}(\vec{x}) = \sum_{i=1} a(f,i)(\vec{x},h(\vec{x})) - h(f^{Q}(\vec{x}))$

\end{enumerate}
where the sum is the iterated addition of $x +_{u} y$ in $I$. \emph{Derivations} (or 1-\emph{cocycles}) are the witnesses for the trivial 2-coboundary; that is, $G_{f} \equiv u$ for all $f \in \tau$. Two extensions $A$ and $A'$ realizing fixed datum $(Q,I,\ast)$ are equivalent if and only if they are equivalent to extensions $A \approx I \rtimes_{\ast, T} Q$ and $A' \approx I \rtimes_{\ast, T'} Q$ such that $T-T'$ is a 2-coboundary. The $2^{\mathrm{nd}}$-cohomology is the abelian group of equivalence classes of $\mathcal V$-compatible 2-cocycles under the addition $x +_{u} y$ in the codomain. Restricting equivalence to a fixed extension $\pi:A \rightarrow Q$ realizing datum $(Q,I,\ast)$, \emph{stabilizing automorphisms} are the automorphisms $\phi: A \rightarrow A$ such that $\pi = \pi \circ \phi$ and $\phi|_{I} = \id$. The derivations of the datum are isomorphic to the group of stabilizing automorphisms of the semidirect product. If $\psi$ is the map defined in Theorem~\ref{thm:idemabelian} which witnesses the isomorphism $I_{\alpha} \rtimes_{\ast, T} \approx A_{T'}(Q,A^{\alpha,\tau},\ast')$, then the relations
\begin{itemize}
	
	\item $\psi( a(f,i)(\vec{x},\vec{a}) , f^{Q}(\vec{x})) = a'(f,i)(x_{1},\ldots,\psi(a_{i},x_{n}) , \ldots,x_{n})$,
	
	\item $\psi(T_{f}(\vec{x}),f^{Q}(\vec{x})) = T'_{f}(\vec{x})$,
	
	\item $\psi(h(x),x) = h'(x)$ for $h \in \mathrm{Der}(Q,I_{\alpha},\ast)$

\end{itemize}
induce the isomorphisms of cohomology
\begin{align*}
H^{2}_{\mathcal V}(Q,I_{\alpha},\ast) &\approx H^{2}_{\mathcal V}(Q,A^{\alpha,\tau},\ast') & \mathrm{Der}(Q,I_{\alpha},\ast) &\approx \mathrm{Der}(Q,A^{\alpha,\tau},\ast').
\end{align*}

\begin{example}
Let $\mathcal V$ be a variety of $R$-modules expanded be multilinear operations named by $F$. For an algebra $M \in \mathcal V$, it is easy to see that congruences $\alpha \in \Con M$ are in bijective correspondence with \emph{ideals} $I \triangleleft M$ which are submodules of $M$ that are \emph{absorbing} for each of the multilinear operations; that is, for each $f \in F$ with $n=\ar f$, $f(a_{1},\ldots,a_{n}) \in I$ whenever some $a_{i} \in I$. For an ideal $I \triangleleft M$, the corresponding congruence is given by $\alpha_{I} = \{(a,b) \in M^{2}: a-b \in I \} \in \Con M$. For a congruence $\alpha \in \Con M$, the corresponding ideal $I_{\alpha}$ is the $\alpha$-class which contains $0 \in M$. Then for any abelian ideal $I \triangleleft M$ we have the representation $M \approx I \rtimes_{\ast,T} M/I$ with 2-cocycle operations
\begin{align}
\begin{split}
T_{+}(x,y) &:= l(x) + l(y) - l(x+y) \\
T_{r}(x) &:= r \cdot l(x) - l(r \cdot x) \\
T_{f}(\vec{x}) &:= f(l(\vec{x})) - l(f^{Q}(\vec{x}))
\end{split}
\end{align}
for $f \in F, r \in R$ and singleton action terms
\begin{align}\label{eqn:actionterms4}
\begin{split}
a(f,i)(x_{1},\ldots,a,\ldots,x_{n}) &:= f((x_{1}),\ldots,a,\ldots,l(x_{n})) - f((x_{1}),\ldots,0,\ldots,l(x_{n})) \\
&= f((x_{1}),\ldots,a,\ldots,l(x_{n}))
\end{split}
\end{align}
for a choice of lifting $l: M/I \rightarrow M$ with $l(0)=0$. Note multilinearity of the operations in this case simplifies the expression for the action terms in Eq~\eqref{eqn:actionterms4}.
\end{example}


\begin{acknowledgments}
The research in this manuscript was supported by NSF China Grant \#12071374. 
\end{acknowledgments}


\end{document}